\documentclass[smallextended]{svjour3}
\usepackage{amsmath,amssymb}
\usepackage[mathscr]{euscript}
\usepackage[letterpaper]{geometry}
\usepackage{comment}
\usepackage{graphicx}
\usepackage{latexsym}
\usepackage{color}

\renewcommand{\leq}{\leqslant} 
\renewcommand{\geq}{\geqslant} 

\newcommand{\C}{\mathcal{C}}
\newcommand{\Po}{\mathcal{P}}
\newcommand{\K}{\mathcal{K}} 
\newcommand{\E}{\mathbb{E}^3}

\title{Wythoffian Skeletal Polyhedra in Ordinary Space, I}
\author{
Egon Schulte\thanks{Supported by NSA-grant H98230-14-1-0124. Email: schulte@neu.edu}\ and\ 
Abigail Williams\thanks{Email:\ abigail.williams13@gmail.com}\\[.03in]
Department of Mathematics\\[.03in]
Northeastern University, Boston, MA 02115, USA}

\begin{document}
\authorrunning{Schulte and Williams}\titlerunning{Wythoffian Skeletal Polyhedra in Ordinary Space, I}\maketitle

\begin{abstract}
Skeletal polyhedra are discrete structures made up of finite, flat or skew, or infinite, helical or zigzag, polygons as faces, with two faces on each edge and a circular vertex-figure at each vertex. When a variant of Wythoff's construction is applied to the forty-eight regular skeletal polyhedra (Gr\"unbaum-Dress polyhedra) in ordinary space, new highly symmetric skeletal polyhedra arise as ``truncations" of the original polyhedra. These Wythoffians are vertex-transitive and often feature vertex configurations with an attractive mix of different face shapes. The present paper describes the blueprint for the construction and treats the Wythoffians for distinguished classes of regular polyhedra. The Wythoffians for the remaining classes of regular polyhedra will be discussed in Part II, by the second author. We also examine when the construction produces uniform skeletal polyhedra. 
\end{abstract}




\section{Introduction}

Since ancient times, mathematicians and scientists have been studying polyhedra in ordinary Euclidean $3$-space $\E$. With the passage of time, various notions of polyhedra have attracted attention and have brought to light exciting new classes of highly symmetric structures including the well-known Platonic and Archimedean solids, the Kepler-Poinsot polyhedra, the Petrie-Coxeter polyhedra, and the more recently discovered Gr\"unbaum-Dress polyhedra (see \cite{Cox7,Cox,Dress1981,Dress1985,Grunpoly}). Over time we can observe a shift from the classical approach of viewing a polyhedron as a solid, to topological and algebraic approaches focussing on the underlying maps on surfaces (see Coxeter-Moser~\cite{CoxMos}), to graph-theoretical approaches highlighting the combinatorial incidence structures and featuring a polyhedron as a skeletal figure in space.

The skeletal approach to polyhedra in $\E$ was pioneered by Gr\"unbaum in~\cite{Grunpoly} and has had an enormous impact on the field. Skeletal polyhedra are discrete geometric structures made up of convex or non-convex, flat (planar) or skew, finite or infinite (helical or zigzag) polygons as faces, with a circular vertex-figure at each vertex, such that every edge lies in exactly two faces. There has been a lot of recent activity in this area: the skeletal regular polyhedra were enumerated by Gr\"unbaum~\cite{Grunpoly} and Dress~\cite{Dress1981,Dress1985} (for a simpler approach to the classification see McMullen \& Schulte~\cite{McSch1997,SchMc}); the skeletal chiral polyhedra were classified in \cite{SchChiral1,SchChiral2} (see also Pellicer \& Weiss~\cite{PellicerWeiss}); the regular polygonal complexes, a more general class of discrete skeletal structures than polyhedra, were classified in Pellicer \& Schulte~\cite{Pellicer,Pellicer2}; and corresponding enumerations for certain classes of regular polyhedra, polytopes, or apeirotopes (infinite polytopes) in higher-dimensional Euclidean spaces were achieved by McMullen~\cite{PMc,PMcFullRank,PMcReal3} (see also Arocha, Bracho \& Montejano~\cite{Aro} and Bracho~\cite{Bracho}). All these skeletal structures are relevant to the study of crystal nets in crystal chemistry (see~\cite{Delgado,OKee,OKeefe,Sch2,Wells}).

The present paper and its successor~\cite{Williams2} by the second author are inspired by the quest for a deeper understanding of the uniform skeletal polyhedra in $\E$, that is, the skeletal analogues of the Archimedean solids (see also \cite{Williams}). There is a large body of literature on the traditional uniform polyhedra and higher-dimensional polytopes (see~\cite{Cox5,Coxrs2,Coxrs3,Johnson3,Mart}). Recall that a convex polyhedron in $\E$ is said to be {\em uniform\/} if its faces are convex regular polygons and its symmetry group is transitive on the vertices. The uniform convex polyhedra are precisely the Archimedean solids and the prisms and antiprisms. The classification for the finite, convex or non-convex, uniform polyhedra with planar faces was essentially obtained in a classical paper by Coxeter, Longuet-Higgins and Miller~\cite{Cox3}, but the completeness of the enumeration was only proved years later, independently, by Skilling~\cite{Skilling2,Skilling} and Har'El~\cite{Zvi}. The classification of arbitrary uniform skeletal polyhedra is a challenging open problem. Even the finite polyhedra with skew faces have not been classified.

The Wythoffians of the regular skeletal polyhedra studied in this paper represent a tractable class of skeletal polyhedra that contains a wealth of new examples of uniform polyhedra with non-planar faces. In fact, our study actually goes a long way in classifying all the uniform skeletal polyhedra in $\E$. The name ``Wythoffian" is derived from Wythoff's construction (see~\cite{Cox,SchMc}). Our approach takes a geometrically regular polyhedron $P$ in $\E$ as input and then produces from it up to seven different kinds of geometric Wythoffians by an analogue of Wythoff's construction. The procedure applies to all forty-eight geometrically regular polyhedra in $\E$ and often produces amazing figures as output. Our goal is to analyze these Wythoffians.

The paper is organized as follows. In Section~\ref{geomab} we begin by reviewing the basic concept of a regular polyhedron, both geometric and abstract, and discussing realizations as a means to connect the abstract theory with the geometric theory. In Section~\ref{abwyth} we introduce the seven Wythoffians at the abstract level and then in Section~\ref{geowyth} provide the blueprint for the realization as geometric Wythoffians in $\E$. Finally, in Section~\ref{wythvarious} we describe the geometric Wythoffians of various distinguished classes of regular polyhedra. The subsequent paper~\cite{Williams2} treats the geometric Wythoffians for the remaining classes of regular polyhedra.

\section{Geometric and abstract polyhedra}
\label{geomab}

We begin by defining a geometric polyhedron as a discrete structure in Euclidean 3-space $\mathbb{E}^3$ rather than as a realization of an abstract polyhedron. 

Given a geometric figure in $\E$, its (geometric) {\em symmetry group\/} consists of all isometries of its affine hull that map the figure to itself. When a figure is linear or planar we sometimes view this group as a subgroup of the isometry group of $\E$, with the understanding that the elements of the group have been extended trivially from the affine hull of the figure to the entire space $\E$. 

\subsection{Geometric polyhedra}

Informally, a geometric polyhedron will consist of a family of vertices, edges, and finite or infinite polygons, all fitting together in a way characteristic for traditional convex polyhedra (see~\cite{Grunpoly} and \cite[Ch. 7E]{SchMc}). For two distinct points $u$ and $u'$ of $\E$ we let $(u,u')$ denote the closed line segment with ends $u$ and $u'$. 

A {\em finite polygon\/}, or simply an {\em $n$-gon\/}, $(v_1, v_2, \dots, v_n)$ in $\mathbb{E}^3$ is a figure formed by distinct points $v_1, \ldots, v_n$, together with the line segments $(v_i, v_{i+1})$, for $i = 1, \dots, n-1$, and $(v_n, v_1)$. Similarly, an {\em infinite polygon\/} consists of an infinite sequence of distinct points $(\ldots, v_{-2},v_{-1}, v_0, v_1, v_2,\ldots)$ and of the line segments $(v_i, v_{i+1})$ for each $i$, such that each compact subset of $\mathbb{E}^3$ meets only finitely many line segments. In either case the points are the {\em vertices\/} and the line segments the {\em edges\/} of the polygon. 

 A polygon is {\em geometrically regular if its geometric symmetry group is a (finite or infinite dihedral) group acting transitively on the {\em flags\/}, that is, the $2$-element sets consisting of a vertex and an incident edge. }

\begin{definition}
\label{poldef}
A {\em geometric polyhedron}, or simply {\em polyhedron} (if the context is clear), $P$ in $\mathbb{E}^3$ consists of a set of points, called {\em vertices}, a set of line segments, called {\em edges}, and a set of polygons, called {\em faces}, such that the following properties are satisfied.
\begin{itemize}\label{defpol}
\item[(a)] The graph defined by the vertices and edges of $P$, called the {\em edge graph\/} of $P$, is connected.
\item[(b)] The vertex-figure of $P$ at each vertex of $P$ is connected. By the {\em vertex-figure\/} of $P$ at a vertex $v$ we mean the graph whose vertices are the neighbors of $v$ in the edge graph of $P$ and whose edges are the line segments $(u,w)$, where $(u, v)$ and $(v, w)$ are adjacent edges of a common face of $P$.
\item[(c)] Each edge of $P$ is contained in exactly two faces of $P$.
\item[(d)] $P$ is {\em discrete\/}, meaning that each compact subset of $\mathbb{E}^3$ meets only finitely many faces of $P$.
\end{itemize}
\end{definition}

Note that the discreteness assumption in Definition~\ref{poldef}(d) implies that the vertex-figure at every vertex of a polyhedron $P$ is a finite polygon. Thus vertices have finite valency in the edge graph of $P$.  The edge graph is often called the {\em $1$-skeleton} of $P$.

A {\em flag\/} of a geometric polyhedron $P$ is a $3$-element set containing a vertex, an edge, and a face of $P$, all mutually incident.  Two flags of $P$ are called {\em adjacent\/} if they differ in precisely one element.  An {\it apeirohedron\/} is an infinite geometric polyhedron.

A geometric polyhedron $P$ in $\E$ is said to be ({\em geometrically\/}) {\em regular} if its symmetry group $G(P)$ is transitive on the flags of $P$. The symmetry group $G(P)$ of a regular polyhedron $P$ is transitive, separately, on the vertices, edges, and faces of $P$. In particular, the faces are necessarily regular polygons, either finite, planar (convex or star-) polygons or non-planar, {\em skew\/}, polygons, or infinite, planar zigzags or helical polygons (see \cite[Ch. 1]{Cox4} or \cite{Grunpoly}). Linear apeirogons do not occur as faces of regular polyhedra. 

We also briefly touch on chiral polyhedra. These are nearly regular polyhedra. A geometric polyhedron $P$ is called ({\em geometrically\/}) {\em chiral\/} if its symmetry group has two orbits on the flags of $P$, such that adjacent flags are in distinct orbits. 

The geometric polyhedra in $\E$ which are regular or chiral all have a vertex-transitive symmetry group and regular polygons as faces. They are particular instances of uniform polyhedra. A geometric polyhedron $P$ is said to be ({\em geometrically\/}) {\em uniform\/} if $P$ has a vertex-transitive symmetry group and regular polygons as faces. The uniform polyhedra with planar faces have attracted a lot attention in the literature. Our methods will provide many new examples of uniform skeletal polyhedra with nonplanar faces.

At times we encounter geometric figures which are not polyhedra but share some of their properties. Examples are the polygonal complexes described in \cite{Pellicer,Pellicer2}. Roughly speaking, a {\em polygonal complex\/} $K$ in $\E$ is a structure with the defining properties (a), (b) and (d) of Definition~\ref{defpol} for polyhedra, but with property (c) replaced by the more general property, (c') say, requiring that each edge of $K$ be contained in exactly $r$ faces of $K$, for a fixed number $r \geq 2$. The polygonal complexes with $r=2$ are just the geometric polyhedra. The vertex-figures of polygonal complexes need not be simple polygons as for polyhedra; they even can be graphs with double edges (edges of multiplicity 2). A polygonal complex is {\em regular\/} if its geometric symmetry group is transitive on the flags. 

The Wythoffians we construct from geometrically regular polyhedra $P$ in $\E$ will usually be generated from the orbit of a single point inside the fundamental region of the symmetry group of $P$. Given a discrete group $G$ of isometries of an $n$-dimensional Euclidean space $\mathbb{E}^n$, an open subset $D$ of $\mathbb{E}^n$ is called a \emph{fundamental region\/} for $G$ if $r(D)\cap D=\varnothing$ for $r\in G\setminus\{1\}$ and $\mathbb{E}^n=\bigcup_{r\in G} r(\text{cl}(D))$, where $cl(D)$ denotes the closure of $D$ (see~\cite{Benson}). Note that our notion of fundamental region is not quite consistent with the notion of a fundamental simplex used in the theory of Coxeter groups or related groups (see \cite[Ch. 3]{SchMc}), where a fundamental simplex by definition is a closed simplex (its interior is a fundamental region according to our definition).

Some of the groups we encounter have complicated fundamental regions. The following procedure produces a possible fundamental region for any given discrete group $G$ of isometries of~$\mathbb{E}^n$. Let $u\in\mathbb{E}^n$ be a point that is not held invariant under any non-identity transformation in $G$. For $r\in G$ define $H[r(u)]$ as the open half space containing $u$ bounded by the hyperplane which perpendicularly bisects the line segment between $u$ and $r(u)$. Then $D:=\bigcap_{r\in G} H[r(u)]$ is a fundamental region of $G$ in $\mathbb{E}^n$. In other words, $D$ is the open Dirichlet-Voronoi region, centered at $u$, of the orbit of $u$ under $G$ in $\mathbb{E}^n$ (see \cite{Conway}).

\subsection{Abstract polyhedra}

While our focus is on geometric polyhedra it is often useful to view a geometric polyhedron as a realization of an abstract polyhedron in Euclidean space. We begin with a brief review of the underlying abstract theory (see \cite[Ch.~2]{SchMc}). 

An {\em abstract polyhedron\/}, or {\em abstract $3$-polytope\/}, is a partially ordered set $\mathcal{P}$ with a strictly monotone {\em rank\/} function with range $\{-1,0,1,2,3\}$. The elements of rank $j$ are the {\em $j$-faces\/} of~$\mathcal{P}$. For $j = 0$, $1$ or $2$, we also call $j$-faces {\em vertices}, {\em edges\/} and {\em facets\/}, respectively. When there is little chance of confusion, we use standard terminology for polyhedra and reserve the term ``face'' for ``$2$-face'' (facet). There is a minimum face $F_{-1}$ (of rank $-1$) and a maximum face~$F_3$ (of rank $3$) in $\Po$; this condition is included for convenience and is often omitted as for geometric polyhedra. The {\em flags\/} (maximal totally ordered subsets) of $\mathcal{P}$ each contain, besides $F_{-1}$ and $F_3$, exactly one vertex, one edge and one facet. In practice, when listing the elements of a flag we often suppress $F_{-1}$ and $F_3$. Further, $\mathcal{P}$ is {\em strongly flag-connected\/}, meaning that any two flags $\Phi$ and $\Psi$ of $\mathcal{P}$ can be joined by a sequence of flags $\Phi = \Phi_{0},\Phi_{1},\ldots,\Phi_{k} =\Psi$, where $\Phi_{i-1}$ and $\Phi_{i}$ are {\em adjacent\/} (differ by one face), and $\Phi \cap\Psi \subseteq \Phi_{i}$ for each $i$. Finally, if $F$ and $G$ are a $(j-1)$-face and a $(j+1)$-face with $F < G$ and $0 \leq j \leq 2$, then there are exactly {\em two\/} $j$-faces $H$ such that $F < H < G$. As a consequence, for $0 \leq j \leq 2$, every flag $\Phi$ of $\mathcal{P}$ is adjacent to just one flag, denoted $\Phi^j$, differing in the $j$-face; the flags $\Phi$ and $\Phi^j$ are said to be {\em $j$-adjacent\/} to each other. 

When $F$ and $G$ are two faces of an abstract polyhedron $\mathcal{P}$ with $F \leq G$, we call
$G/F := \{H \mid F \leq H \leq G\}$ a {\em section\/} of $\mathcal{P}$. We usually identify a face $F$ with the section $F/F_{-1}$. The section $F_{3}/F$ is the {\em co-face\/} of $\mathcal{P}$ at $F$, or the {\em vertex-figure\/} at $F$ if $F$ is a vertex.

If all facets of an abstract polyhedron $\Po$ are $p$-gons for some $p$, and all vertex-figures are $q$-gons for some $q$, then $\Po$ is said to be of ({\em Schl\"{a}fli\/}) {\em type\/} $\{p,q\}$; here $p$ and $q$ are permitted to be infinite. We call an abstract polyhedron {\em locally finite\/} if all its facets and all its vertex-figures are finite  polygons.

An {\em automorphism\/} of an abstract polyhedron $\Po$ is an incidence preserving bijection of $\Po$ (that is, if $\varphi$ is the bijection, then $F\leq G$ in $\Po$ if and only if $\varphi(F)\leq\varphi(G)$ in $\Po$.) By $\Gamma(\Po)$ we denote the (combinatorial) {\em automorphism group\/} of $\Po$.

We call an abstract polyhedron $\mathcal{P}$ {\em regular\/} if $\Gamma(\mathcal{P})$ is transitive on the flags of $\mathcal{P}$. Let $\Phi := \{F_{0},F_{1},F_{2}\}$ be a {\em base\/} flag of $\mathcal{P}$. The automorphism group $\Gamma(\mathcal{P})$ of a regular polyhedron $\mathcal{P}$ is generated by {\em distinguished generators\/} $\rho_{0},\rho_{1},\rho_{2}$ ({\em with respect to\/}~$\Phi$), where $\rho_{j}$ is the unique automorphism which fixes all faces of $\Phi$ but the $j$-face. These generators satisfy the standard Coxeter-type relations
\begin{equation} 
\label{autgreg}
\rho_{0}^2 = \rho_{1}^2 = \rho_{2}^2 = 
(\rho_{0}\rho_{1})^{p} = (\rho_{1}\rho_{2})^{q} =
(\rho_{0}\rho_{2})^{2} = 1
\end{equation}
determined by the {\em type\/} $\{p,q\}$ of $\mathcal{P}$ (when $p=\infty$ or $q=\infty$ the corresponding relation is superfluous and hence is omitted); in general there are also other independent relations. Note that, in a natural way, the automorphism group of the facet of $\mathcal{P}$ is $\langle\rho_{0},\rho_{1}\rangle$, while that of the vertex-figure is $\langle\rho_{1},\rho_{2}\rangle$. 

An abstract polyhedron $\mathcal{P}$ is said to be {\em chiral\/} if $\Gamma(\mathcal{P})$ has two orbits on the flags, such that adjacent flags are in distinct orbits. Note that the underlying abstract polyhedron of a geometrically chiral (geometric) polyhedron must be (combinatorially) chiral or (combinatorially) regular.

In analogy with the geometric case we could define an abstract polyhedron $\Po$ to be (combinatorially) ``uniform'' if $\Po$ has regular facets and $\Gamma(\Po)$ acts transitively on the vertices of $\Po$. However, the facets of any abstract polyhedron trivially are combinatorially regular, so being uniform just reduces to being vertex-transitive under the automorphism group. 

The {\em Petrie dual\/} of a (geometric or abstract) regular polyhedron $P$ has the same vertices and edges as $P$; its facets are the {\em Petrie polygons\/} of $\mathcal{P}$, which are paths along the edges of $\Po$ such that any two successive edges, but not three, belong to a facet of $\mathcal{P}$. 

\subsection{Realizations}

The abstract theory is connected to the geometric theory through the concept of a realization. Let $\mathcal{P}$ be an abstract polyhedron, and let $\mathcal{F}_j$ denote its set of $j$-faces for $j=0,1,2$. Following \cite[Sect.~5A]{SchMc}, a {\em realization\/} of $\mathcal{P}$ is a mapping
$\beta\colon\mathcal{F}_0\rightarrow E$ of the vertex-set $\mathcal{F}_0$ into some Euclidean space $E$. Then define $\beta_0 := \beta$ and $V_{0} := \beta(\mathcal{P}_0)$, and write $2^X$ for the family of subsets of a set $X$. The realization $\beta$ recursively induces two surjections:\ a surjection $\beta_{1}\colon \mathcal{F}_{1} \rightarrow V_{1}$, with $V_{1}\subset 2^{V_{0}}$ consisting of the elements
\[ \beta_{1}(F) := \{\beta_{0}(G) \mid G\in \mathcal{F}_{0} \mbox{ and } G\leq F\} \]
for $F\in\mathcal{F}_1$; and a surjection $\beta_{2}\colon \mathcal{F}_{2} \rightarrow V_{2}$, with $V_{2}\subset 2^{V_{1}}$ consisting of the elements
\[ \beta_{2}(F) := \{\beta_{1}(G) \mid G\in \mathcal{F}_{1} \mbox{ and } G\leq F\} \]
for $F\in\mathcal{F}_2$. Even though each $\beta_j$ is determined by $\beta$, it is helpful to think of the realization as given by all the $\beta_j$. A realization $\beta$ is said to be {\em faithful\/} if each $\beta_j$ is a bijection; otherwise, $\beta$ is {\em degenerate\/}. Note that not every abstract polyhedron admits a realization in a Euclidean space.  (In different but related contexts, a realization is sometimes called a representation~\cite{GRun5,PiZi}.)

In our applications, $E=\E$ and all realizations are faithful (and discrete). In this case, the vertices, edges and facets of $\Po$ are in one-to-one correspondence with certain points, line segments, and (finite or infinite) polygons in $\E$, and it is safe to identify a face of $\Po$ with its image in $\E$. The resulting family of points, line segments, and polygons then is a geometric polyhedron in $\E$ and is denoted by $P$; it is understood that $P$ inherits the partial ordering of $\Po$. We frequently identify $\Po$ and $P$. Note that the symmetry group of a faithful realization is a subgroup of the automorphism group.

Conversely, all geometric polyhedra as defined above arise as realizations of abstract polyhedra. In particular, the geometrically regular polyhedra in $\E$ are precisely the $3$-dimensional realizations of abstract regular polyhedra which are discrete and faithful and have a flag-transitive symmetry group. These polyhedra have been extensively studied (see~\cite[Sect.~7E]{SchMc}). We briefly review them in Section~\ref{geowyth} as they form the basis of our construction.

For geometrically regular polyhedra $P$ in $\E$ we prefer to denote the distinguished generators of $G(P)$ by $r_0,r_1,r_2$. Thus, if $\Phi = \{F_{0},F_{1},F_{2}\}$ is again a {\em base\/} flag of $P$, and $r_j$ the involutory symmetry of $P$ fixing all faces of $\Phi$ but the $j$-face, then $G(P)=\langle r_0,r_1,r_2\rangle$ and the Coxeter-type relations 
\begin{equation} 
\label{greg}
r_{0}^2 = r_{1}^2 = r_{2}^2 = 
(r_{0}r_{1})^{p} = (r_{1}r_{2})^{q} =
(r_{0}r_{2})^{2} = 1
\end{equation}
hold, where again $\{p,q\}$ is the type of $P$. Here $q$ must be finite since $P$ is discrete; however, $p$ still can be infinite. When $P$ is geometrically regular the groups $\Gamma(P)$ and $G(P)$ are isomorphic; in particular, the mapping $\rho_j\mapsto r_j\;(j=0,1,2)$ extends to an isomorphism between the groups.

Two realizations of an abstract regular polyhedron $\Po$ can be combined to give a new realization of $\Po$ in a higher-dimensional space. Suppose we have two (not necessarily faithful) realizations of $\Po$ in two Euclidean spaces, say $P$ with generators $r_{0},r_{1},r_{2}$ in $E$ and $P'$ with generators $r_{0}',r_{1}',r_{2}'$ in $E'$ (possibly some $r_{j}=1$ or $r_{j}'=1$ if $P$ or $P'$ is not faithful). Then their {\em blend}, denoted $P \# P'$, is a realization of $\Po$  in $E\times E'$ obtained by Wythoff's construction as an orbit structure as follows (see \cite{Cox} and \cite[Ch. 5A]{SchMc}). Write $R_j$ and $R_{j}'$ for the {\em mirror\/} (fixed point set) of a distinguished generator $r_j$ in $E$ or $r_{j}'$ in $E'$, respectively. The cartesian products $R_{0}\times R_{0}'$, $R_{1}\times R_{1}'$ and $R_{2}\times R_{2}'$, respectively, then are the mirrors for involutory isometries $s_{0}$, $s_{1}$ and $s_2$ of $E\times E'$ which generate the symmetry group $G(P \# P')$ of the blend. Indeed, if $v\in R_{1}\cap R_{2}$ and $v'\in R_{1}'\cap R_{2}'$ are the base (initial) vertices of the two realizations, then the point $w:=(v,v')$ in $E\times E'$ can be chosen as the base (initial) vertex for the blend $P \# P'$. Then the base edge and base face are determined by the orbits of $w$ under the subgroups $\langle s_0\rangle$ and $\langle s_0,s_1\rangle$, respectively. Finally, the vertices, edges, and faces of the entire polyhedron $P \# P'$ are the images of the base vertex, base edge, or base face under the entire group $\langle s_0,s_1,s_2\rangle$. A realization which cannot be expressed as a blend in a non-trivial way is called \emph{pure}.

\section{Wythoffians of abstract polyhedra}
\label{abwyth}

Every abstract polyhedron $\mathcal{P}$ naturally gives rise to generally seven new abstract polyhedra, the abstract Wythoffians of $\mathcal{P}$. These Wythoffians have appeared in many applications, usually under different names (see \cite{Conway,PiRa}); they are often called {\em truncations\/} of the respective polyhedron or map (see \cite{Cox}). 
(In the literature, the word ``Wythoffian'' is mostly used as an adjective, not a noun, to describe a figure obtained by Wythoff's construction. The use of ``Wythoffian" in~\cite{deza} is similar to ours.)

\subsection{Wythoffians from the order complex}

It is convenient to construct the Wythoffians from the order complex of $\Po$. The {\em order complex\/} $\C := \C(\Po)$ of an abstract polyhedron $\Po$ is the $2$-dimensional abstract simplicial complex, whose vertices are the proper faces of $\Po$, and whose simplices are the chains (totally ordered subsets) of $\Po$ which only contain proper faces of $\Po$ (see \cite[Ch. 2C]{SchMc}). The maximal simplices in $\C$ are in one-to-one correspondence with the flags of $\Po$, and are $2$-dimensional. The {\em type\/} of a vertex of $\C$ is its rank as a face of $\Po$. More generally, the {\em type\/} of a simplex $\Omega$ of $\C$ is the set of types of the vertices of $\Omega$. Thus every $2$-simplex has type $\{0,1,2\}$. Two $2$-simplices of $\C$ are {\em $j$-adjacent\/} if and only if they differ in their vertices of type $j$. With this type function on chains, the order complex acquires the structure of a {\em labelled\/} simplicial complex.

The defining properties of $\Po$ translate into strong topological properties of $\C$. In particular, each
$2$-simplex of $\C$ is $j$-adjacent to exactly one other $2$-simplex, for $j=0,1,2$. When rephrased for $\C$, the strong flag-connectedness of $\Po$ says that, for any two $2$-simplices $\Phi$ and $\Psi$ of $\C$ which intersect in a face $\Omega$ (a simplex or the empty set) of $\C$, there exists a sequence 
$\Phi = \Phi_{0},\Phi_{1},\ldots,\Phi_{k-1},\Phi_{k} = \Psi$ of $2$-simplices of $\C$, all containing $\Omega$, such that $\Phi_{i-1}$ and $\Phi_{i}$ are adjacent for $i = 1,\ldots,k$. 

Recall that the {\em star\/} of a face $\Omega$ in a simplicial complex is the subcomplex consisting of all the simplices which contain $\Omega$, and all their faces. The {\em link\/} of $\Omega$ is the subcomplex consisting of all the simplices in the star of $\Omega$ which do not intersect $\Omega$. For an abstract polyhedron~$\Po$, the structure of the link of a vertex in its order complex $\C$ depends on the number of $2$-simplices it is contained in. Every vertex of $\C$ of type $1$ has a link isomorphic to a $4$-cycle. If a vertex $F$ of $\C$ is of type $2$, and the 2-face $F/F_{-1}$ of $\Po$ is a $p$-gon, then the link of $F$ in $C$ is a $2p$-cycle if $p$ is finite, or an infinite path (an infinite $1$-dimensional simplicial complex in which every vertex lies in exactly two $1$-simplices) if $p$ is infinite. Similarly, if $F$ is a vertex of $\C$ of type $0$, and the vertex-figure $F_{3}/F$ of $\Po$ at $F$ is a $q$-gon, then the link of $F$ in $\C$ is a $2q$-cycle if $q$ is finite, or an infinite path if $q$ is infinite.

If $\Po$ is a locally finite abstract polyhedron with order complex $\C$, then $\Po$ can be viewed as a face-to-face tessellation on a (compact or non-compact) closed surface $S$ by topological polygons, and $\C$ as a triangulation of $S$ refining $\Po$ in the manner of a ``barycentric subdivision''. If $\Po$ has faces or vertex-figures which are apeirogons, then the link of the corresponding vertices in $\C$ is not a $1$-sphere and so $\Po$ is not supported by a closed surface; however, $\C$ still has the structure of a $2$-dimensional pseudo-manifold, which we again denote  by $S$. (A $2$-dimensional pseudo-manifold is a topological space $X$ with a $2$-dimensional triangulation $\K$ such that the following three conditions hold:\ first, $X$ is the union of all triangles of $\K$; second, every edge of $\K$ lies in exactly two triangles of $\K$; and third, any two triangles of $\K$ can be joined by a finite sequence of triangles of $\K$ such that successive triangles in the sequence intersect in an edge~\cite{Seif}.) In our applications, the vertex-figures of $\Po$ are always finite polygons, whereas the faces are often apeirogons.

If $\Po$ is regular and $\rho_0,\rho_1,\rho_2$ are the generators of $\Gamma(\Po)$ associated with a base flag $\Phi$ of $\Po$, then $\Gamma(\Po)$ acts on $S$ as a group of homeomorphisms of $S$ and the distinguished generators appear as ``combinatorial reflections'' in the sides of the $2$-simplex $\Phi$ of $\C$. The 2-simplex $\Phi$ is a {\em fundamental triangle\/} for the action of $\Gamma(\Po)$ on $S$, meaning that the orbit of every point of $S$ under $\Gamma(\Po)$ meets $\Phi$ in exactly one point.  (Recall our previous remark about the notion of fundamental simplex, or in this case, fundamental triangle.) In fact, every $2$-simplex of $\C$ is a fundamental triangle for $\Gamma(\Po)$ on $S$, with a conjugate set of distinguished generators occurring as ``combinatorial reflections'' in its sides. For a regular polyhedron, the order complex can be completely described in terms of $\Gamma(\Po)$ since this is already true for $\Po$ itself (see \cite[Sect. 2C]{SchMc}).

For example, consider the regular tessellation $\Po = \{4,4\}$ of the Euclidean plane by squares, four coming together at a vertex. Here $\C$ appears as the actual barycentric subdivision of the tessellation, and any triangle in $\C$ can serve as the fundamental region for the symmetry group, which in this case is a Euclidean plane reflection group. 

An alternative approach to abstract Wythoffians using dissections of fundamental triangles is described in Pisanski \& Zitnik~\cite{PiZi} (see also~\cite{delrio}).

\subsection{Wythoffians of abstract regular polyhedra}

The Wythoffians of an abstract regular polyhedron $\Po$ are derived as orbit structures from the order complex $\C$, or equivalently, from the underlying surface or pseudo-manifold $S$. The construction can be carried out at a purely combinatorial level for arbitrary abstract polyhedra without reference to automorphism groups (by exploiting the action of the monodromy groups~\cite{Mon}). However, as we are mainly interested in geometric Wythoffians derived from geometrically regular polyhedra, we will concentrate on regular polyhedra and exploit their groups. The method employed is known as Wythoff's construction (see~\cite{Cox,SchMc}). 

Now let $\Po$ be an abstract regular polyhedron with order complex $\C$ and surface or pseudo-manifold $S$. Suppose the base flag $\Phi=\{F_0,F_1,F_2\}$ of $\Po$ is realized as a fundamental triangle on $S$ with vertices $F_0,F_1,F_2$. This fundamental triangle $\Phi$ naturally partitions into seven subsets:\ its vertices, the relative interiors of its edges, and its relative interior. The closure of each of these subsets is a simplex in $\C$ and hence has a type $I$. Thus each subset is naturally associated with a nonempty subset $I$ of $\{0,1,2\}$ specifying the generators which move the points of $\Phi$ in the subset; more precisely, the subset belonging to $I$ consists of the points of $\Phi$ which are transient under the generators $\rho_i$ with $i\in I$ but invariant under the generators $\rho_i$ with $i\notin I$. More explicitly, for the vertex $F_i$ of $\Phi$ the type $I$ is $\{i\}$; for the relative interior of the edge joining $F_i$ and $F_j$ it is $\{i,j\}$; and for the relative interior it is~$\{0,1,2\}$. 

\begin{table}
\begin{center}
\begin{tabular}{|c|c|c|c|c|}
\hline
\raisebox{-.5\totalheight}[\totalheight][1\totalheight]{\large{\textbf{\textit{I}}}} & \raisebox{-.5\totalheight}[\totalheight][1\totalheight]{\large{\textbf{Initial vertex}}}& \raisebox{-.5\totalheight}[\totalheight][1\totalheight]{\large{\textbf{Ringed diagram}}} & \raisebox{-.5\totalheight}[\totalheight][1\totalheight]{\large{\textbf{Wythoffian}}} & \raisebox{-.5\totalheight}[\totalheight][1\totalheight]{\large{  \textbf{Vertex Symbol}  }}\\
\hline
\raisebox{-1.8\height}{$\{0\}$} & \raisebox{-.78\totalheight}{ \includegraphics*[clip, scale = .45]{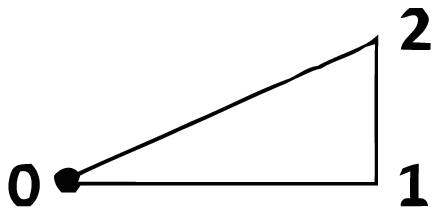}}& \raisebox{-1.15\totalheight}{ \includegraphics*[clip, scale = .25]{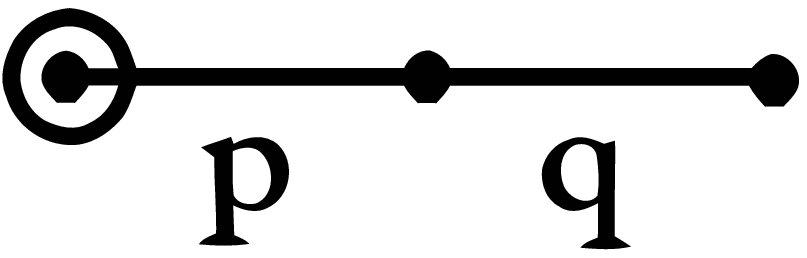}}& \raisebox{-1.8\height}{$P^{0}$}& \raisebox{-1.8\height}{  $(p^q)$  }\\
\hline
\raisebox{-1.8\height}{$\{1\}$} & \raisebox{-.78\totalheight}{ \includegraphics*[clip, scale = .45]{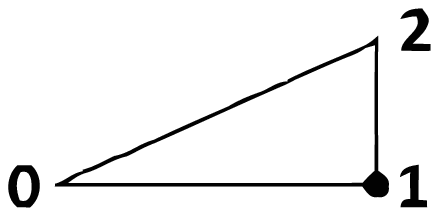}}&\raisebox{-1.15\totalheight}{ \includegraphics*[clip, scale = .25]{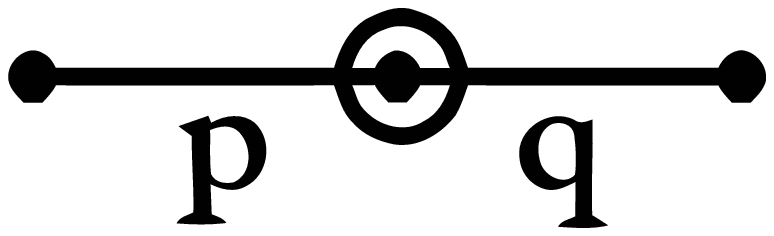}}&\raisebox{-1.8\height}{$P^{1}$}& \raisebox{-1.8\height}{ $(p.q.p.q)$ }\\
\hline
\raisebox{-1.8\height}{$\{2\}$} & \raisebox{-.8\totalheight}{ \includegraphics*[clip, scale = .45]{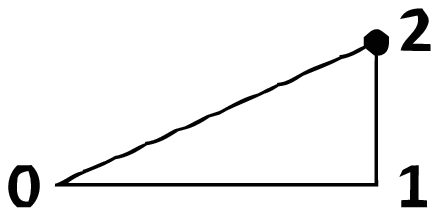}}&\raisebox{-1.15\totalheight}{ \includegraphics*[clip, scale = .25]{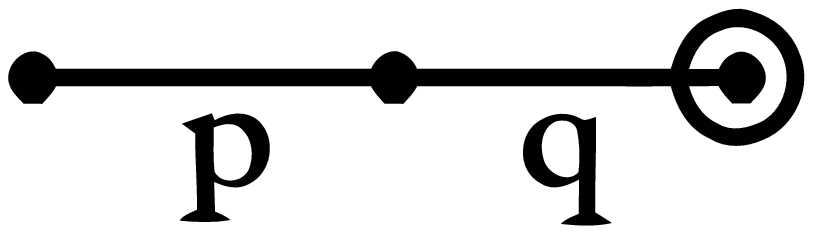}}& \raisebox{-1.8\height}{$P^{2}$}&\raisebox{-1.8\height}{ $(q^p)$ }\\
\hline
\raisebox{-1.8\height}{$\{0,1\}$} & \raisebox{-.78\totalheight}{ \includegraphics*[clip, scale = .45]{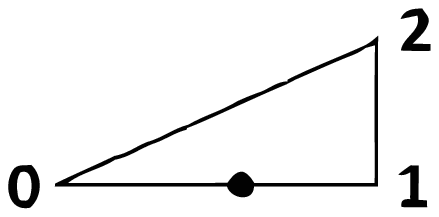}}&\raisebox{-1.15\totalheight}{ \includegraphics*[clip, scale = .25]{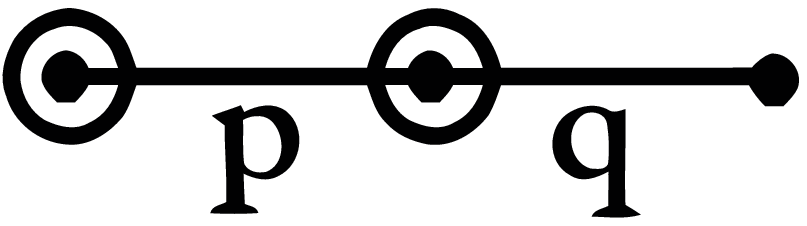}}&\raisebox{-1.8\height}{ $P^{01}$}&\raisebox{-1.8\height}{ $(p.q.q)$ }\\
\hline
\raisebox{-1.8\height}{$\{0,2\}$} & \raisebox{-.78\totalheight}{ \includegraphics*[clip, scale = .45]{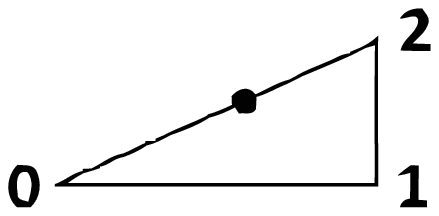}}&\raisebox{-1.15\totalheight}{ \includegraphics*[clip, scale = .25]{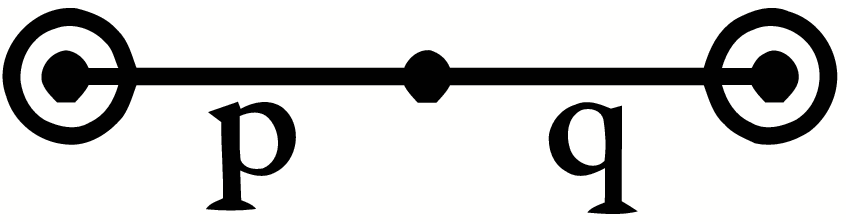}}& \raisebox{-1.8\height}{$P^{02}$} &\raisebox{-1.8\height}{ $(p.4.q.4)$ }\\
\hline
\raisebox{-1.8\height}{$\{1,2\}$} & \raisebox{-.78 \totalheight}{ \includegraphics*[clip, scale = .45]{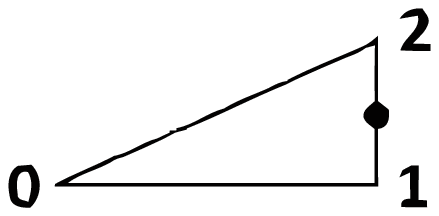}}&\raisebox{-1.15\totalheight}{ \includegraphics*[clip, scale = .25]{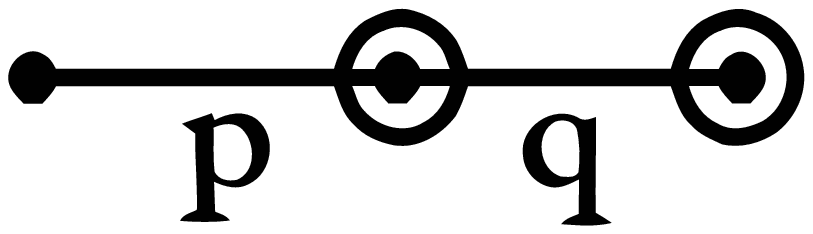}}& \raisebox{-1.8\height}{$P^{12}$} &\raisebox{-1.8\height}{ $(2p.2p.q)$ }\\
\hline
\raisebox{-1.8\height}{$\{0,1,2\}$ }& \raisebox{-.8\totalheight}{ \includegraphics*[clip, scale = .45]{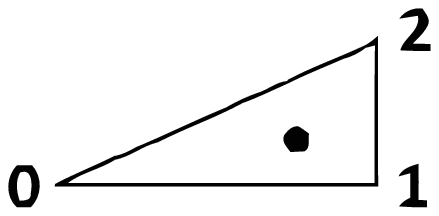}}&\raisebox{-1.15\totalheight}{ \includegraphics*[clip, scale = .25]{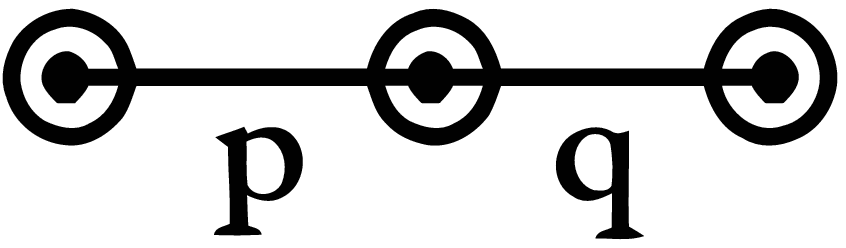}}& \raisebox{-1.8\height}{$P^{012}$}&\raisebox{-1.8\height}{ $(2p.2q.4)$ }\\
\hline
\end{tabular}
\caption{Notation for the Wythoffian, $\mathcal{P}^I$, based on choice of $I$.}
\label{notationtable}
\end{center}
\end{table}

The abstract Wythoffians to be defined will be in one-to-one correspondence with the seven subsets in the partition of $\Phi$ and hence be parametrized by subsets $I$ of $\{0,1,2\}$. Each subset in the partition is characterized as the set of possible locations for the initial vertex of the corresponding Wythoffian; different choices of initial vertices within each subset will produce isomorphic Wythoffians. Thus a subset $I$ indexing an abstract Wythoffian of $\Po$ specifies precisely the generators, namely the generators $\rho_i$ with $i\in I$, under which the corresponding initial vertex is transient. The generators $\rho_i$ with $i\notin I$ then leave the initial vertex invariant; in fact, the choice of initial vertex within the fundamental triangle is such that its stabilizer in $\Gamma(\mathcal{P})$ is precisely given by the subgroup $\langle\rho_{i}\mid i\notin I\rangle$. It then follows that the vertices of the Wythoffian are in one-to-one correspondence with the left cosets of this subgroup in $\Gamma(\Po)$. We write $\Po^I$ for the Wythoffian associated with $I$. 

Table~\ref{notationtable} indicates the seven possible placements for the initial vertex inside the fundamental triangle. 
We have adopted an analogue of Coxeter's~\cite{Cox} diagram notation for truncations of regular convex polyhedra, and have included the corresponding diagrams in the third column; following Coxeter's convention, a node of the diagram for $\Po^I$ is ringed if and only if its label belongs to $I$. We also refer to these diagrams as {\em ringed diagrams}. The fourth column shows the Wythoffian (in Latin letters, for the realizations), where we have written $\Po^i$, $\Po^{ij}$ and $\Po^{ijk}$ in place of $\Po^{\{i\}}$, $\Po^{\{i,j\}}$ or $\Po^{\{i,j,k\}}$, respectively.   The final column gives the vertex symbol for each Wythoffian. These symbols will discussed in further detail in 
Section~\ref{wythoffforregular}.    Note that in the present context the basic, ``unringed" Coxeter diagrams are not generally representing Coxeter groups~\cite[Ch. 3]{SchMc} as for regular convex polyhedra.  Here they are representing symmetry groups of arbitrary regular polyhedra in $\E$. For a regular convex polyhedron, the seven Wythoffians correspond to the seven possible ways of ``truncating" the given polyhedron~\cite{Cox}.

The {\em abstract Wythoffian\/} $\Po^I$ for a given subset $I\subseteq \{0,1,2\}$ then is constructed as follows. Choose a point $v$, the {\em initial vertex\/} of $\Po^I$, inside the fundamental chamber $\Phi$ on $S$ such that $\rho_i(v)\neq v$ for $i\in I$ and $\rho_i(v)=v$ for $i\notin I$. We first generate the base faces for $\Po^I$. Here we need to broaden the term \emph{base $l$-face} to include any $l$-face of $\Po^I$ incident with $v$ whose vertex set on $S$ is the orbit of $v$ under precisely $l$ distinguished generators of $\Gamma(\mathcal{P})$. Unlike in the case of regular polyhedra we now can have up to three different kinds of base $l$-face for $l=1,2$. For instance, when $I=\{0,1,2\}$ there is a base 1-face corresponding to each distinguished generator in $\Gamma(\mathcal{P})$. To fully define the poset of faces obtained by Wythoff's construction we will give explicit definitions of each type of base $l$-face for $l=0,1,2$.

There is only one base 0-face, namely
\begin{equation}
\label{fzero}
F_{0}:=v. 
\end{equation} 
For each $i\in I$ we define the base 1-face 
\begin{equation}
\label{fone}
F_1^{i}:=\{\rho(v)\, |\, \rho\in\langle \rho_i\rangle\}=\{ v,\rho_i(v)\}.
\end{equation} 
Thus the number of base $1$-faces in $\Po^I$ is $|I|$. The base 2-faces of $\Po^I$ will be parametrized by the set $I^2$ of 2-element subsets of $\{0,1,2\}$ given by 
\begin{equation}
\label{itwo}
I^2 := \left\{
\begin{array}{ll}
\{\{i,j\}\,|\, j=i\pm 1\} &\mbox{ if } I=\{i\},i=0,1,2,\\
\{\{0,1\},\{1,2\}\} &\mbox{ if } I=\{0,1\},\{1,2\},\\
\{\{0,1\},\{1,2\},\{0,2\}\} &\mbox{ if } I=\{0,2\},\{0,1,2\}.
\end{array}
\right.
\end{equation}
Now for each $\{i,j\}\in I^2$ we can define a base 2-face 
\begin{equation}
\label{ftwo}
F_2^{ij}:=F_2^{ji}:=\{\rho(F_1^{k})\,|\, \rho\in\langle \rho_i,\rho_j\rangle,\, k\in I\cap \{i,j\}\}.
\end{equation}
The full Wythoffian $\Po^I$ then is the union (taken over $l=0,1,2$) of the orbits of the base $l$-faces under $\Gamma(\Po)$. In $\Po^I$, vertices are points on $S$, edges are 2-element subsets consisting of vertices, and $2$-faces are sets of edges. The partial order between faces of consecutive ranks is given by containment (meaning that the face of lower rank is an element of the face of higher rank), and the full partial order then is the transitive closure. When a least face (of rank~$-1$) and a largest face (of rank $3$) are appended $\Po^I$ becomes an abstract polyhedron. Note that when $\Po$ is locally finite the abstract Wythoffian $\Po^I$ can be realized as a face-to-face tessellation on the surface $S$ in much the same way in which the geometric Wythoffians of the regular plane tessellation $\{4,4\}$ were derived; then edges are simple curves and faces are topological polygons on $S$. 

The Wythoffian $\mathcal{P}^{\{0\}}$ is isomorphic to $\Po$ itself, and $\mathcal{P}^{\{2\}}$ is isomorphic to the dual $\Po^*$ of~$\Po$. Thus both are regular. The Wythoffian $\mathcal{P}^{\{1\}}$ is isomorphic to the medial of $\Po$; it has $p$-gonal faces and $q$-gonal faces, if $\Po$ is of type $\{p,q\}$, and its vertices have valency $4$. Recall that the {\em medial\/} of a polyhedron is a new polyhedron (on the same surface), with vertices at the ``midpoints'' of the old edges and with edges joining two new vertices if these are the midpoints of adjacent edges in an old face (see~\cite{PiRa,PiZi}). 

The two Wythoffians $\Po^{\{0,1\}}$ and $\Po^{\{1,2\}}$ each have two base $2$-faces and thus two kinds of 2-face: $\Po^{\{0,1\}}$ has $2p$-gons and $q$-gons, and $\Po^{\{1,2\}}$ has $p$-gons and $2q$-gons. Each has $3$-valent vertices. On the other hand, $\Po^{\{0,2\}}$ and $\Po^{\{0,1,2\}}$ each have three base $2$-faces and thus three kinds of 2-face:\ $\Po^{\{0,2\}}$ has $p$-gons, $4$-gons and $q$-gons, and $\Po^{\{0,1,2\}}$ has $2p$-gons, $4$-gons and $2q$-gons. The vertices of $\Po^{\{0,2\}}$ are $4$-valent and those of $\Po^{\{0,1,2\}}$ $3$-valent.

Observe that the exchange of indices $0\leftrightarrow 2$ on an index set $I$ for a Wythoffian, results in the index set for the Wythoffian of the dual polyhedron $\Po^*$; that is, if $I^{*}= \{2-i|i\in I\}$ then $\Po^{I^*}=(\Po^*)^I$.  Thus the dual $\Po^*$ has the same set of seven Wythoffians as the original polyhedron $\Po$. Moreover, if $\Po$ is self-dual then the Wythoffians $\Po^I$ and $\Po^{I^*}$ are isomorphic for each $I$. 

It is worth noting that the abstract Wythoffians $\Po^I$ described in this section can be described purely combinatorially without any explicit reference to the underlying surface. This is of little interest when $\Po$ is locally finite, since then the $2$-faces of the Wythoffians are topological polygons with finitely many edges. However, if $\Po$ has apeirogonal $2$-faces or vertex-figures, respectively, the base $2$-face of $\Po^I$ generated from the subgroup $\langle\rho_{0},\rho_1\rangle$ or $\langle\rho_{1},\rho_2\rangle$ of $\Gamma(\Po)$ is an apeirogon and does not bound a disk in $S$. In our applications, while the $2$-faces of $\Po$ may be infinite, the vertex-figures of $\Po$ will always be finite.  In this case, if $\Po$ has apeirogons as $2$-faces then $\Po^I$ also has an apeirogonal $2$-face, except when $I=\{2\}$ and $\Po^I = \Po^*$.

Note that the Wythoffians of an abstract polyhedron $\Po$ can also described in terms of the monodromy group of the polyhedron (see~\cite{Mon}). In the case of a regular polyhedron $\Po$, the monodromy group and automorphism group are isomorphic and either can be chosen to define the Wythoffians $\Po^I$. However, as we will work in a geometric context where automorphisms become isometries, we have adopted an automorphism based approach to Wythoffians.

\section{Wythoffians of geometric polyhedra}
\label{geowyth}

In this section we discuss Wythoffians for geometrically {\em regular\/} polyhedra $P$ in $\E$. In particular, we explain how an abstract Wythoffian associated with $P$ as an abstract polyhedron, can often itself be realized faithfully in $\E$ in such a way that all combinatorial symmetries of $P$ are realized as geometric symmetries, and thus be viewed as a geometric Wythoffian of~$P$. In fact, whenever a realization exists there are generally many such realizations. For a point $u\in\E$ we let $G_{u}(P)$ denote the stabilizer of $u$ in $G(P)$.

The key idea is to place the initial vertex for the realization inside a specified fundamental region of the symmetry group $G(P)$ in $\E$ and then let Wythoff's construction applied with the generating reflections of $G(P)$ produce the desired geometric Wythoffian. The precise construction is detailed below. The fundamental region of $G(P)$ can be quite complicated and is generally not a simplicial cone as for the Platonic solids. The generating symmetries $r_0,r_1,r_2$ of $G(P)$ corresponding to the abstract symmetries $\rho_0,\rho_1,\rho_2$ of $\Gamma(P)$ are involutory isometries in $\E$ and therefore are point reflections, halfturns (line reflections), or plane reflections, with mirrors of dimension $0$, $1$ or $2$, respectively. In order to realize an abstract Wythoffian $P^I$ (with $I\subseteq \{0,1,2\}$) of the given geometric polyhedron $P$ in $\E$, the initial vertex $v$ must be chosen such that 
\begin{equation}
\label{inipla}
G_{v}(P) = \langle r_{i}\mid i\notin I\rangle.
\end{equation} 
This {\it initial placement condition\/} will allow us to construct a {\it faithful} realization of $P^I$. In fact, (\ref{inipla}) is a necessary and sufficient condition for the existence of a faithful realization of $P^I$ in $\E$ which is induced by the given realization of $P$ in the sense that all geometric symmetries of $P$ are also geometric symmetries of $P^I$. Note that condition (\ref{inipla}) implies the more easily verifiable condition
\begin{equation}
\label{iniplaforr}
r_i(v)\neq v\;(i\in I),\;\;r_i(v)=v\;(i\notin I),
\end{equation}
which for specific points $v$ usually is equivalent to (\ref{inipla}).

The shape of the geometric Wythoffians will vary greatly with the choice of initial vertex. Our assumption that $v$ be chosen inside the fundamental region for $G(P)$ is, strictly speaking, not required. 
The initial placement condition for $v$ alone guarantees that a faithful realization of $P^I$ can be found by Wythoff's construction. However, if the initial vertex $v$ is chosen inside the fundamental region, then the original polyhedron $P$ and its Wythoffian $P^I$ are similar looking in shape and so their intrinsic relationship is emphasized. 

By the very nature of the construction, Wythoffians are vertex-transitive and have vertex-transitive faces. If the faces are actually regular polygons, then the Wythoffian is a geometrically uniform polyhedron in $\E$.

\subsection{Regular polyhedra in $\E$}

We briefly review the classification of the geometrically regular polyhedra in $\E$ following the classification scheme of \cite[Sect.~7E]{SchMc} (or \cite{McSch1997}). There are 48 such regular polyhedra, up to similarity and scaling of components (if applicable):\ 18 finite polyhedra, 6 planar apeirohedra, 12 blended apeirohedra, and 12 pure (non-blended) apeirohedra. They are also known as the Gr\"unbaum-Dress polyhedra.

The finite regular polyhedra comprise the five Platonic solids $\{3,3\}$, $\{3,4\}$, $\{4,3\}$, $\{3,5\}$, $\{5,3\}$ and the four Kepler-Poinsot star-polyhedra $\{3,\frac{5}{2}\}$, $\{\frac{5}{2},3\}$, $\{5,\frac{5}{2}\}$, $\{\frac{5}{2},5\}$, where the faces and vertex-figures are planar but are permitted to be star polygons (the entry $\frac{5}{2}$ indicates pentagrams as faces or vertex-figures); and the Petrie-duals of these nine polyhedra. 

The planar regular apeirohedra consist of the three regular plane tessellations $\{4,4\}$, $\{3,6\}$ and $\{6,3\}$, and their Petrie-duals $\{\infty,4\}_4$, $\{\infty,6\}_3$ and $\{\infty,3\}_6$, respectively.

There are twelve regular apeirohedra that are ``reducible'' and have components that are lower-dimensional regular figures. These apeirohedra are {\it blends\/} of a planar regular apeirohedron $P$ and a line segment $\{\,\}$ or linear apeirogon $\{\infty\}$. The notion of a blend used in this context is a variant of the notion of a blend of two realizations of abstract polyhedra described earlier (but is technically not the same). The formal definition is as follows. We let $P'$ denote the line segment $\{\,\}$ or the linear apeirogon $\{\infty\}$.

Suppose the symmetry groups of $P$ and $P'$ are, respectively, $G(P)=\langle r_0,r_1,r_2\rangle$ and $G(P')=\langle r_{0}'\rangle$ or $\langle r_{0}',r_{1}'\rangle$. For our purposes, $G(P)$ acts on a plane in $\mathbb{E}^3$ while $G(P')$ acts on a line perpendicular to that plane; in particular, these two groups commute at the level of elements. The blending process requires us first to take the direct product of the groups, $G(P)\times G(P')$, viewed as a subgroup of the full isometry group of $\E$. The new regular apeirohedron, the {\em blend\/} $P\# P'$, then is obtained from the subgroup of $G(P)\times G(P')$ generated by the set of involutions 
\[ (r_0,r_{0}'),(r_1,1),(r_2,1)\]
or 
\[ (r_0,r_{0}'),(r_1,r_{1}'),(r_2,1),\] 
respectively; this subgroup is the symmetry group of the blend, and the involutions are the distinguished generators. Thus $G(P\# P')$ is a subgroup of $G(P)\times G(P')$. In particular, if the plane of $P$ and line of $P'$ meet at the origin, and $v$ and $v'$ are the initial vertices of $P$ and $P'$ for Wythoff's construction, then the point $(v,v')$ in $\E$ is the initial vertex of the blend. More explicitly, the blend, $P\# \{\,\}$, of $P$ and $\{\,\}$ has symmetry group $\langle (r_0,r_{0}'),(r_1,1),(r_2,1)\rangle$ while the blend, $P\#\{\infty\}$, of $P$ and $\{\infty\}$ has symmetry group $\langle (r_0,r_{0}'),(r_1,r_{1}'),(r_2,1)\rangle$. Throughout we will simplify the notation from $(r,r')$ to $rr'$ for an element of $G(P) \times G(P')$. 

For example, the blend of the standard square tessellation $\{4,4\}$ and the linear apeirogon $\{\infty\}$, denoted $\{4,4\}\#\{\infty\}$, is an apeirohedron in $\E$ whose faces are helical apeirogons (over squares), rising as ``spirals'' above the squares of $\{4,4\}$ such that $4$ meet at each vertex; the orthogonal projections of $\{4,4\}\#\{\infty\}$ onto their component subspaces recover the original components, that is, the square tessellation and the linear apeirogon. Each blended apeirohedron represents an entire family of apeirohedra of the same kind, where the apeirohedra in a family are determined by a parameter describing the relative scale of the two component figures; our count of 12 refers to the 12 kinds rather than the individual apeirohedra.

Finally there are twelve regular apeirohedra that are ``irreducible'', or {\it pure\/} (non-blended). These are listed in Table~\ref{tabpure} (see \cite[p. 225]{SchMc}). The first column gives the {\it mirror vector\/} of an apeirohedron; its components, in order, are the dimensions of the mirrors of the generating symmetries $r_0$, $r_1$ and $r_2$ of $G(P)$ (this is the dimension vector of \cite[Ch. 7E]{SchMc}). The last two columns say whether the faces and vertex-figures are planar, skew, or helical regular polygons. In the second, third, and fourth columns, the (rotation or full) symmetry group of the Platonic solid at the top indexing that column is closely related to the special group of each apeirohedron listed below it; the {\em special group\/} is the quotient of the symmetry group by the translation subgroup. The three polyhedra in the first row are the well-known Petrie-Coxeter polyhedra (see \cite{Cox7}), which along with those in the third row comprise the pure regular polyhedra with finite faces. The pure polyhedra with helical faces are listed in the second and last row. Infinite zigzag polygons do not occur as faces of pure polyhedra.

The fine Schl\"afli symbol used to designate a polyhedron signifies extra defining relations for the symmetry group (see \cite[Ch. 7E]{SchMc}). For example, the parameters $l$, $m$ and $n$ in the symbols $\{p,q\}_l$, $\{p,q\}_{l,m}$ and $\{p,q\,|\,n\}$ indicate the relations $(r_{0}r_{1}r_{2})^{l}=1$, ${(r_{0}(r_{1}r_{2})^{2})}^{m}=1$, or $(r_{0}r_{1}r_{2}r_{1})^{n}=1$, respectively; together with the standard Coxeter relations they form a presentation for the symmetry group of the corresponding polyhedron. Note that $l$, $m$ and~$n$, respectively, give the lengths of the Petrie polygons (1-zigzags), the $2$-zigzags (paths traversing edges where the new edge is chosen to be the second on the right, but reversing orientation on each step, according to some local orientation on the underlying surface), and the holes (paths traversing edges where the new edge is chosen to be the second on the right on the surface).

\begin{table}[htb]
\begin{center}
\begin{tabular}{|c||ccc|c|c|}
\hline
mirror vector& $\{3,3\}$ & $\{3,4\}$ & $\{4,3\}$ &faces&vertex-figures\\
\hline \hline
(2,1,2) \,&\, $\{6,6 | 3\}$ \,&\, $\{6,4 | 4\}$ \,&\, $\{4,6 | 4\}$ \,&\,planar\,&\,skew\, \\[.05in]
(1,1,2) \,&\, $\{\infty,6\}_{4,4}$ \,&\, $\{\infty,4\}_{6,4}$ \,&\, $\{\infty,6\}_{6,3}$ \,&\,helical\,&\,skew\, \\[.05in]
(1,2,1) \,&\, $\{6,6\}_{4}$ \,&\, $\{6,4\}_{6}$ \,&\, $\{4,6\}_{6}$\,&\,skew\,&\,planar\, \\[.05in]
(1,1,1) \,&\, $\{\infty,3\}^{(a)}$ \,&\, $\{\infty,4\}_{\cdot,*3}$ \,&\,
$\{\infty,3\}^{(b)}$ \,&\,helical\,&\,planar\, \\[.05in]
\hline
\end{tabular}
\end{center}
\caption{The pure apeirohedra in $\E$}
\label{tabpure}
\end{table}

\subsection{Geometric Wythoffians of regular polyhedra}
\label{wythoffforregular}

Let $P$ be a geometrically regular polyhedron in $\E$ with symmetry group $G(P)=\langle r_0,r_1,r_2\rangle$, and let $I\subseteq \{0,1,2\}$. If we write $R_i$ for the mirror of a distinguished 
generator~$r_i$ in $\E$, and $\overline{X}$ for the complement in $\E$ of a subset $X$ of $\E$, then the weaker form (\ref{iniplaforr}) of the initial placement condition in (\ref{inipla}) for the initial vertex $v$ is equivalent to requiring that $v$ lies in 
\begin{equation}
\label{iniset}
M_{I}\;:=\;\,\bigcap_{i\in I} \overline{R_i} \;\,\cap\;\,\bigcap_{i\notin I} R_i \,.
\end{equation}
Thus, if $\widehat{M}_{I}$ denotes the set of permissible choices of initial vertices $v$, then $\widehat{M}_{I}$ is a subset of~$M_I$.  In general we would expect the complement of $\widehat{M}_{I}$ in $M_I$ to be ``small".  If $I\neq \{0,1,2\}$ the affine hull of $M_I$ is a proper affine subspace of $\E$ given by $\bigcap_{i\notin I} R_i$; in this case, $M_I$ must lie in a plane. If $I= \{0,1,2\}$ the affine hull of $M_I$ is $\E$.

Now suppose the initial vertex $v$ is chosen in $\widehat{M}_{I}$ (and also lies in a specified fundamental region of $G(P)$ in $\E$). To construct the geometric Wythoffians of $P$ we follow the same pattern as in (\ref{fzero}), (\ref{fone}), (\ref{itwo}) and (\ref{ftwo}). We often write $P^I(v)$ in place of $P^I$ in order to emphasize the fact that $P^I$ is generated from $v$. The base 0-face of $P^I(v)$ is again given by 
\begin{equation}
\label{fzerogeo}
F_{0}:=v.
\end{equation} 
For each $i\in I$ there is a base 1-face, 
\begin{equation}
\label{fonegeo}
F_1^{i}:= (v,r_i(v)), 
\end{equation} 
which is a line segment; and for each $\{i,j\}\in I^2$ there is a base 2-face, 
\begin{equation}
\label{ftwogeo}F_2^{ij}:=F_2^{ji}:=\{r(F_1^{k})\,|\, r\in\langle r_i,r_j\rangle,\, k\in I\cap \{i,j\}\},
\end{equation}
which forms a finite or infinite polygon according as $\langle r_i,r_j\rangle$ is a finite or infinite dihedral group. The full {\em geometric Wythoffian\/} $P^I(v)$ then is the union (taken over $l=0,1,2$) of the orbits of the base $l$-faces under $G(P)$. 

We often write $P^i(v)$, $P^{ij}(v)$, $P^{ijk}(v)$ in place of $P^{\{i\}}(v)$, $P^{\{i,j\}}(v)$ or $P^{\{i,j,k\}}(v)$, respectively, and similarly without $v$ as qualification. 

Observe that our construction of geometric Wythoffians always uses a geometrically regular polyhedron in $\E$ as input and then produces from it a realization of its abstract Wythoffian. Thus the pair of abstract polyhedra $(\Po,\Po^I)$ is simultaneously realized in $\E$ as a pair of geometric polyhedra $(P,P^I)$. The following lemma shows that the geometric Wythoffians are indeed faithful realizations of the abstract Wythoffians. 

\begin{lemma}
\label{faith}
Let $P$ be a geometrically regular polyhedron in $\E$, and let $I\subseteq \{0,1,2\}$. Then for each $v\in\widehat{M}_{I}$ the geometric Wythoffian $P^I(v)$ is a faithful realization of the abstract Wythoffian $\Po^I$.
\end{lemma}

\begin{proof}
The initial placement condition in (\ref{inipla}) for $v$ implies that there is a one-to-one correspondence between the vertices of $\Po^I$ and $P^I(v)$. In fact, by construction, the vertices of $\Po^I$ and $P^I(v)$, respectively, are in one-to-one correspondence with the left cosets of the stabilizers of the initial vertices in $\Gamma(\Po)$ or $G(P)$, which are given by $\langle\rho_{i}\mid i\notin I\rangle$ and $G_{v}(P)$. But the group isomorphism between $\Gamma(\Po)$ and $G(P)$ naturally takes the vertex stabilizer $\langle\rho_{i}\mid i\notin I\rangle$ to $\langle r_{i}\mid i\notin I\rangle$, and by (\ref{inipla}) the latter subgroup of $G(P)$ coincides with $G_{v}(P)$. Thus there is a bijection between the two vertex sets, and the number of vertices is the index of $\langle r_{i}\mid i\notin I\rangle$ in $G(P)$.

As for both the abstract and geometric Wythoffian the base edges and base faces are entirely determined by their vertices, and the overall construction method for the two polyhedra is the same, the one-to-one correspondence between the two vertex sets extends to an isomorphism between the two polyhedra.\hfill$\Box$
\end{proof}

By construction, the Wythoffian $P^I(v)$ inherits all geometric (and combinatorial) symmetries of $P$ and is (trivially) vertex-transitive under $G(P)$. Thus all vertices are surrounded alike (and in particular in the same way as $v$). Following standard notation for classical Archimedean solids and tilings we will introduce a vertex symbol for $P^I(v)$ that describes the neighborhood of a vertex and hence collects important local data. 

Let $u$ be a vertex of $P^I(v)$ of valency $k$, let $G_1,\ldots, G_k$ (in cyclic order) be the $2$-faces containing $u$, and let $G_j$ be a $q_j$-gon for $j=1,\ldots,k$ (with $q_j=\infty$ if $G_j$ is an apeirogon). Then we call $(G_1,G_2,G_3,\ldots,G_k)$ and $(q_1.q_2.q_3\,\ldots\,q_k)$ the {\em vertex configuration\/} and {\em vertex symbol\/} of $P^I(v)$ at $u$, respectively. The vertex configuration and vertex symbol at a vertex are determined up to cyclic permutation and reversal of order. By the vertex-transitivity, the vertex symbols of $P^I(v)$ at different vertices are the same and so we can safely call the common symbol the vertex symbol of $P^I(v)$ (or $P^I$). If a vertex symbol contains a string of $m$ identical entries,~$q$, we simply shorten the string to $q^m$. 

 As we will see there are several instances where certain abstract Wythoffians of geometrically regular polyhedra cannot be realized as geometric Wythoffians. This is already true for the geometrically regular polyhedra themselves. Not every geometrically regular polyhedron has  a geometrically regular polyhedron as a dual. There are several possible obstructions to this. If the original polyhedron has infinite faces, then the dual would have to have vertices of infinite valency, which is forbidden by our discreteness assumption. Thus local finiteness is a necessary condition for pairs of geometric duals to exist. However, local finiteness is not a sufficient condition. For example, the (abstract) dual of the Petrie dual of the cube, $\{3,6\}_4$, cannot be realized as a geometric polyhedron in $\E$ while the Petrie dual of the cube itself, $\{6,3\}_4$, is one of the finite regular polyhedra in $\E$. The abstract polyhedron $\{3,6\}_4$ is a triangulation of the torus; since its edges have multiplicity $2$, they cannot be geometrically represented by straight line segments in $\E$. Geometric polyhedra must necessarily have a simple edge graph.

In practice we often employ a padded vertex symbol to describe the finer geometry of the vertex-configuration. We use symbols like $p_c$, $p_s$, $\infty_k$, or $t\infty_2$, respectively, to indicate that the faces are (not necessarily regular) \underbar{c}onvex $p$-gons, \underbar{s}kew $p$-gons, helical polygons over \underbar{$k$}-gons, or truncated planar zigzag. (The $k=2$ describes a planar zigzag viewed as a helical polygon over a ``$2$-gon'', where here a $2$-gon is a line segment traversed in both directions.  A truncated planar zigzag is obtained by cutting off the vertices of a planar zigzag, while maintaining segments of the old edges as new edges.) There are other shorthands that we introduce when they occur. For example, a symbol like $(8_s^2.6_c.3^2.6_s)$ would say that each vertex is surrounded (in cyclic order) by a skew octagon, another skew octagon, a convex hexagon, a triangle, another triangle, and a skew hexagon. 

We should point out that there are uniform skeletal polyhedra that cannot occur as geometric Wythoffians of regular polyhedra in $\E$. The simplest example is the snub cube, which is an Archimedean solid whose symmetry group is the octahedral rotation group and hence does not contain plane reflections. On the other hand, all 18 finite regular polyhedra and thus their geometric Wythoffians have reflection groups as symmetry groups. Note that the snub cube can be derived by Wythoff's construction from the octahedral rotation group rather than the full octahedral group. 

The geometrically chiral polyhedra in $\E$ are also examples of uniform skeletal polyhedra that cannot arise as geometric Wythoffians of regular polyhedra (see \cite{SchChiral1,SchChiral2}). This immediately follows from a comparison of the structure of the faces and the valencies of the vertices for the Wythoffians and chiral polyhedra, except possibly when $I=\{0\}$ or $\{2\}$. In these two cases the Wythoffians are regular and thus cannot coincide with a chiral polyhedron.

We have not yet fully explored the ``snub-type'' polyhedra that arise from regular skeletal polyhedra $P$ via Wythoff's construction applied to the ``rotation subgroup" $G^+(P)$ of the symmetry group $G(P)$. 
This subgroup is generated by the symmetries $r_0r_1,r_1r_2$ and consists of all symmetries of $P$ that realize combinatorial rotations of $P$; that is, $G^+(P)$ is the image of the combinatorial rotation subgroup $\Gamma^{+}(P):=\langle \rho_0\rho_1,\rho_1\rho_2\rangle$ of $\Gamma(P)=\langle \rho_0,\rho_1,\rho_2\rangle$ under the representation $\Gamma(P)\mapsto G(P)$ in $\E$. Note that $r_0r_1$ and $r_1r_2$ may not actually be proper isometries and hence $G^+(P)$ may not only consist of proper isometries.

A similar remark also applies to possible geometric ``snub-type'' Wythoffians of the chiral polyhedra in~$\E$. 

\section{The Wythoffians of various regular polyhedra}
\label{wythvarious}

In this section we treat the Wythoffians of a number of distinguished classes of regular polyhedra in $\E$, including in particular the  four finite polyhedra with octahedral symmetry and various families of apeirohedra (the two planar and the two blended apeirohedra derived from the square tiling, as well as the three Petrie-Coxeter polyhedra). As the fundamental regions of the symmetry groups vary greatly between the various kinds of polyhedra, we address the possible choices of initial vertices in the subsections. The geometric shape of a geometric Wythoffian will greatly depend on the choice of initial vertex, and different choices may lead to geometric Wythoffians in which corresponding faces look quite differently and may be planar versus skew.  Figures of distinguishing features of the resulting Wythoffians are included; the pictures show the base faces in different colors. 

We leave the analysis of the Wythoffians for the remaining classes of regular polyhedra to the subsequent paper~\cite{Williams2} by the second author. 

\subsection{Finite polyhedra with octahedral symmetry}

There are four regular polyhedra in $\E$ with an octahedral symmetry group:\ the octahedron $\{3,4\}$ and cube $\{4,3\}$, and their Petrie-duals $\{6,4\}_3$ and $\{6,3\}_4$, respectively. The octahedron and cube produce familiar figures as Wythoffians each related to an Archimedean solid (see~\cite{Senechal}), but already their Petrie duals produce interesting new structures. The sets of distinguished generators for the four individual symmetry groups can all be expressed in terms of the set for the octahedron $\{3,4\}$. We write $G(\{3,4\})=\langle s_0,s_1,s_2\rangle$, where $s_0,s_1,s_2$ are the distinguished generators. All of the initial vertices used for Wythoffians with octahedral symmetry are chosen from within the standard fundamental region of the octahedral group, which is a closed simplicial cone bounded by the reflection planes of $s_0$, $s_1$ and $s_2$.  (Recall our previous remark about the notion of fundamental simplex, or in this case, fundamental simplicial cone.) Each of the Wythoffians in this section is related to an Archimedean solid, as the figures will show. In fact, for the Wythoffians of the convex regular polyhedra we can choose the initial vertex so that the resulting polyhedron is uniform.  That is not the case with all skeletal regular polyhedra.  For example, the Wythoffians $P^{02}$ and $P^{012}$ derived from $\{6,4\}_3$ cannot be uniform, though we can still see a relationship between them and the Archimedean solids. 

The Wythoffians of the octahedron are shown in Figure \ref{octahedron}. The first Wythoffian, $P^0$, is the regular octahedron $\{3,4\}$ itself. The Wythoffian $P^1$ is a uniform cuboctahedron. Examining $P^2$ we get the dual to the octahedron, the regular cube. The Wythoffian $P^{01}$ is a polyhedron which is isomorphic to the truncated octahedron. For a particular choice of initial vertex $P^{01}$ is the uniform truncated octahedron. The polyhedron $P^{02}$ is isomorphic to the rhombicuboctahedron, and for a carefully chosen initial vertex $P^{02}$ is the uniform rhombicuboctahedron. For $P^{12}$ Wythoff's construction yields a polyhedron isomorphic to the truncated cube which for a specifically chosen initial vertex is the uniform truncated cube. The Wythoffian $P^{012}$ is isomorphic to the truncated cuboctahedron, and for a certain initial vertex is the uniform truncated cuboctahedron. 

\begin{figure}[h]
\begin{center}
\begin{tabular}{cccc}
\raisebox{-\totalheight}{ \includegraphics*[clip, scale = .2]{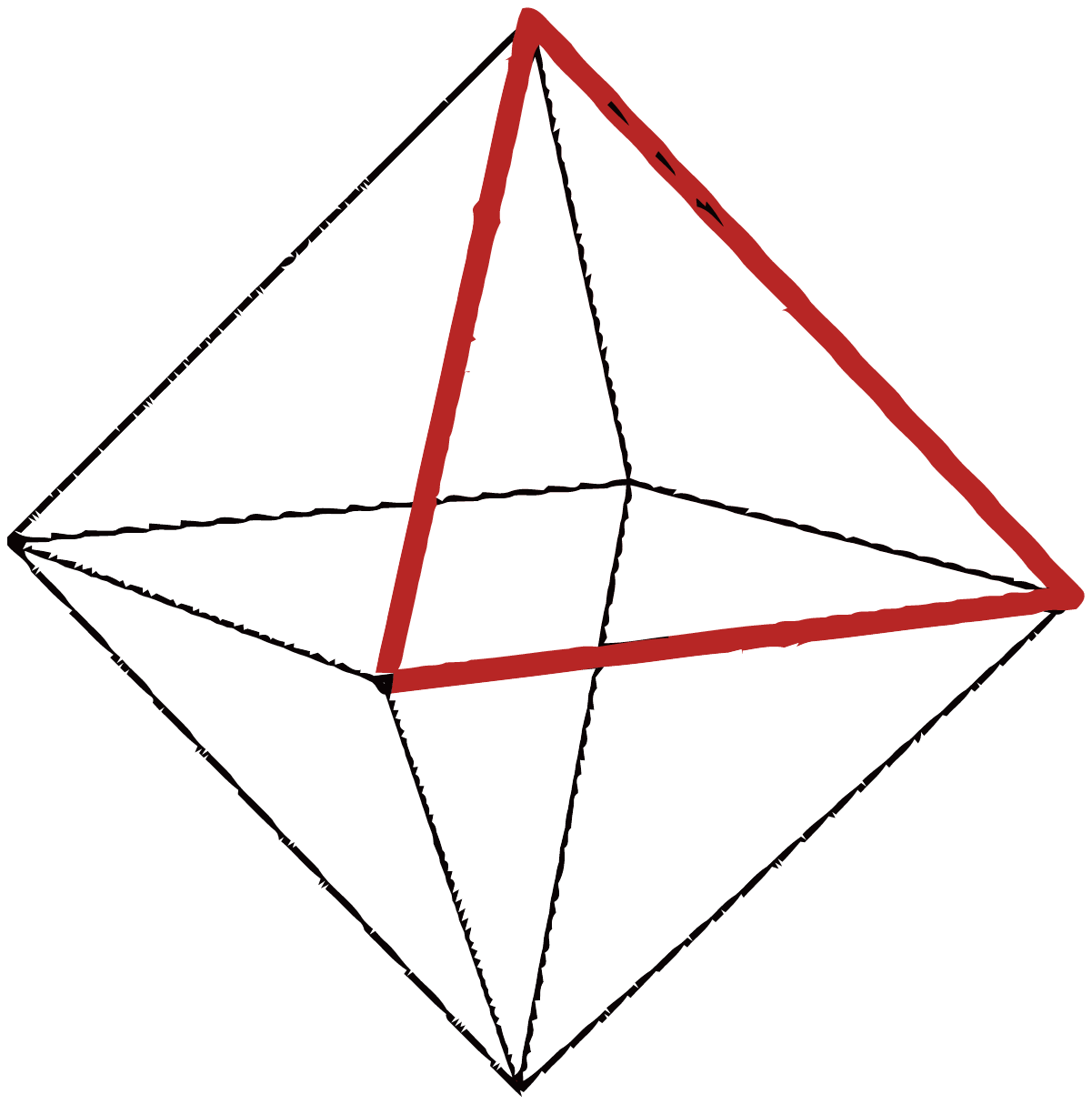}}& \raisebox{-\totalheight}{ \includegraphics*[clip, scale = .2]{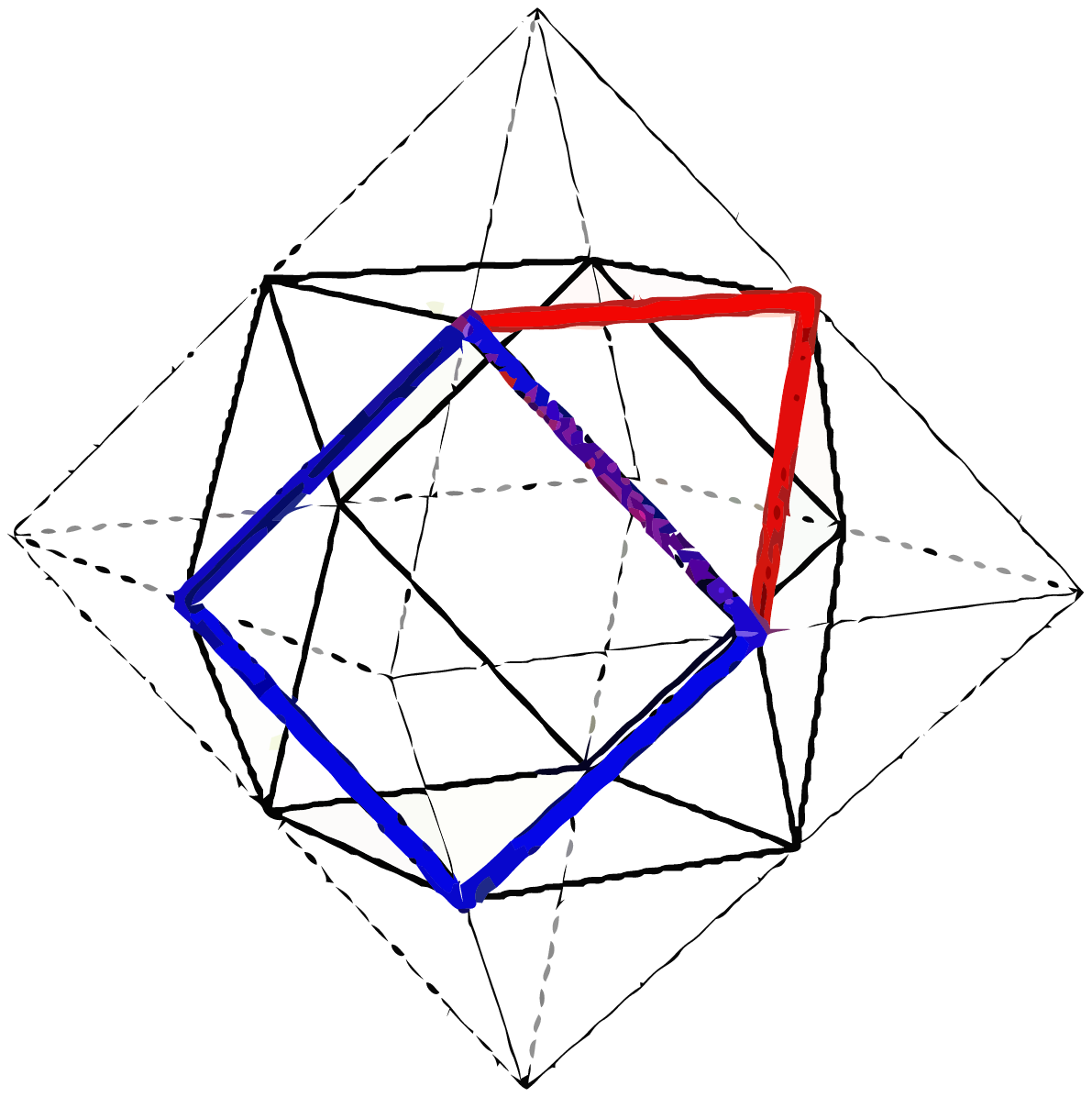}}& \raisebox{-\totalheight}{ \includegraphics*[clip, scale = .2]{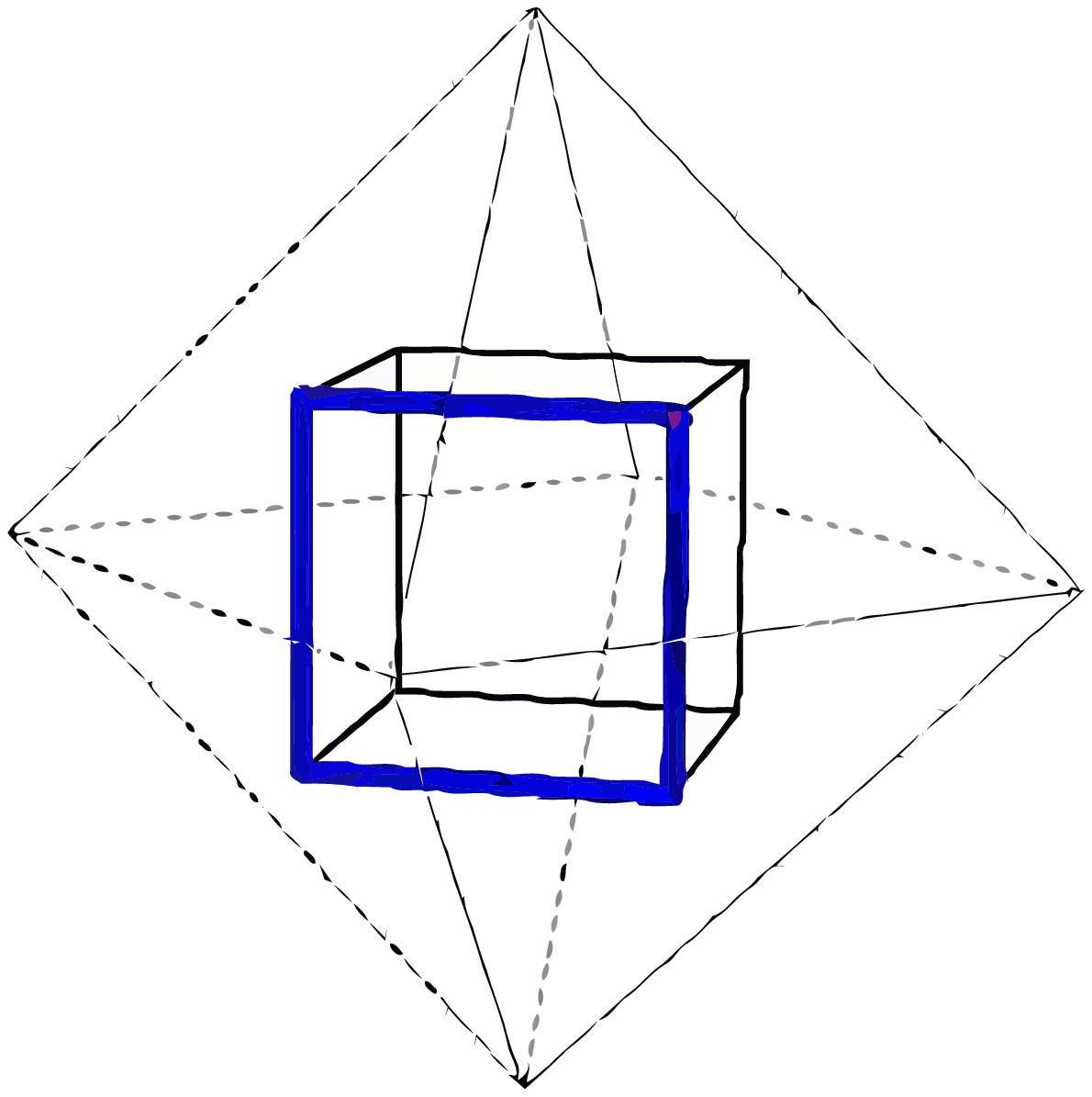}}&\raisebox{-\totalheight}{ \includegraphics*[clip, scale = .2]{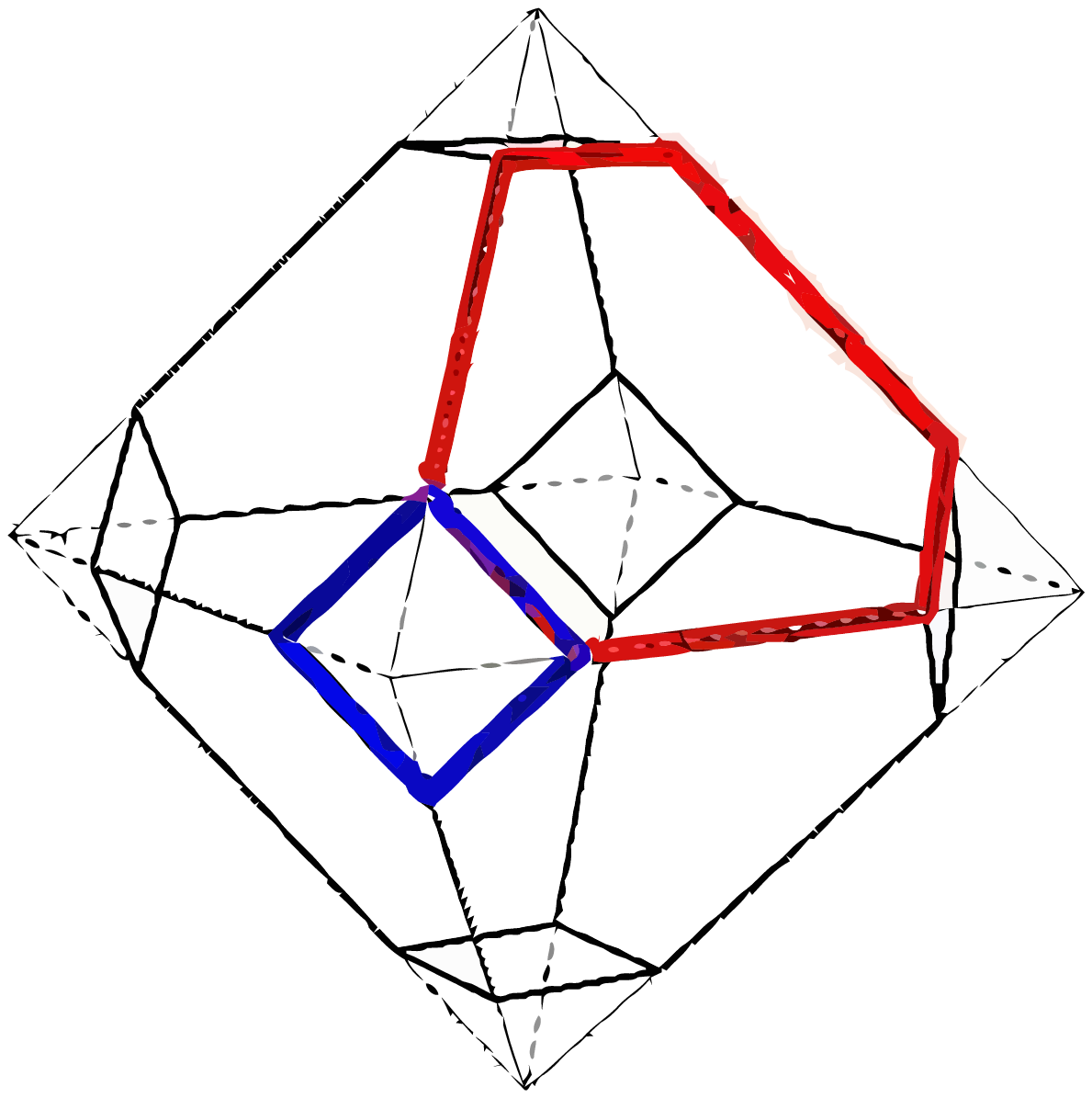}} \\
$P^{0}$ & $P^{1}$ & $P^{2}$&$P^{01}$\\
&&&\\
\raisebox{-\totalheight}{ \includegraphics*[clip, scale = .2]{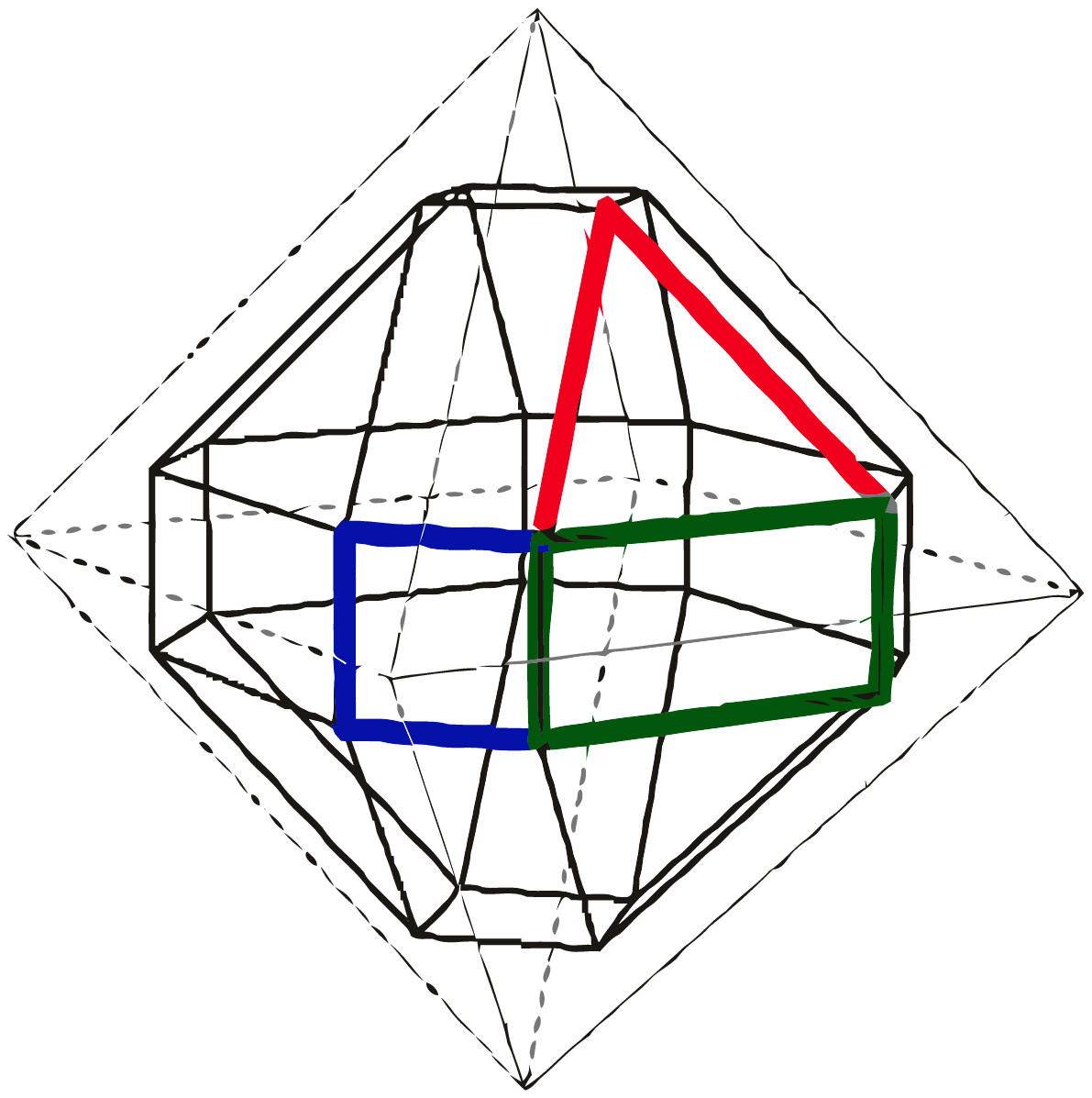}}& \raisebox{-\totalheight}{ \includegraphics*[clip, scale = .2]{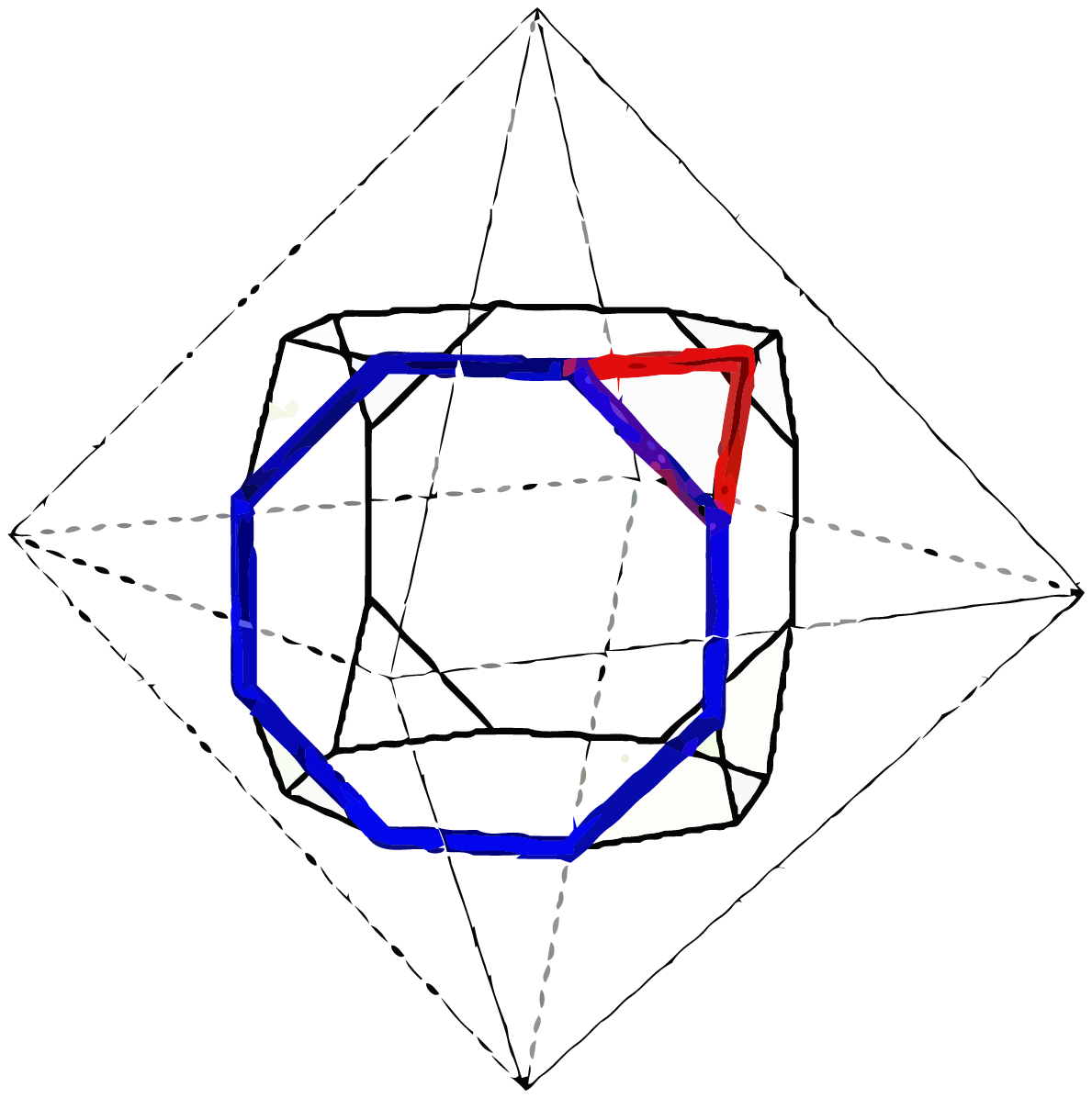}}& \raisebox{-\totalheight}{ \includegraphics*[clip, scale = .2]{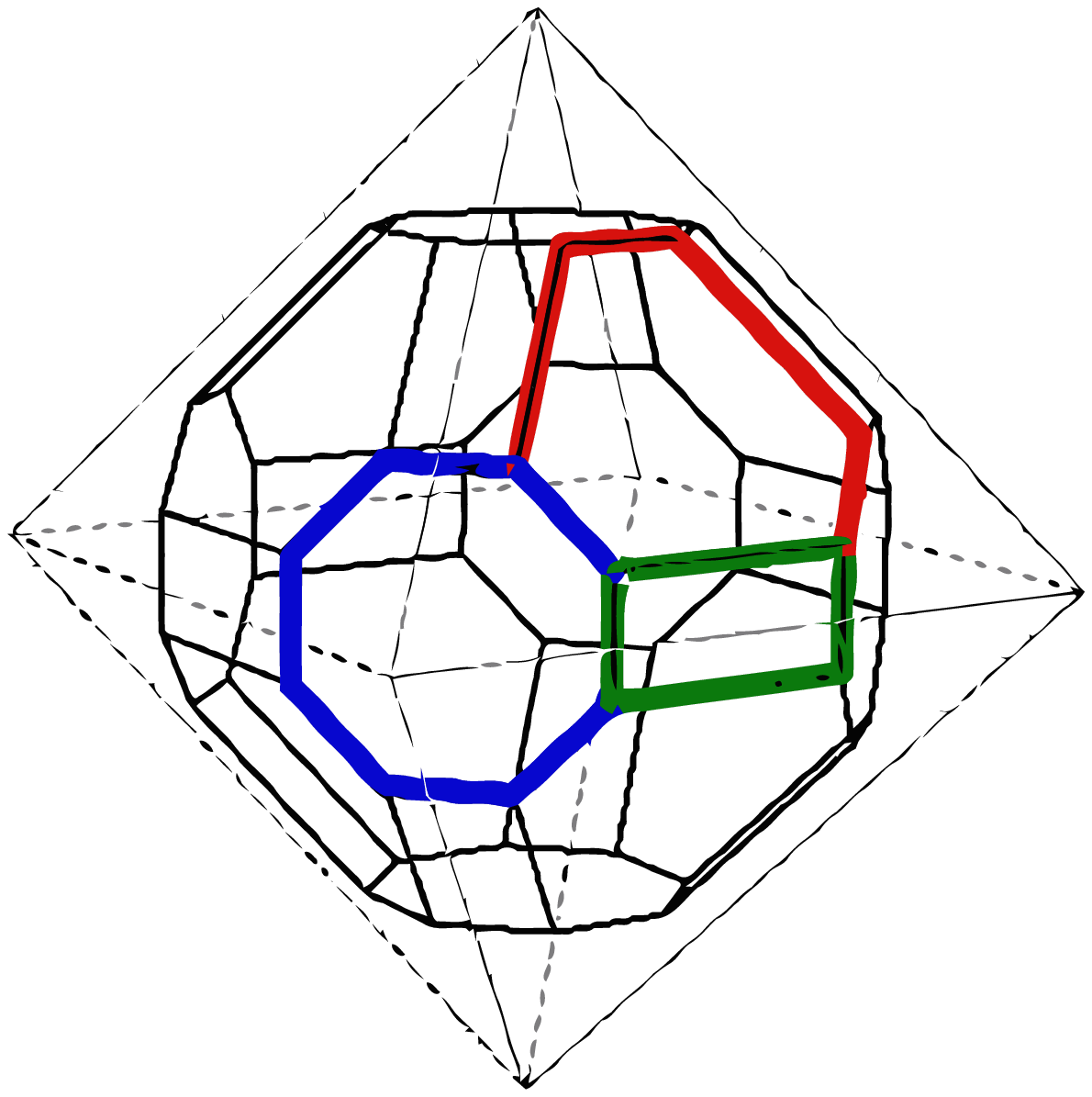}}&\\
$P^{02}$& $P^{12}$ & $P^{012}$&
\end{tabular}\caption{The Wythoffians derived from $\{3,4\}$.}\label{octahedron}\end{center}\end{figure}

For the Wythoffians of the cube $\{4,3\}$ we can exploit the duality between the cube and the octahedron (using the generators $s_2,s_1,s_0$ for $\{4,3\}$). In fact, interchanging 0 and~2 in the superscripts from the Wythoffians of the octahedron (of Figure~\ref{octahedron}) results in the Wythoffians of the cube, and vice versa. We will not reproduce the results for the cube in detail. 

The Petrie dual $\{6,4\}_3$ of $\{3,4\}$ has a group of the form $G(\{6,4\}_3)=\langle r_0,r_1,r_2\rangle$, where $r_{0}=s_{0}s_{2}$, $r_{1}:=s_{1}$, $r_{2}:=s_{2}$ and $s_0,s_1,s_2$ are as above. Given these generators, we are limited in our choice of initial vertex. As the rotation axis of the halfturn $r_0$ lies in the reflection plane of $r_2$, any point invariant under $r_0$ is also invariant under $r_2$. Thus there is no point that is invariant under only $r_0$ or under both $r_0$ and $r_1$ and not $r_2$. As such there is no polyhedron $P^2$ nor a polyhedron $P^{12}$. For pictures of the Wythoffians, see Figure \ref{octa-petrie}.

The first Wythoffian, $P^0$, is the regular polyhedron $\{6,4\}_3$ itself which has four regular, skew hexagonal faces which all meet at each vertex; the vertex symbol is $(6_s^4)$. The vertex figure is then a convex square,  as for the octahedron with which $P^0$ shares an edge graph. 

The Wythoffian $P^{1}$ shares its edge graph with the cuboctahedron. The faces are four convex, regular hexagons of type $F_2^{\{0,1\}}$ (the equatorial hexagons of the cuboctahedron) and six convex squares of type $F_2^{\{1,2\}}$. The vertex symbol is $(4_c.6_c.4_c.6_c)$. The hexagons all intersect leading to a vertex figure which is a crossed quadrilateral (like a bowtie). This is a uniform polyhedron with planar faces, in the notation of ~\cite{Cox3} it is $\frac{4}{3}\ 4\ |\ 3$. 

The polyhedron $P^{01}$ shares an edge graph with a polyhedron which is isomorphic to a truncated octahedron. It has four skew dodecagons of type $F_2^{\{0,1\}}$ (truncations of skew hexagonal faces of $\{6,4\}_3$) and six convex squares of type $F_2^{\{1,2\}}$. The vertex symbol is $(4_c.12_s^2)$ with an isosceles triangle as the vertex figure. 

The Wythoffian $P^{02}$ shares a vertex set with a polyhedron which is isomorphic to a rhombicuboctahedron. There are four skew hexagons of type $F_2^{\{0,1\}}$, six convex squares of type $F_2^{\{1,2\}}$, and twelve crossed quadrilaterals of type $F_2^{\{0,2\}}$. At each vertex a crossed quadrilateral, a square, a crossed quadrilateral, and a skew hexagon occur in cyclic order yielding a convex quadrilateral vertex figure with vertex symbol $(4_{\, \bowtie}.4_c.4.6_s)$, where $4_{\,\bowtie}$ indicates a crossed quadrilateral. 

For $P^{012}$ the resulting polyhedron shares a vertex set with a polyhedron which is isomorphic to the truncated cuboctahedron. The figure has four skew dodecagonal faces of type $F_2^{\{0,1\}}$ (truncated skew hexagons), six convex octagons of type $F_2^{\{1,2\}}$ (truncated squares), and twelve crossed quadrilaterals of type $F_2^{\{0,2\}}$. The vertex symbol is $(4_{\,\bowtie}.8_c.12_s)$ with a triangular vertex figure. 

\begin{figure}[h]
\begin{center}
\begin{tabular}{ccccc}
\raisebox{-\totalheight}{ \includegraphics*[clip, scale = .20]{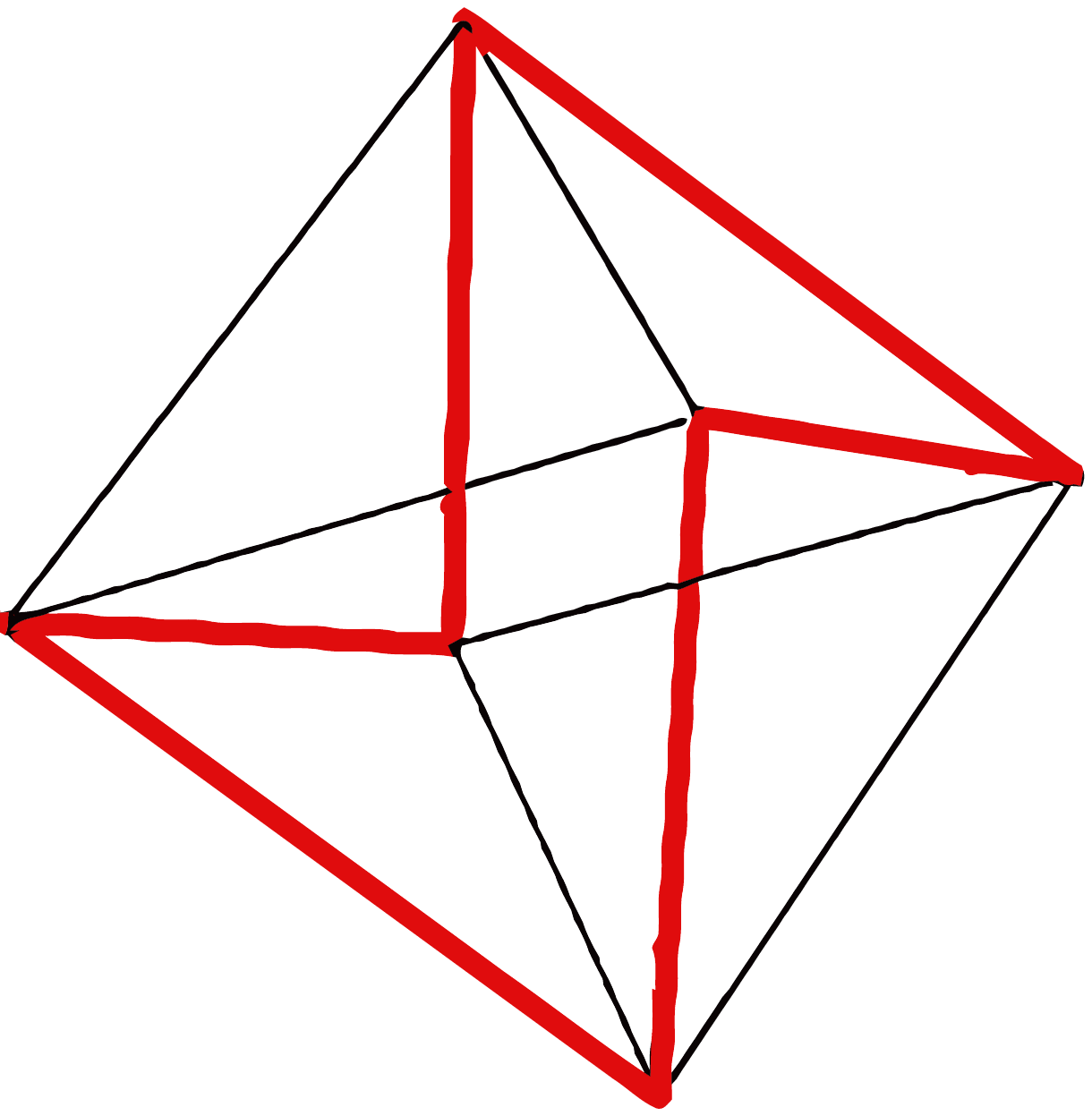}}& \raisebox{-\totalheight}{ \includegraphics*[clip, scale = .20]{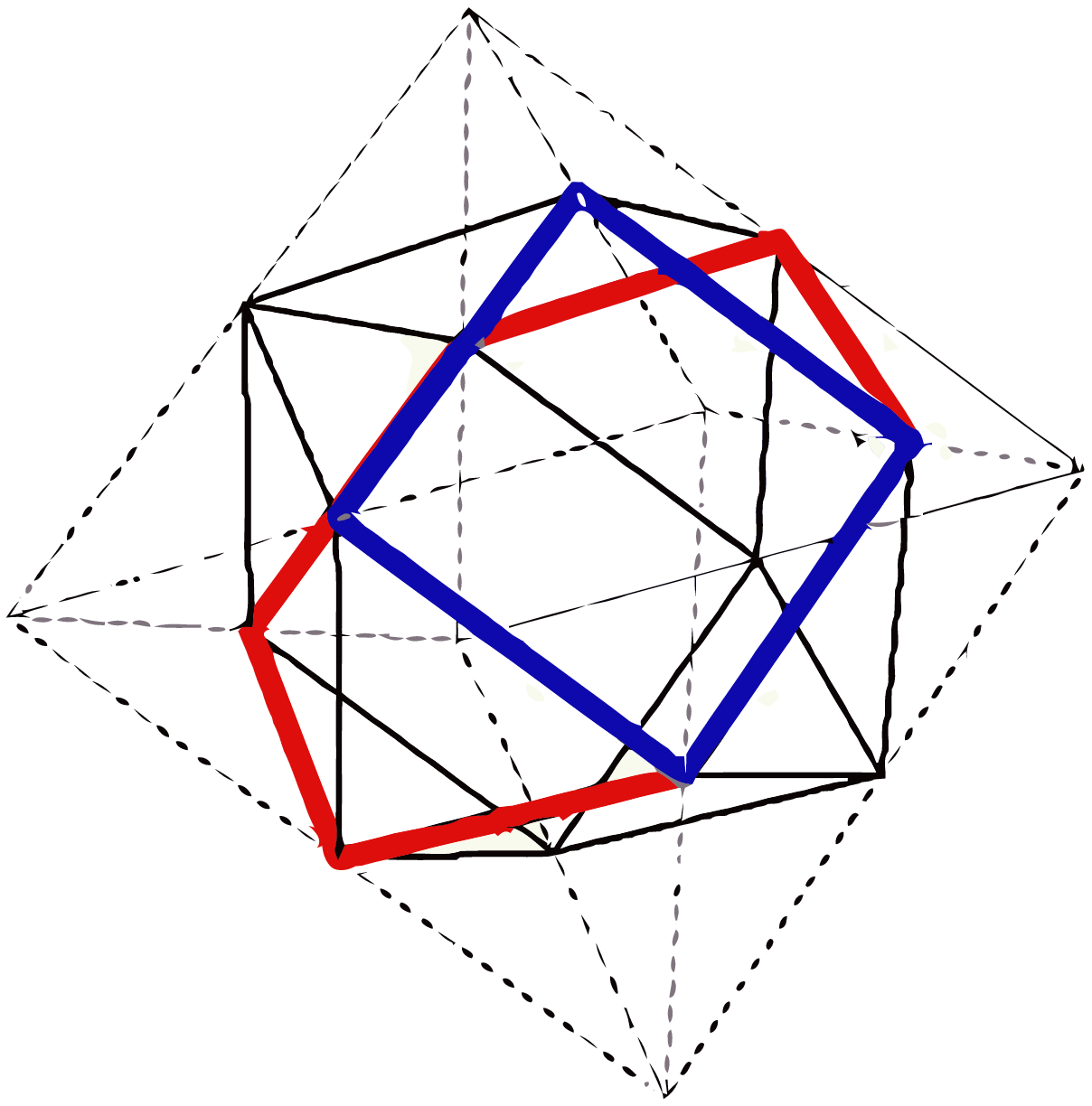}}&\raisebox{-\totalheight}{ \includegraphics*[clip, scale = .2]{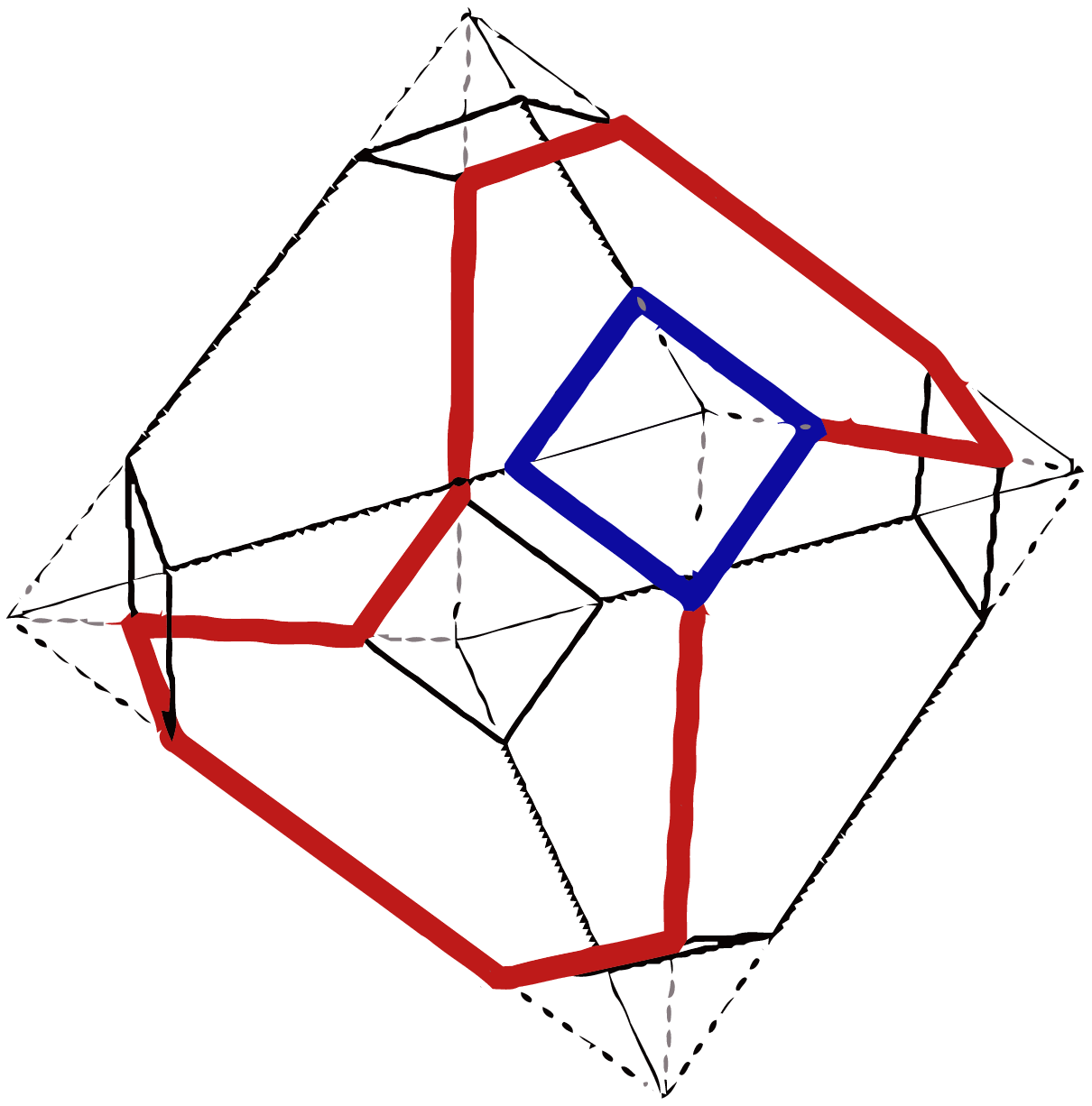}}& \raisebox{-\totalheight}{ \includegraphics*[clip, scale = .2]{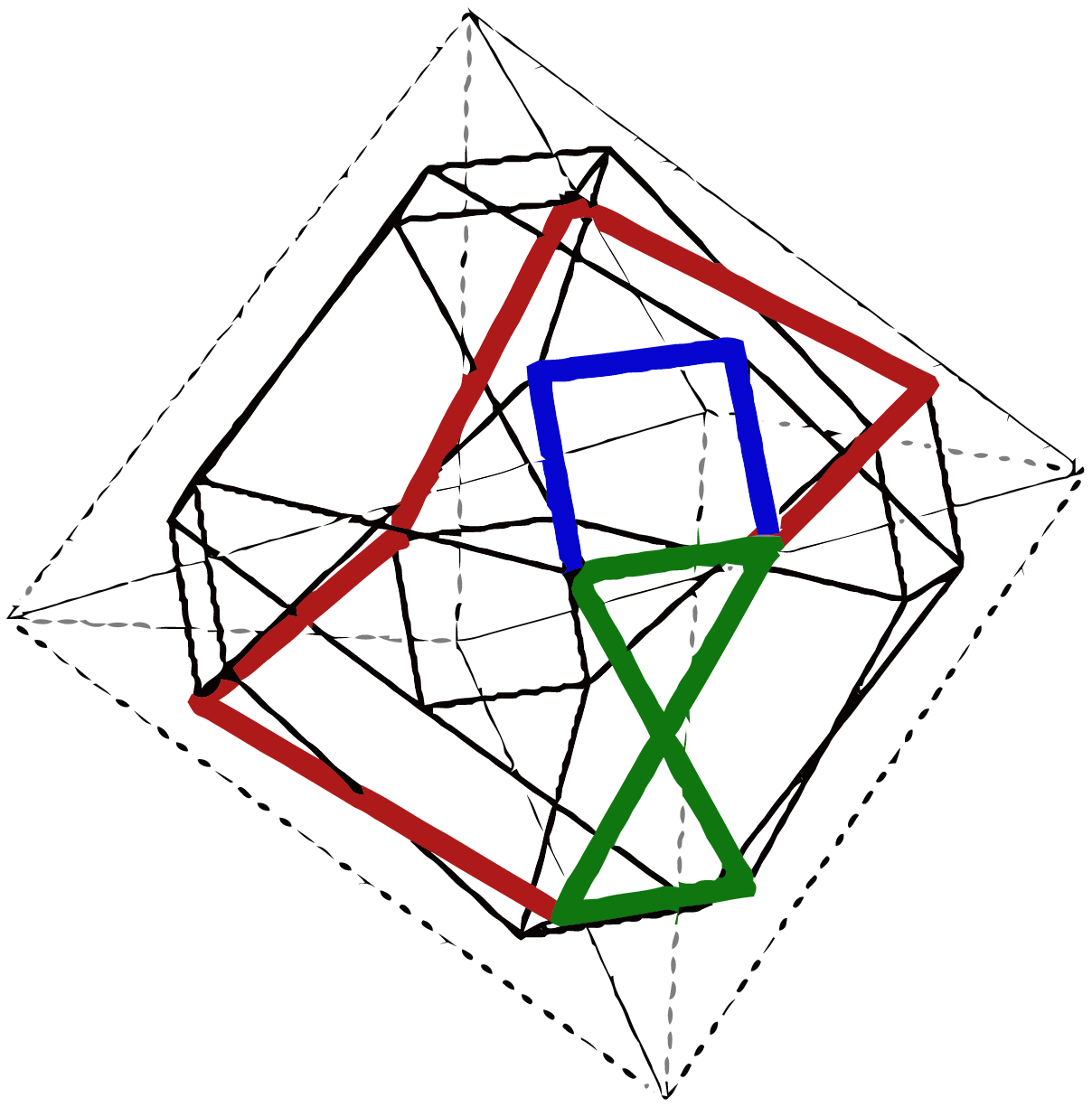}}&\raisebox{-\totalheight}{ \includegraphics*[clip, scale = .2]{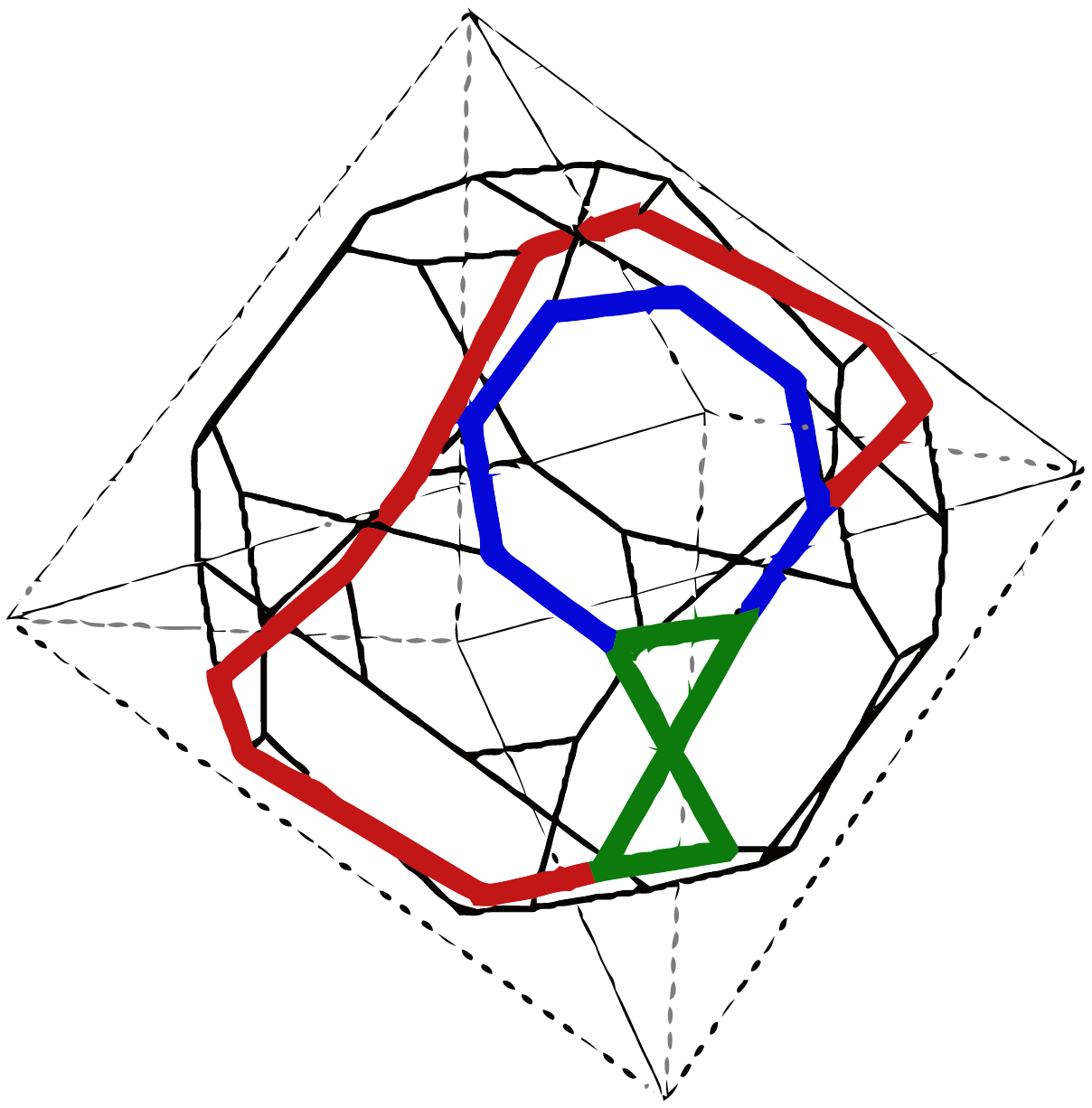}} \\
&&&&\\
$P^0$&$P^1$&$P^{01}$&$P^{02}$&$P^{012}$ 
\end{tabular}\caption{The Wythoffians derived from $\{6,4\}_3$.}\label{octa-petrie}\end{center}\end{figure}

The final geometrically regular polyhedron with octahedral symmetry is the Petrie-dual of the cube, $\{6,3\}_4$. Its symmetry group is $G(\{6,3\}_4)=\langle r_0,r_1,r_2\rangle$, where $r_{0}=s_{2}s_{0}$, $r_{1}:=s_{1}$, $r_{2}:=s_{0}$ and $s_0,s_1,s_2$ are as above. The duality between the octahedron and the cube can again be seen here. The generators $s_2s_0,\ s_1,$ and $s_0$ of $G(\{6,3\}_4)$ are obtained from the generators of $G(\{6,4\}_3)$ by interchanging $s_0$ and $s_2$. The Wythoffians of $\{6,3\}_4$ also share many similarities with the Wythoffians of $\{6,4\}_3$. As with $\{6,4\}_3$, every vertex which is stabilized by $r_0$ is also stabilized by $r_2$. Thus there is no point which is stabilized by $r_0$ alone, nor is there one which is stabilized by both $r_0$ and $r_1$. Consequently, there is no polyhedron $P^2$ and no polyhedron $P^{12}$. For pictures of the Wythoffians, see Figure \ref{cube-petrie}.

The first Wythoffian, $P^0$, is the regular polyhedron $\{6,3\}_4$ itself. It shares its edge graph with the cube and thus has eight vertices and twelve edges. The four faces are the Petrie polygons of the cube which are regular, skew hexagons. Three faces meet at each vertex, with a vertex symbol $(6_s^3)$ and a regular triangle as the vertex figure. 

The Wythoffian $P^1$ has the same edge graph as the cuboctahedron. There are four intersecting, regular, convex hexagons of type $F_2^{\{0,1\}}$ (the equatorial hexagons of the cuboctahedron) and eight regular triangles of type $F_2^{\{1,2\}}$. The vertex symbol is $(3.6_c.3.6_c)$ with a vertex figure of a crossed quadrilateral. This is a uniform polyhedron with planar faces, in the notation of ~\cite{Cox3} it is $\frac{3}{2}\ 3\ |\ 3$. 

When the initial vertex is stabilized by $r_2$ alone then the resulting polyhedron, $P^{01}$, shares its edge graph with a polyhedron which is isomorphic to the truncated cube. Then there are four skew dodecagons of type $F_2^{\{0,1\}}$ (truncations of the skew hexagonal faces of $\{6,3\}_4$) and eight regular triangles of type $F_2^{\{1,2\}}$. The vertex symbol is $(3.12_s^2)$ and the polyhedron has an isosceles triangle as a vertex figure. 

The Wythoffian $P^{02}$ shares its vertex set with a polyhedron which is isomorphic to a rhombicuboctahedron. The faces are four regular hexagons (convex or skew depending on the exact choice of initial vertex) of type $F_2^{\{0,1\}}$, eight regular triangles of type $F_2^{\{1,2\}}$, and twelve crossed quadrilaterals of type $F_2^{\{0,2\}}$. The vertex symbol is $(3.4_{\,\bowtie}.6.4_{\,\bowtie})$ and the vertex figure is a convex trapezoid. 

For $P^{012}$ the resulting polyhedron shares its vertex set with a polyhedron which is isomorphic to the truncated cuboctahedron. It has four dodecagons (which may be skew or convex depending on the choice of initial vertex) of type $F_2^{\{0,1\}}$ (truncated hexagons), eight convex hexagons of type $F_2^{\{1,2\}}$ (truncated triangles), and twelve crossed quadrilaterals of type $F_2^{\{0,2\}}$.  The vertex figure is a triangle and the
vertex symbol is $(4_{\,\bowtie}.6_c.12)$.

\begin{figure}[h]
\begin{center}
\begin{tabular}{ccccc}
\raisebox{-\totalheight}{ \includegraphics*[clip, scale = .2]{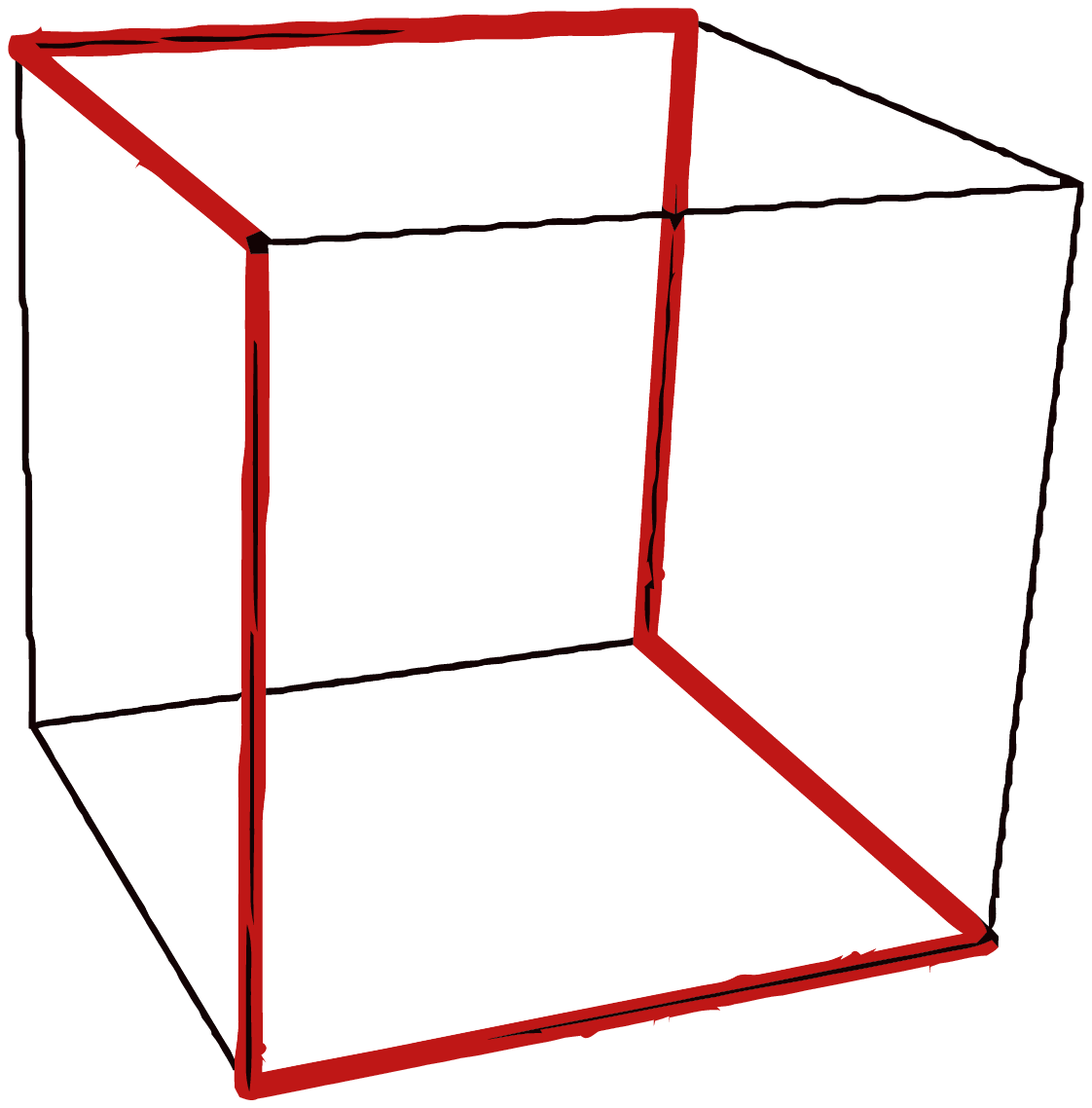}}& \raisebox{-\totalheight}{ \includegraphics*[clip, scale = .2]{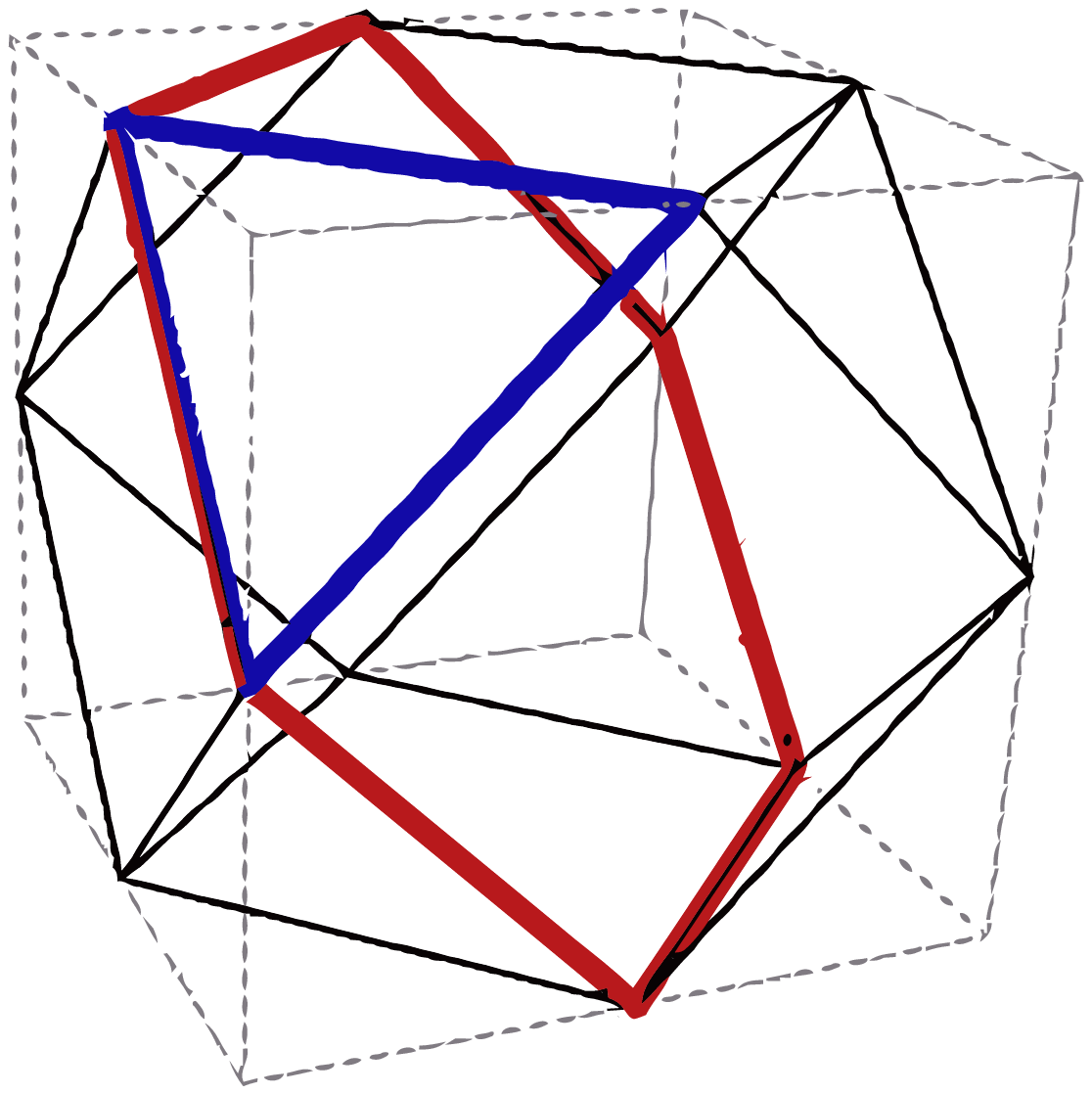}}&\raisebox{-\totalheight}{ \includegraphics*[clip, scale = .2]{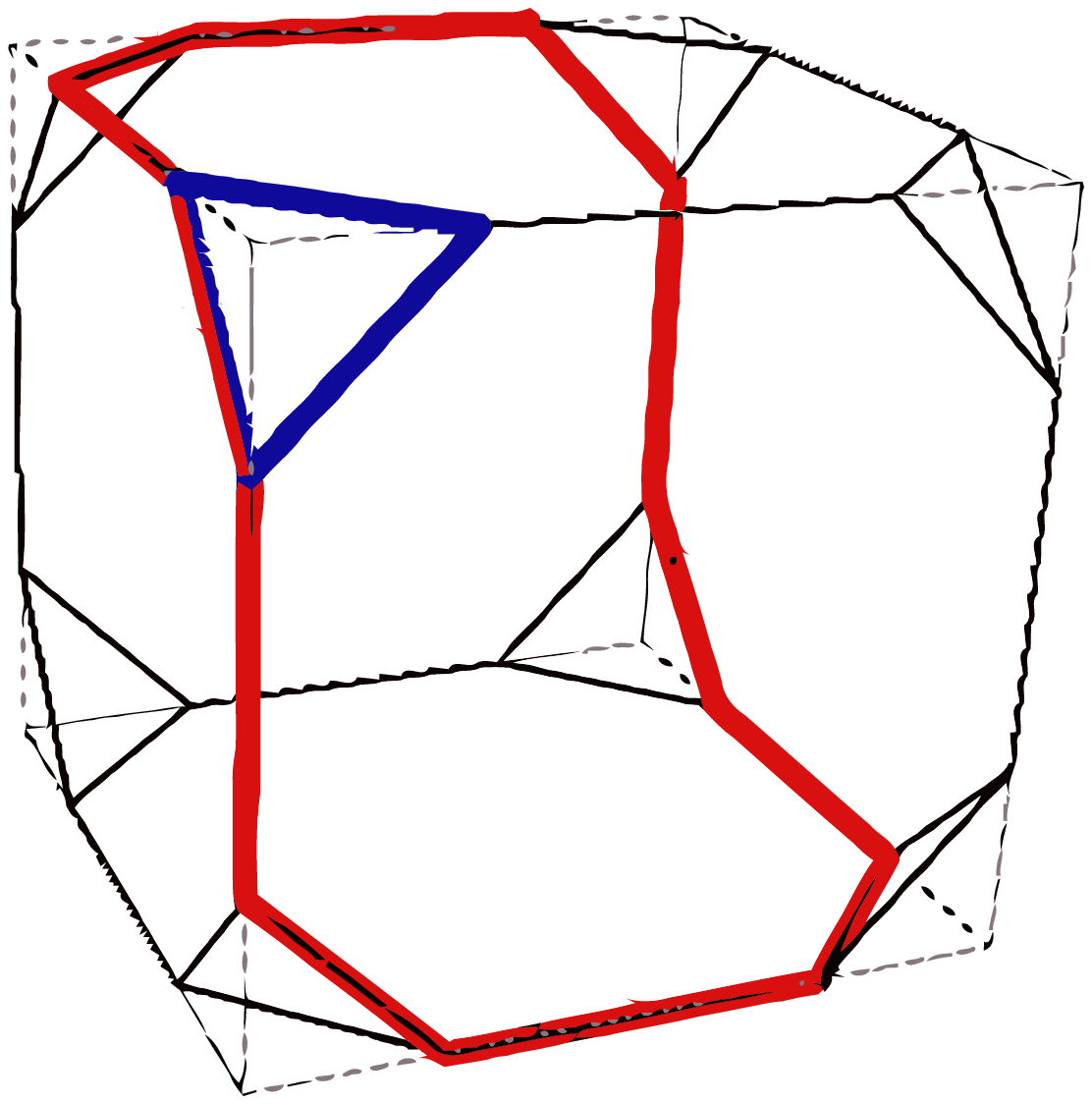}}& \raisebox{-\totalheight}{ \includegraphics*[clip, scale = .2]{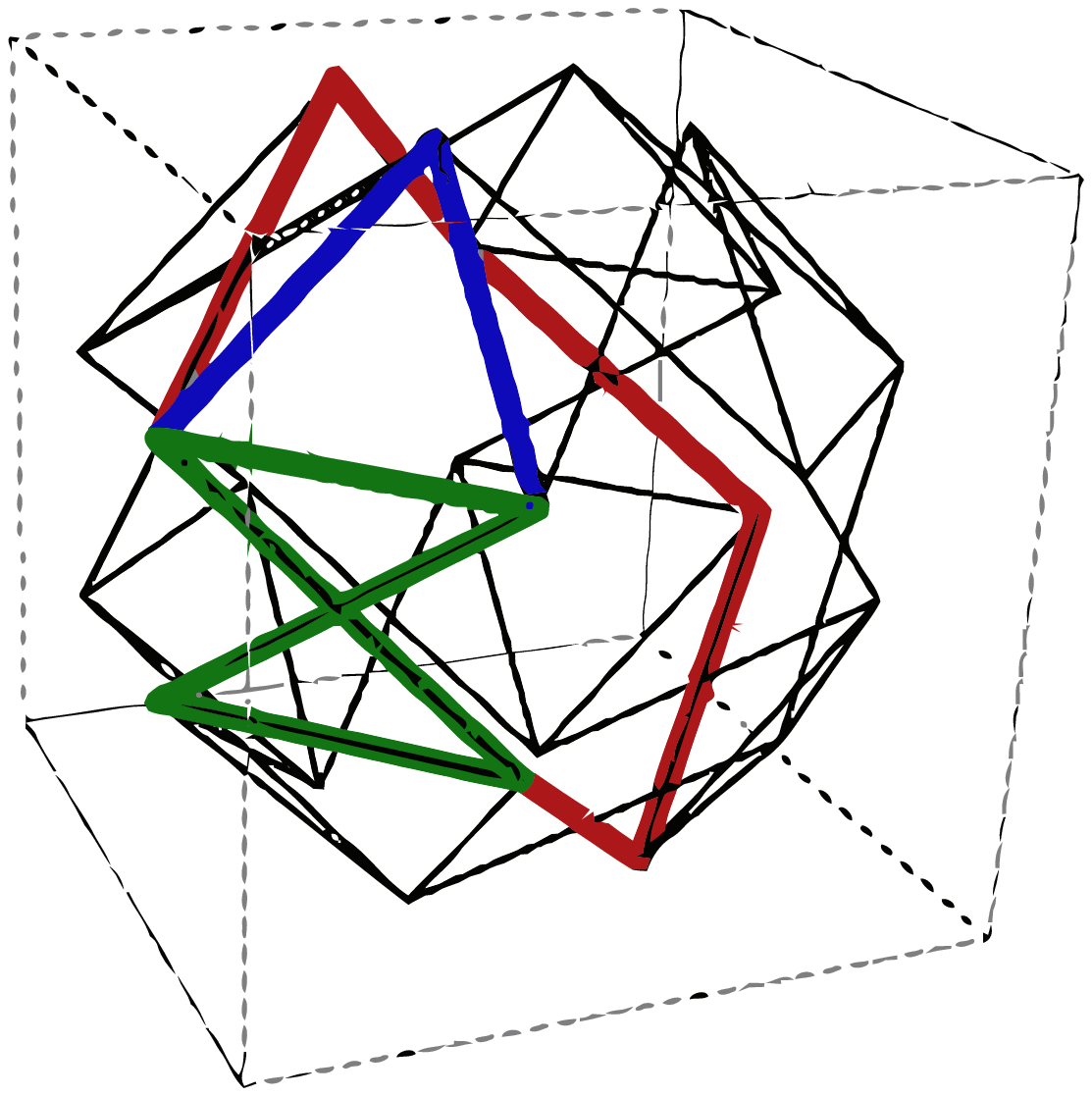}}&\raisebox{-\totalheight}{ \includegraphics*[clip, scale = .2]{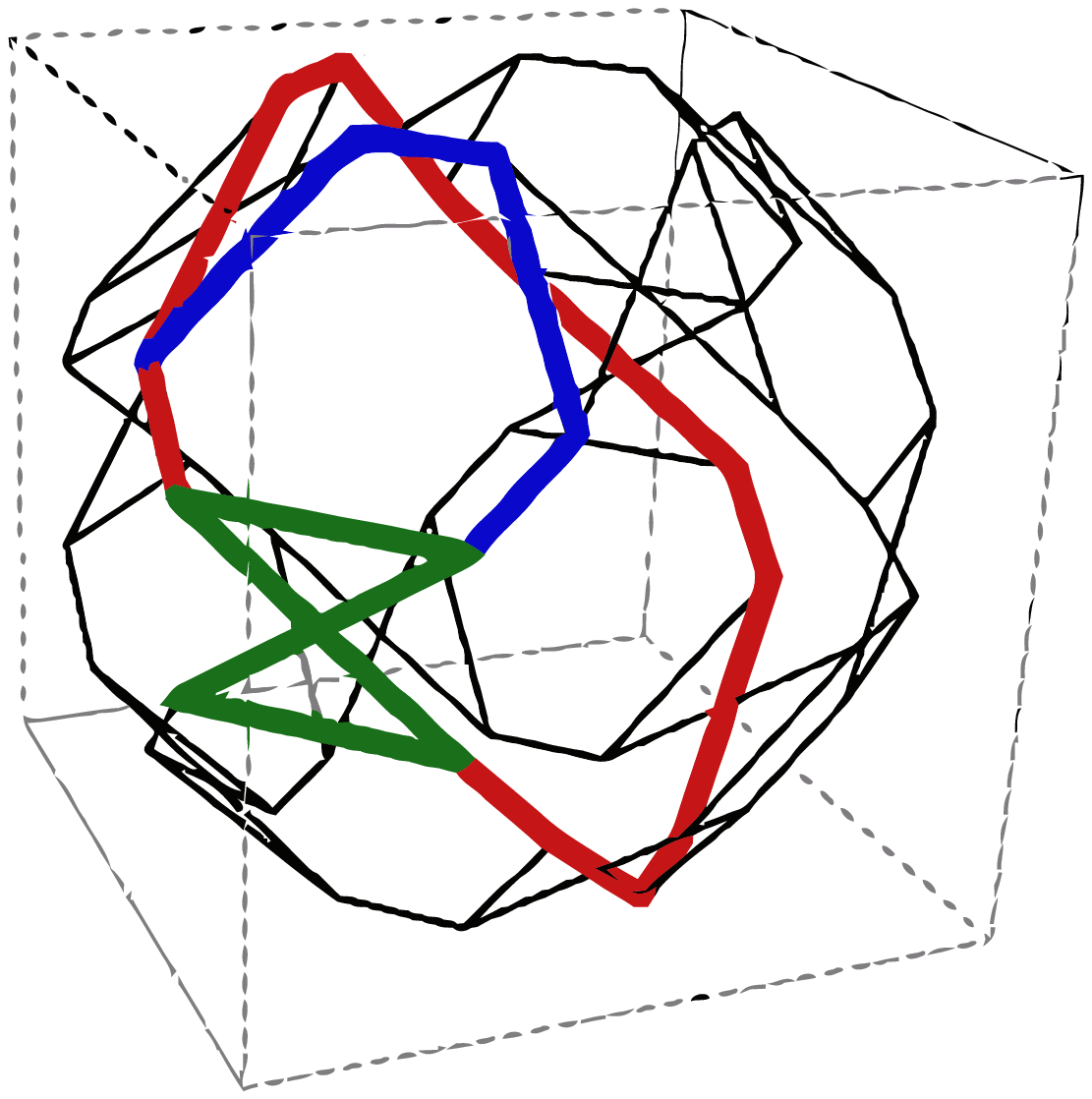}}\\
&&&&\\
$P^0$&$P^1$&$P^{01}$&$P^{02}$&$P^{012}$ 
\end{tabular}\caption{The Wythoffians derived from $\{6,3\}_4$.}\label{cube-petrie}\end{center}\end{figure}

\subsection{Planar polyhedra derived from the square tiling}

The square tiling of the plane is the (self-dual) regular geometric apeirohedron $\{4,4\}$, with symmetry group $G(\{4,4\})=\langle s_0,s_1,s_2\rangle$. The second regular apeirohedron we investigate is its Petrie dual, $\{\infty,4\}_4$. All initial vertices for Wythoffians of these two polyhedra are chosen from the fundamental triangle of $\{4,4\}$. Pictures of the Wythoffians are in Figures~\ref{4,4} and \ref{infty,4_4}, with base faces indicated in color. The Wythoffians for $\{4,4\}$ are well-known but those for $\{\infty,4\}_4$ certainly have not received much attention (however, see \cite[Sect. 12.3]{GrunShep}).

Beginning with the Wythoffians of $P=\{4,4\}$ we first note that $P^0$ is the regular apeirohedron $\{4,4\}$ itself. All faces are convex squares of type $F_2^{\{0,1\}}$. Four squares meet at each vertex, giving a vertex symbol $(4_c^4)$ and a convex square vertex figure. By the self-duality of $\{4,4\}$ this is also the Wythoffian $P^2$ (which is $P^0$ for the dual of $\{4,4\}$).

In $P^1$ the apeirohedron has two types of face: convex squares of type $F_2^{\{0,1\}}$ and congruent convex squares of type $F_2^{\{1,2\}}$. The vertex figures are convex squares since the vertex symbol is $(4_c^4)$. This is again a regular apeirohedron, a similar copy of the original square tessellation. 

The apeirohedron $P^{01}$ has two distinct types of 2-faces. The first type of base face is a convex octagon of type $F_2^{\{0,1\}}$ (truncated square), and the second type is a convex square of type $F_2^{\{1,2\}}$. Two octagons and one square meet at each vertex yielding an isosceles triangle for a vertex figure with vertex symbol $(4_c.8_c^2)$. The initial vertex can be chosen so that the octagons are regular in which case the Wythoffian is a uniform apeirohedron, the Archimedean tessellation $(4.8.8)$. Again by the self-duality of $\{4,4\}$ this is also the Wythoffian $P^{1,2}$ (which is $P^{0,1}$ for the dual of $\{4,4\}$). 

The apeirohedron $P^{02}$ has three different types of 2-faces. The first is a convex square face of type $F_2^{\{0,1\}}$, the second is a convex square of type $F_2^{\{1,2\}}$, and the final type of face is a convex rectangle of type $F_2^{\{0,2\}}$. At each vertex there is a square of the first kind, a rectangle, a square of the second kind, and a rectangle, giving a vertex symbol $(4_{c}.4_{c}.4_{c}.4_{c})$. The resulting vertex figure is convex quadrilateral. When the initial vertex is chosen so that the base edges have the same length, the rectangles are squares and the Wythoffian is a congruent copy of the original tessellation. 

For $P^{012}$, the apeirohedron has two different octagonal faces and a rectangular face. The first type of convex octagons are of type $F_2^{\{0,1\}}$ (truncated squares), the second type of convex octagons are of type $F_2^{\{1,2\}}$ (truncated squares), and the convex rectangles are of type $F_2^{\{0,2\}}$. One octagon of each type and a rectangle come together at each vertex to make a triangular vertex figure with vertex symbol $(4_c.8_c^2)$. When the initial vertex is chosen to make the base edges have equal length then the faces are regular polygons and the Wythoffian is again the Archimedean tessellation $(4.8.8)$ with an isosceles triangle as the vertex figure.

\begin{figure}[h]
\begin{center}
\begin{tabular}{cccc}
\raisebox{-\totalheight}{ \includegraphics*[clip, scale = .285]{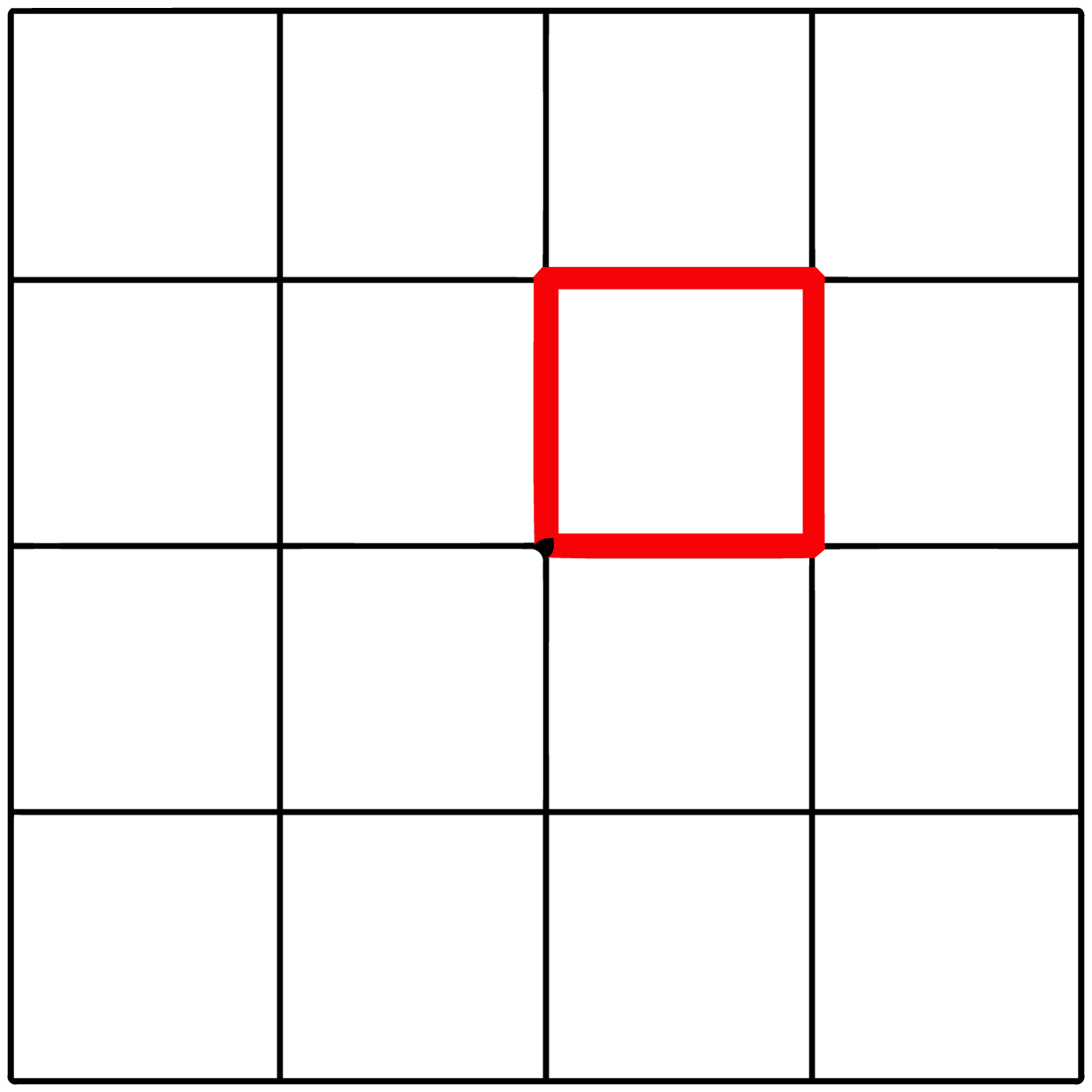}}& \raisebox{-\totalheight}{ \includegraphics*[clip, scale = .285]{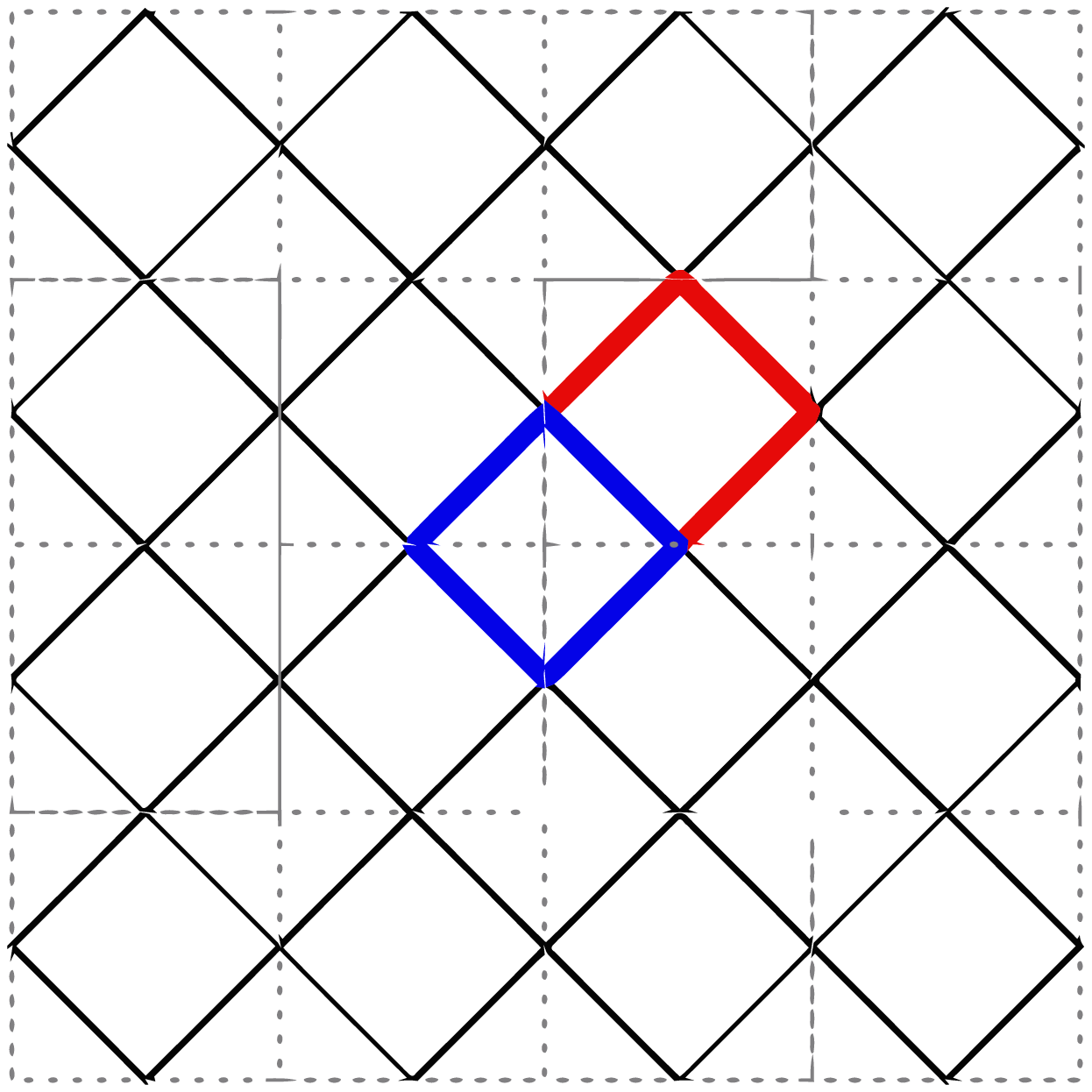}}& \raisebox{-\totalheight}{ \includegraphics*[clip, scale = .285]{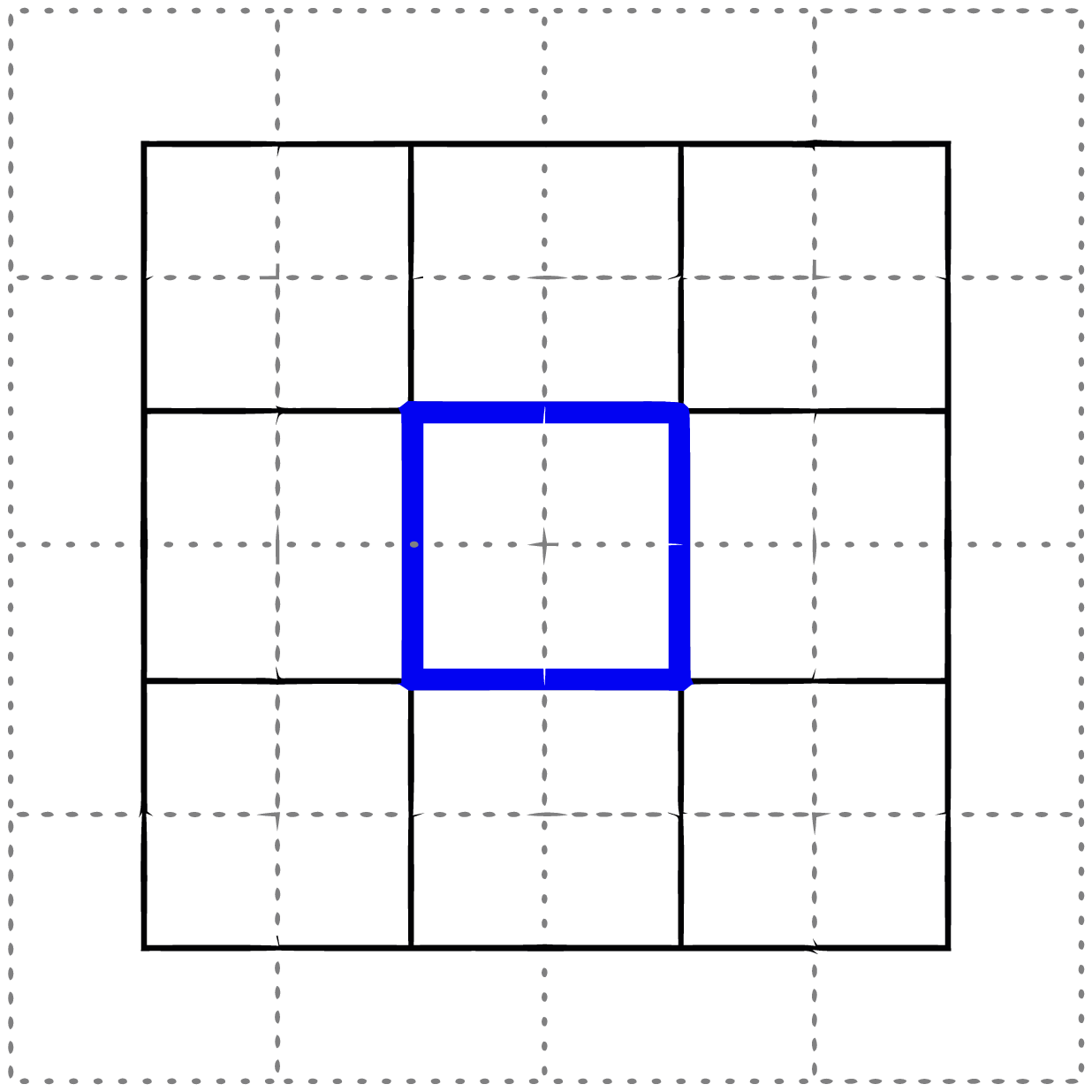}}&\raisebox{-\totalheight}{ \includegraphics*[clip, scale = .285]{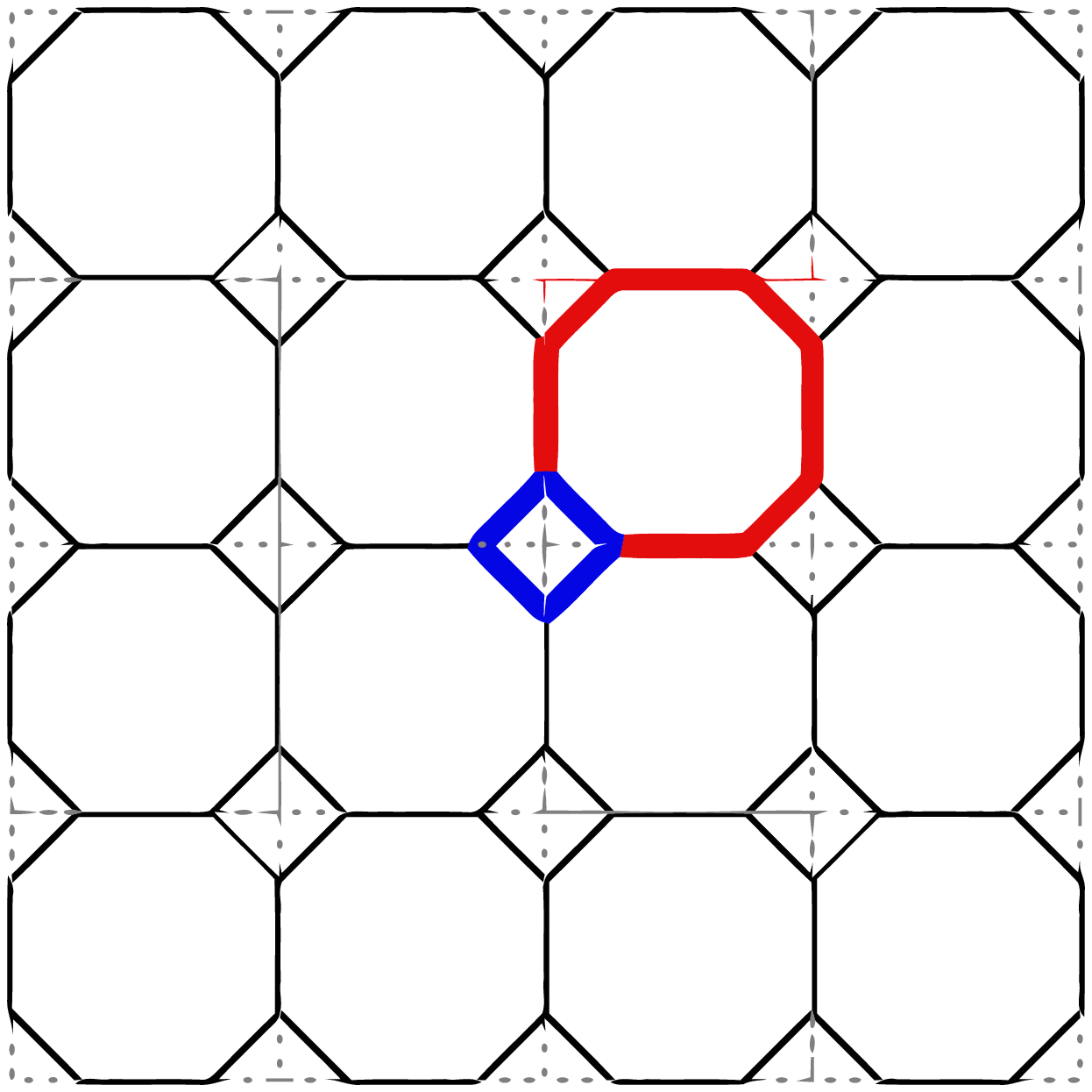}} \\
&&&\\
$P^{0}$ & $P^{1}$ & $P^{2}$&$P^{01}$\\
&&&\\
\raisebox{-\totalheight}{ \includegraphics*[clip, scale = .285]{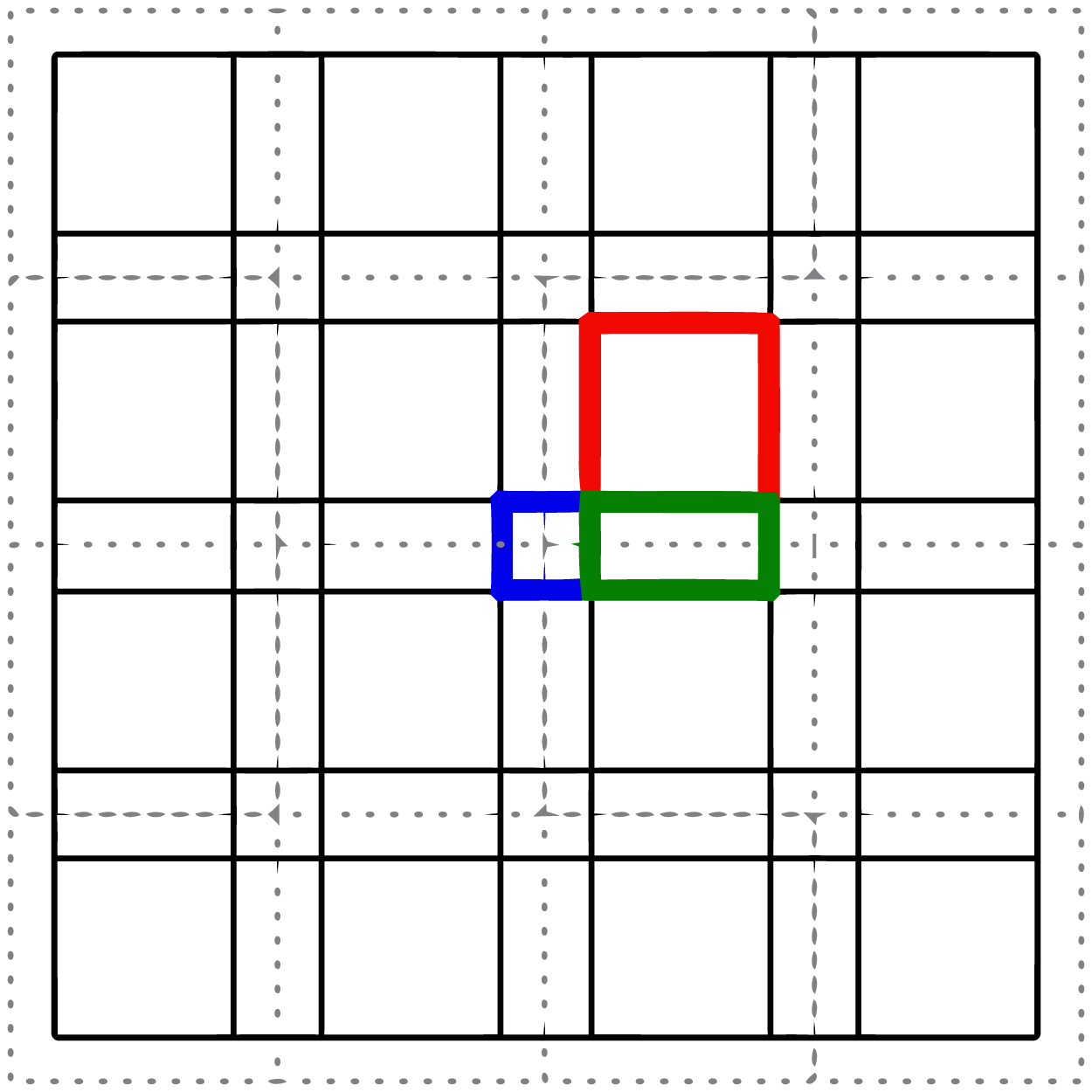}}& \raisebox{-\totalheight}{ \includegraphics*[clip, scale = .285]{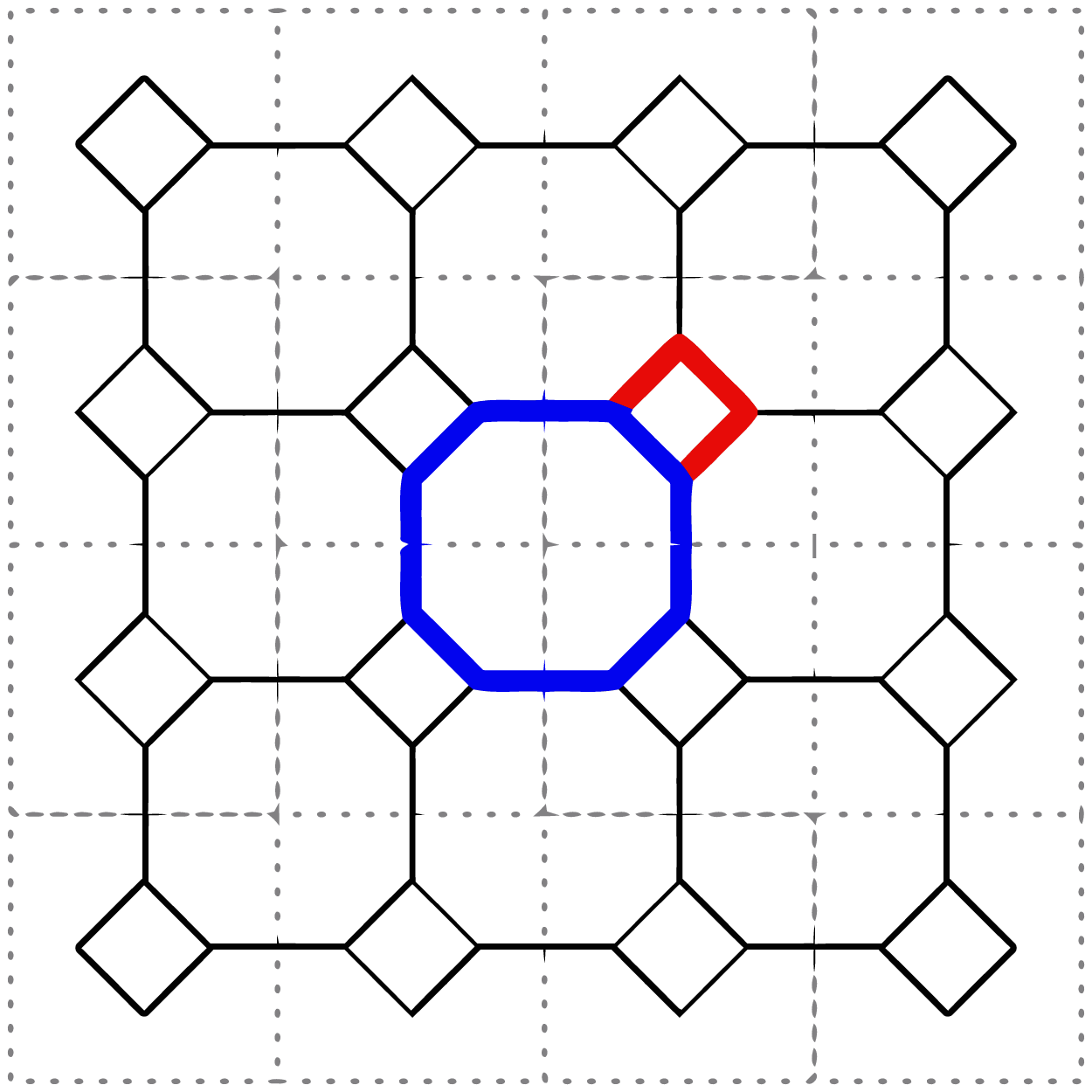}}& \raisebox{-\totalheight}{ \includegraphics*[clip, scale = .285]{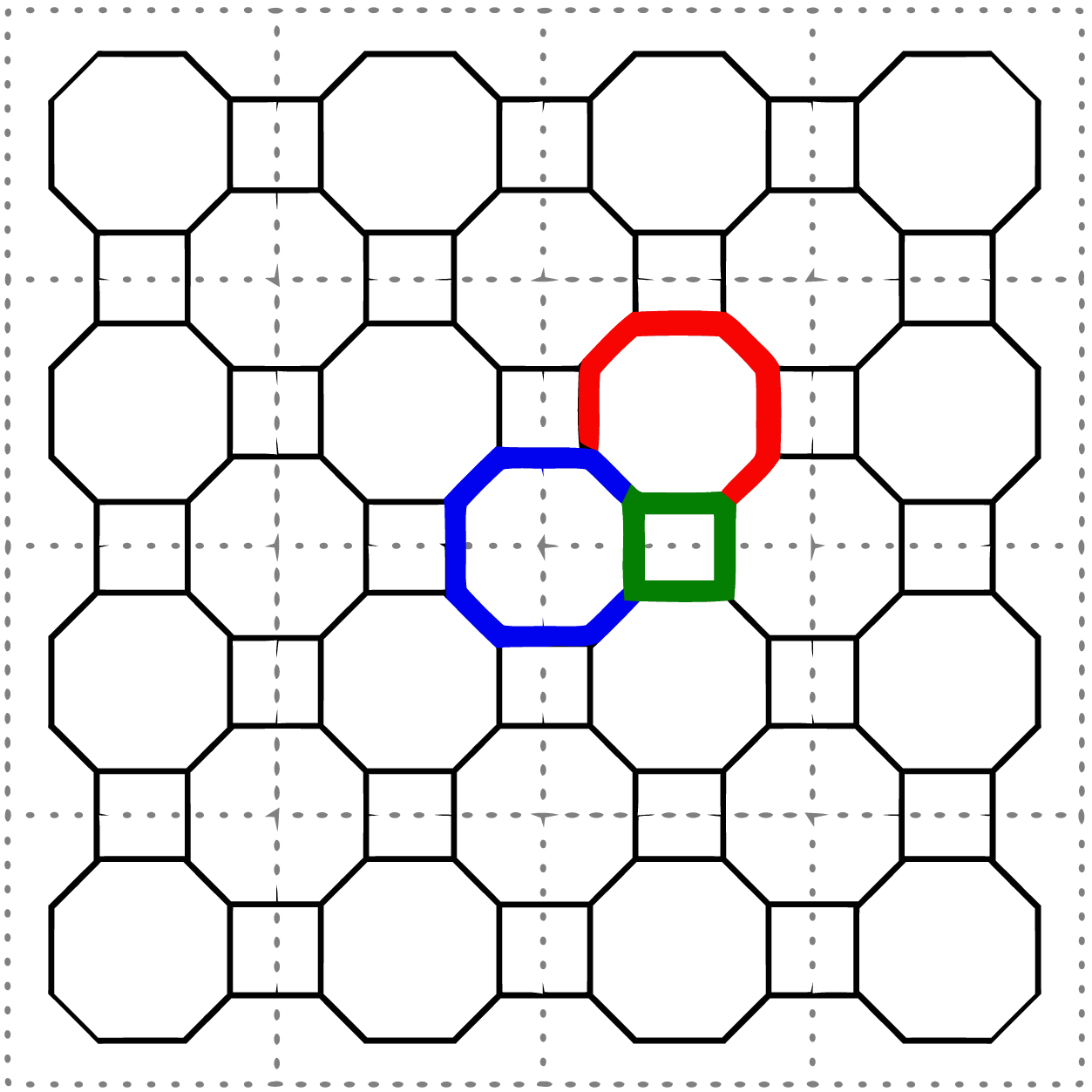}}&\\
&&&\\
 $P^{02}$& $P^{12}$ & $P^{012}$&
\end{tabular}\caption{The Wythoffians derived from $\{4,4\}$.}\label{4,4}\end{center}\end{figure}

The symmetry group of the regular apeirohedron $\{\infty,4\}_4$ is given by $G(\{\infty,4\}_4)=\langle r_0,r_1,r_2\rangle$, where $r_{0}=s_{0}s_{2}$, $r_{1}:=s_{1}$, $r_{2}:=s_{2}$ and $s_0,s_1,s_2$ are the generators of $G(\{4,4\})$. Since the center of the point reflection $r_0$ lies on the reflection line of $r_2$, every point held invariant by $r_0$ is also invariant under $r_2$ so there is no polyhedron $P^2$ or $P^{12}$ in this case. For pictures of the Wythoffians of $\{\infty,4\}_4$ see Figure~\ref{infty,4_4}.

The Wythoffian $P^0$ is the regular apeirohedron $\{\infty,4\}_4$ itself. Its 2-faces are apeirogons which appear as infinite zigzags whose consecutive edges meet at an angle of $\frac{\pi}{2}$. Four apeirogons meet at each vertex, giving a square vertex figure with vertex symbol $(\infty_2^4)$.

The apeirohedron $P^1$ only has two types of base faces. They are linear apeirogons of type $F_2^{\{0,1\}}$ and convex squares of type $F_2^{\{1,2\}}$. About each vertex there is an apeirogon, a square, an apeirogon, and a square, with vertex symbol $(\infty.4_c.\infty.4_c)$. The apeirogons dissect the plane into squares, exactly half of which are the square faces of type $F_2^{\{1,2\}}$. The vertex figure is a crossed quadrilateral. All faces of this Wythoffian are regular polygons so this shape is a uniform apeirohedron with squares and linear apeirogons as faces. 

The apeirohedron $P^{01}$ has finite and infinite faces. The apeirogonal faces are of type $F_2^{\{0,1\}}$, each of which is a truncated zigzag. The finite faces of this apeirohedron are convex squares of type $F_2^{\{1,2\}}$. The vertex symbol is $(4_c.t\infty_2.t\infty_2)$ and the resulting vertex figure is an isoceles triangle (recall that $t$ indicates truncation). The truncated zigzags are not regular apeirogons so this Wythoffian is not a uniform apeirohedron. 

The Wythoffian $P^{02}$ is an apeirohedron whose faces are regular zigzags of type $F_2^{\{0,1\}}$ where the angle between consecutive edges is greater than $\frac{\pi}{2}$, convex squares of type $F_2^{\{1,2\}}$, and crossed quadrilaterals of type $F_2^{\{0,2\}}$. The vertex figure is a convex quadrilateral with vertex symbol $(4_{\,\bowtie}.4_c.4.\infty_2)$. The crossed quadrilaterals are not regular so this is not a uniform apeirohedron. 

The final Wythoffian is $P^{012}$. There are apeirogonal faces of type $F_2^{\{0,1\}}$ which are truncated zigzags. There are also convex octagonal faces of type $F_2^{\{1,2\}}$ (truncated squares) and crossed quadrilaterals of type $F_2^{\{0,2\}}$. There is one apeirogon, one octagon, and one quadrilateral at each vertex yielding a triangular vertex figure with vertex symbol  $(4_{\,\bowtie}.8_c.t\infty_2)$.  The truncated zigzags and crossed quadrilaterals are not regular polygons so the polyhedron is not uniform.

\begin{figure}[h]
\begin{center}
\begin{tabular}{ccc}
\raisebox{-\totalheight}{ \includegraphics*[clip, scale = .3]{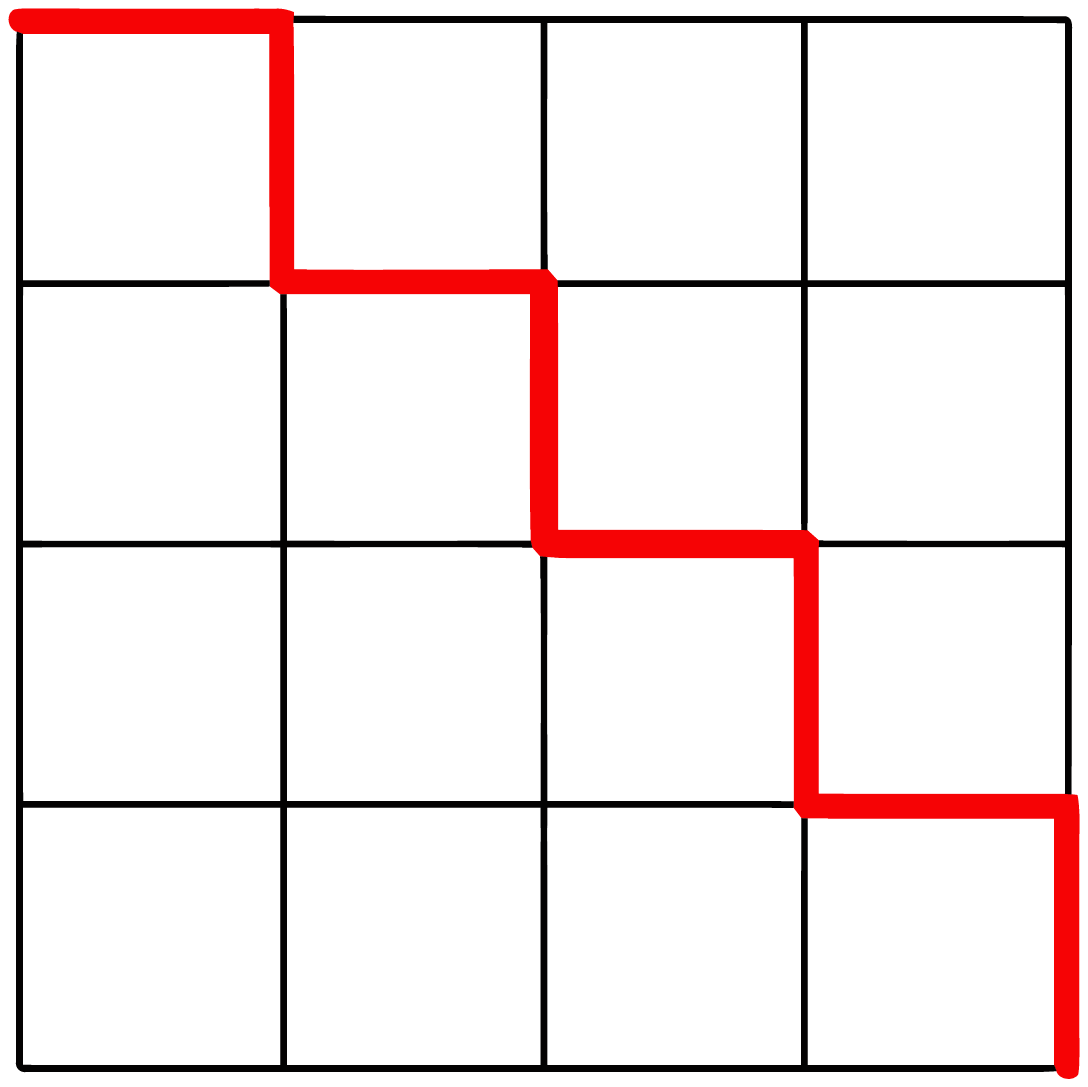}}& \raisebox{-1.1\totalheight}{ \includegraphics*[clip, scale = .3]{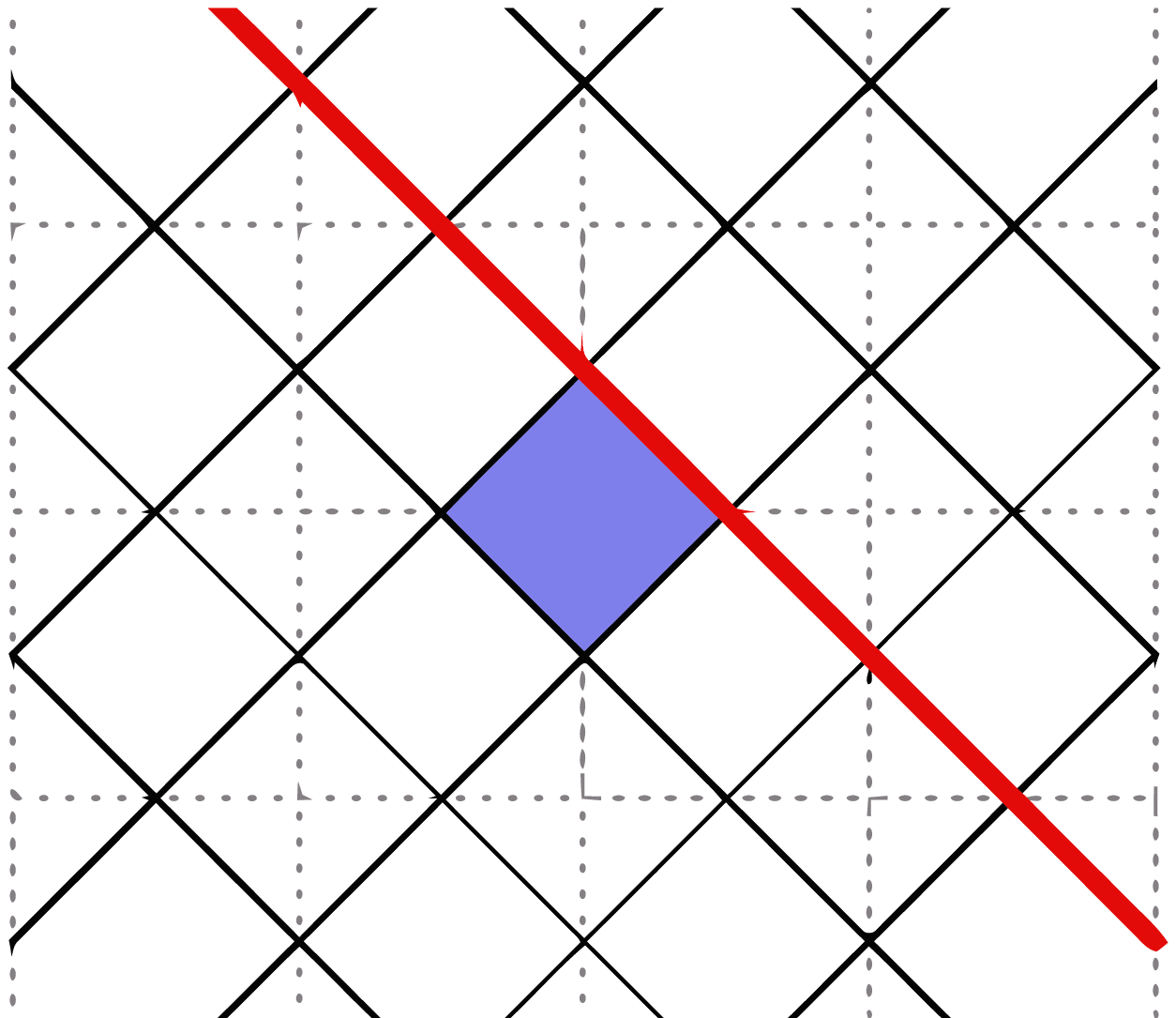}}&\raisebox{-\totalheight}{ \includegraphics*[clip, scale = .3]{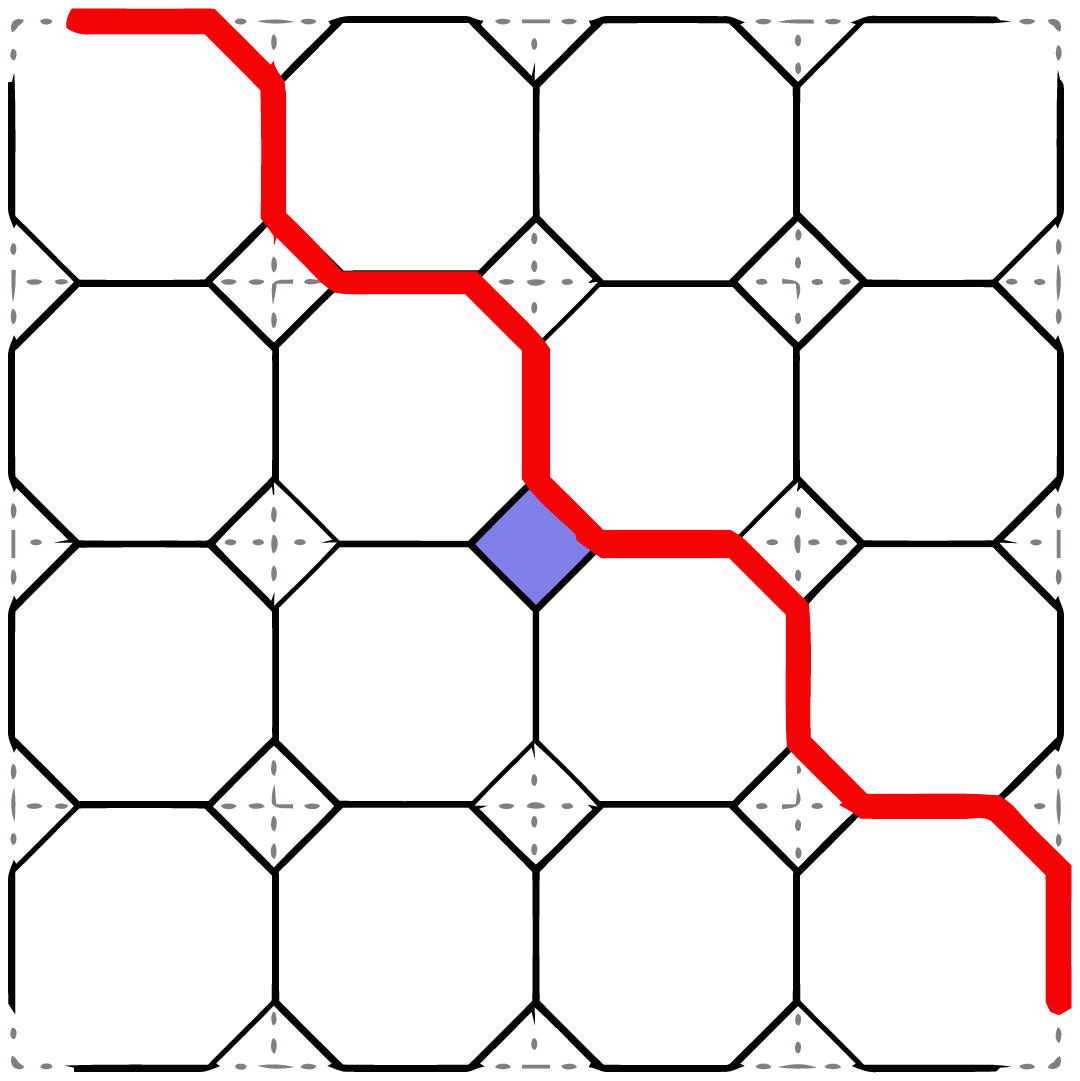}}\\
&&\\
$P^0$&$P^1$&$P^{01}$\\
 \raisebox{-\totalheight}{ \includegraphics*[clip, scale = .3]{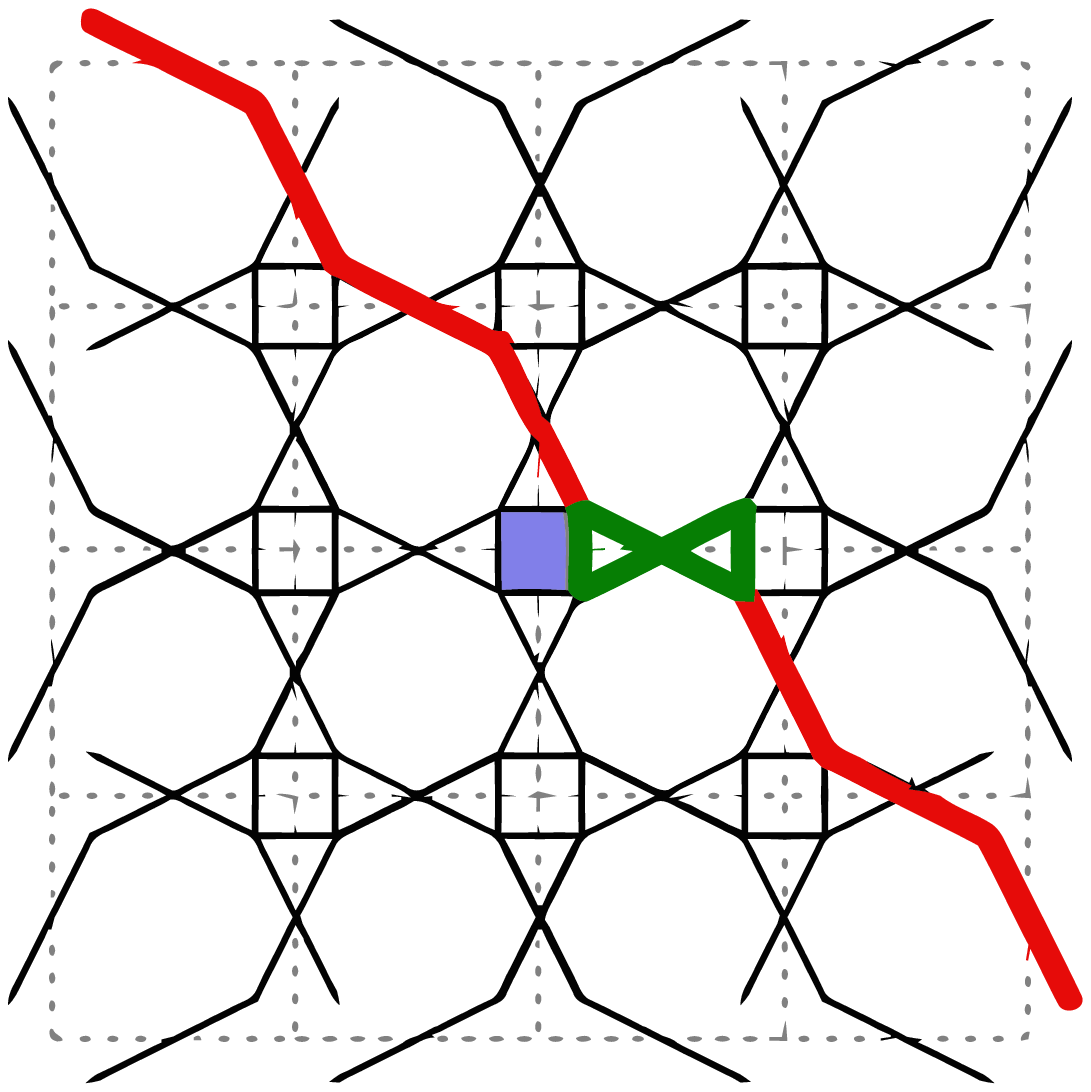}}&\raisebox{-\totalheight}{ \includegraphics*[clip, scale = .3]{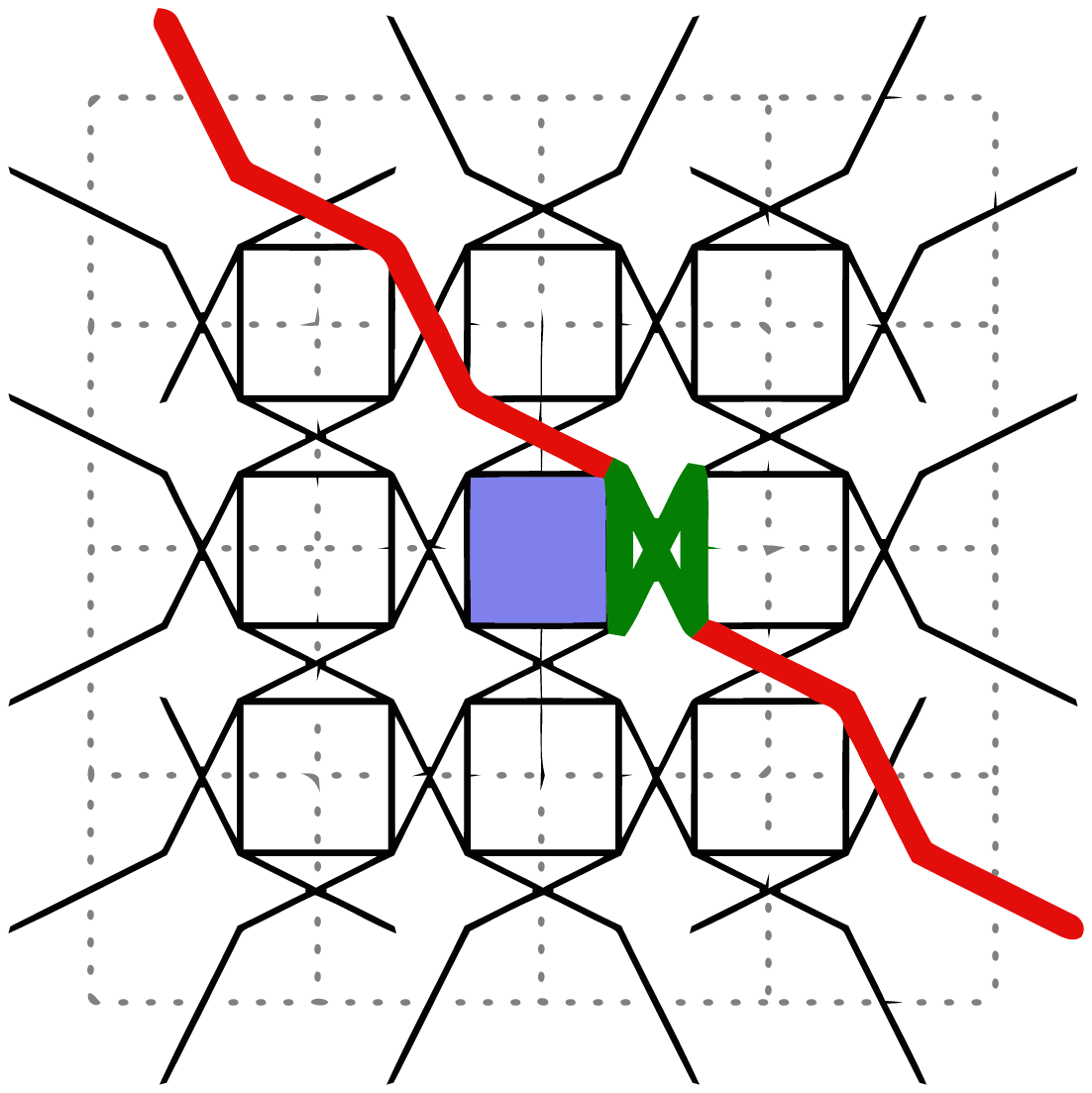}}&\\
&&\\
$P^{02}$&$P^{012}$ 
\end{tabular}\caption{The Wythoffians derived from $\{\infty,4\}_4$.}\label{infty,4_4}\end{center}\end{figure}

\subsection{Blended polyhedra derived from the square tiling}

Next we investigate the Wythoffians of the regular polyhedra $\{4,4\}\#\{\,\}$ and $\{4,4\}\#\{\infty\}$, the blends of the square tiling $\{4,4\}$ with a line segment $\{\,\}$ or linear apeirogon $\{\infty\}$, respectively, as well as their Petrie duals $\{\infty,4\}_{4}\#\{\,\}$ and $\{\infty,4\}_{4}\#\{\infty\}$. Suppose the symmetry groups of $\{4,4\}$, $\{\,\}$ and $\{\infty\}$ are given by $G(\{4,4\})=\langle s_0,s_1,s_2\rangle$, $G(\{\,\})=\langle t_0\rangle$ and $G(\{\infty\})=\langle t_0,t_1\rangle$, each with all generators viewed as plane reflections in $\mathbb{E}^3$. Note that the reflection planes for $s_0$, $s_1$, $s_2$ are perpendicular to the reflection planes for $t_0$ or $t_0,t_1$ (which are parallel to one another), respectively.

 In general the projection of the Wythoffians in this section onto the reflection plane of $t_0$  is congruent to a Wythoffian of $\{4,4\}$ or $\{\infty,4\}_4$.  In some instances if the initial vertex is chosen from the boundary of the fundamental region, the projection of the Wythoffian of the blended polyhedron will no longer appear as a Wythoffian of $\{4,4\}$ or $\{\infty,4\}_4$.  Specifically, the Wythoffians $P^{01}$, $P^{02}$, and $P^{012}$ of $\{4,4\}\#\{\,\}$ and $\{\infty,4\}_4\#\{\,\}$ will not project onto the reflection plane of $t_0$ as the Wythoffians of $\{4,4\}$ and $\{\infty,4\}_4$, respectively, if the initial vertex lies in the reflection plane of $s_0$.  For $\{4,4\}\#\{\infty\}$ and $\{\infty,4\}_4\#\{\infty\}$, if the initial vertex lies in the reflection plane of $s_0$ then $P^{01}$, $P^{02}$, and $P^{012}$ will not project onto Wythoffians of $\{4,4\}$ and $\{\infty,4\}_4$, respectively.  Similarly, for these two blends, if the initial vertex lies in the reflection plane of $s_1$ then $P^{01}$, $P^{12}$, and $P^{012}$ will not project onto Wythoffians of $\{4,4\}$ and $\{\infty,4\}_4$, respectively.  In all other cases discussed below the Wythoffians project onto Wythoffians of $\{4,4\}$ or $\{\infty,4\}_4$. 

The first apeirohedron we examine is $\{4,4\}\#\{\,\}$, which is isomorphic to $\{4,4\}$ and combinatorially self-dual. 
Its symmetry group is $G(\{4,4\}\#\{\,\})=\langle r_0,r_1,r_2\rangle$ with $r_{0}:=s_{0}t_{0}$, $r_{1}:=s_1$ and $r_{2}:=s_2$. Here, the generator $r_0$ is a half-turn and the generators $r_1$ and $r_2$ are plane reflections. Note that a generic apeirohedron $\{4,4\}\#\{\,\}$ is not geometrically self-dual; in fact, reversing the order of the generators of the group and running Wythoff's construction does not generally produce an apeirohedron similar to the original one.

Some care will have to be taken in our choice of initial vertex to ensure an interesting Wythoffian. If a point, $v$, is invariant under $t_0$ then the Wythoffian of $\{4,4\}\#\{\,\}$ with initial vertex $v$ is the same as the (planar) Wythoffian of $\{4,4\}$ with initial vertex $v$. For the following Wythoffians assume that none of the initial vertex choices are invariant under~$t_0$, and consequently we will not look at any initial vertices which are invariant under $r_0$. This excludes $P^{1}$, $P^{2}$, and $P^{12}$ as geometric Wythoffians. (Note, however, that by the combinatorial self-duality of $\{4,4\}\#\{\,\}$ there are abstract Wythoffians of these types isomorphic to $P^{1}$, $P^{0}$, and $P^{01}$, respectively.) All initial vertices are chosen from the fundamental region corresponding to $\{4,4\}\#\{\,\}$ which is a one-sided infinite cylinder over a triangle formed as the union of a pair of $0$-adjacent triangles in the barycentric subdivision of $\{4,4\}$. For pictures of the Wythoffians, see Figure \ref{4,4 line}.

The first Wythoffian, $P^0$, is $\{4,4\}\#\{\,\}$ itself. Its 2-faces are all skew squares, $\{4\}\#\{\,\}$, of type $F_2^{\{0,1\}}$. Four faces meet at each vertex, yielding a vertex symbol $(4_s^4)$ and a convex square as the vertex figure. The projection of this Wythoffian, that is, of $\{4,4\}\#\{\,\}$, onto the reflection plane of $t_0$ appears as $\{4,4\}$. 

In the next apeirohedron, $P^{01}$, the faces of type $F_2^{\{0,1\}}$ are skew octagons (truncated skew squares) and the faces of type $F_2^{\{1,2\}}$ are convex squares. Two octagons and one convex square meet at each vertex giving an isosceles triangle as a vertex figure with vertex symbol $(4.8_s^2)$. The truncated skew  squares are not regular so this is not a uniform apeirohedron. The projection of this Wythoffian onto the reflection plane of $t_0$ appears as the Wythoffian $P^{01}$ of $\{4,4\}$. 

In the apeirohedron $P^{02}$, the faces of type $F_2^{\{0,1\}}$ are skew squares, the faces of type $F_2^{\{1,2\}}$ are convex squares, and the faces of type $F_2^{\{0,2\}}$ are convex rectangles. Cyclically, about each vertex, there is a skew square, a rectangle, a square, and a rectangle, giving the vertex symbol $(4_s.4_c.4_c.4_c)$. The resulting vertex figure is a convex quadrilateral. For a specifically chosen initial vertex the faces of type $F_2^{\{0,2\}}$ are squares and the Wythoffian is a uniform apeirohedron with one kind of planar square and one kind of non-planar square. The projection of this Wythoffian onto the reflection plane of $t_0$ appears as the Wythoffian $P^{02}$ of $\{4,4\}$. 

For the Wythoffian $P^{012}$ the faces of type $F_2^{\{0,1\}}$ are skew octagons (truncated skew squares), the faces of type $F_2^{\{1,2\}}$ are convex octagons (truncated squares), and the faces of type $F_2^{\{0,2\}}$ are convex rectangles. At each vertex there is one face of each type, yielding a vertex symbol $(4_c.8_s.8_c)$ and a triangular vertex figure. The truncated  squares are not regular so the Wythoffian is not a uniform apeirohedron. The projection of this Wythoffian onto the reflection plane of $t_0$ appears as the Wythoffian $P^{012}$ of $\{4,4\}$.

\begin{figure}[h]
\begin{center}
\begin{tabular}{cccc}
\raisebox{-\totalheight}{ \includegraphics*[clip, scale = .29]{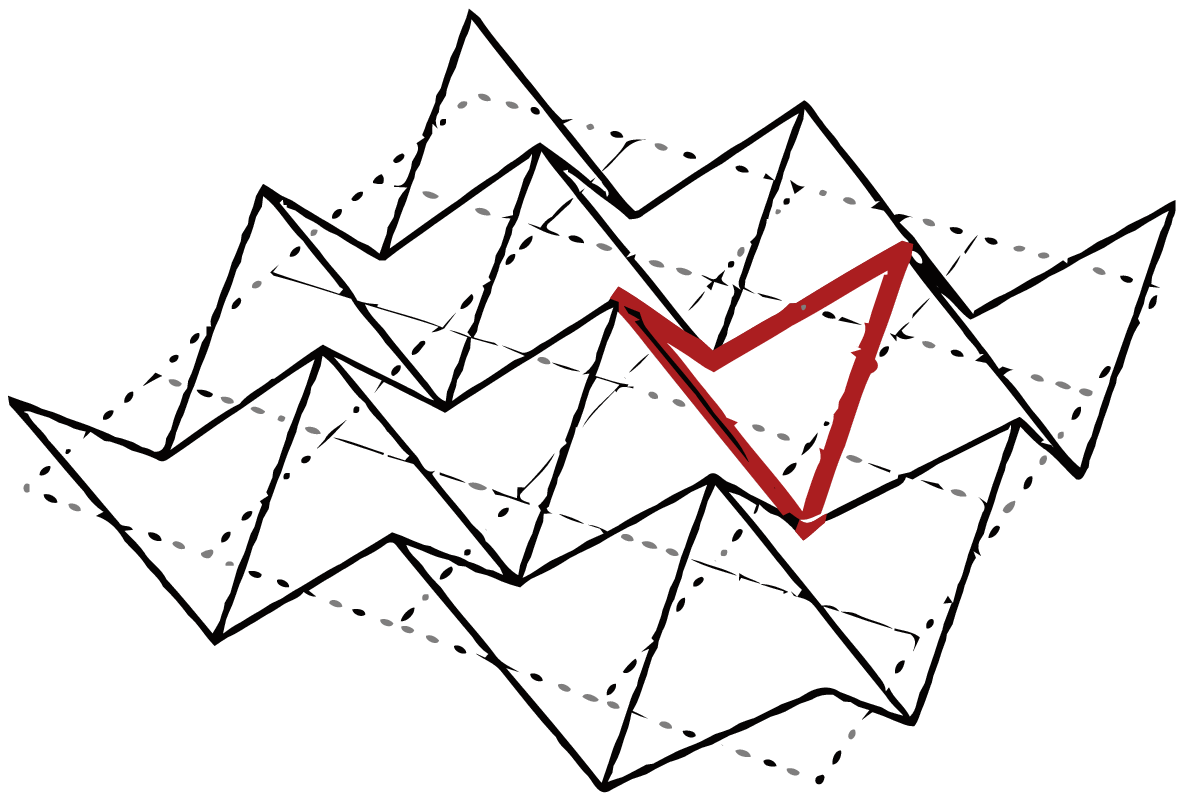}}&\raisebox{-\totalheight}{ \includegraphics*[clip, scale = .29]{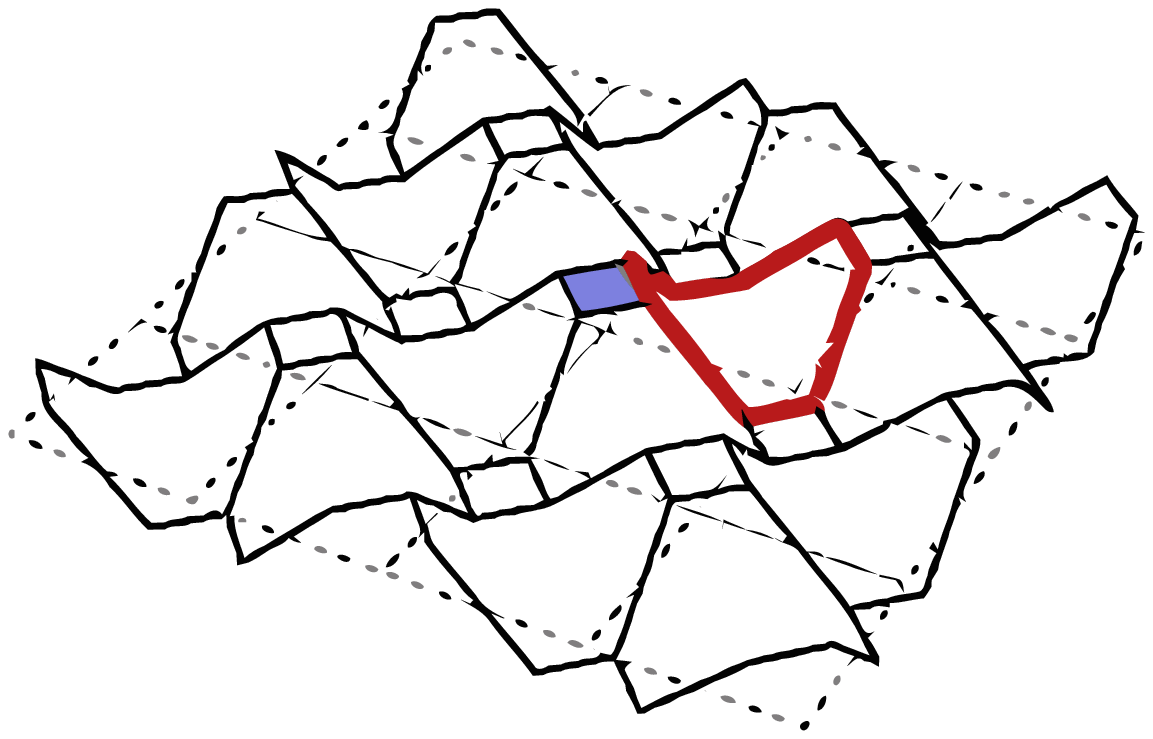}} &\raisebox{-\totalheight}{ \includegraphics*[clip, scale = .29]{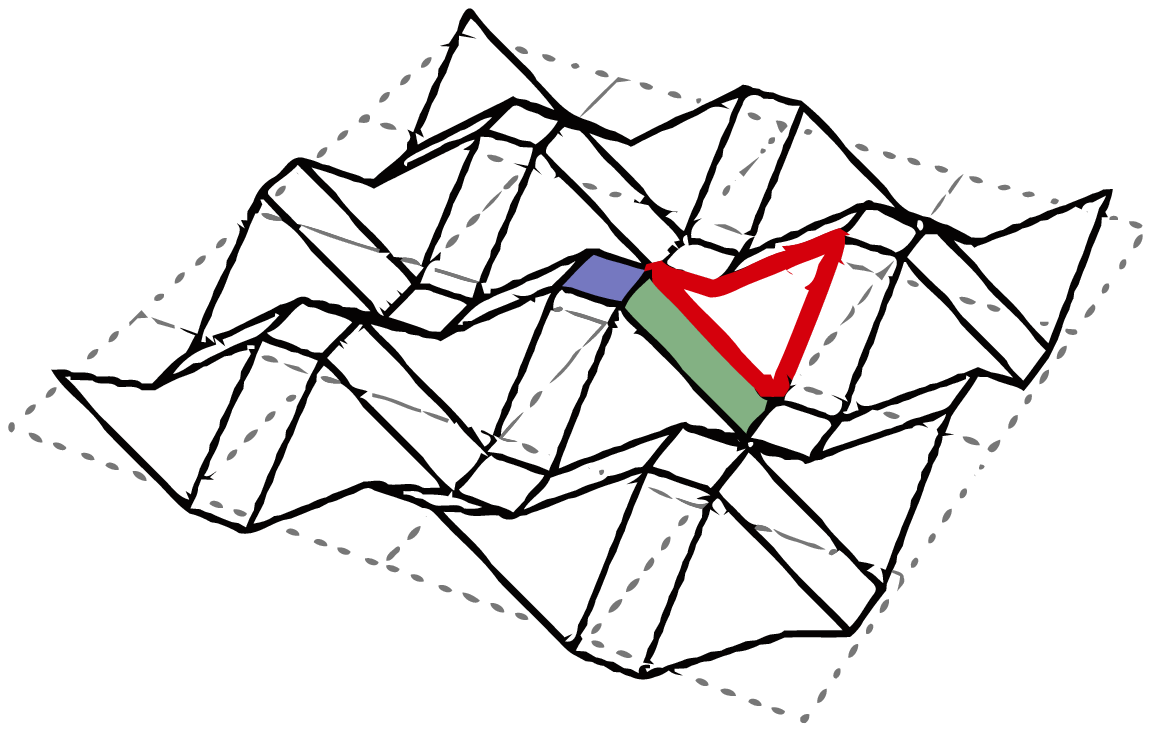}}& \raisebox{-\totalheight}{ \includegraphics*[clip, scale = .29]{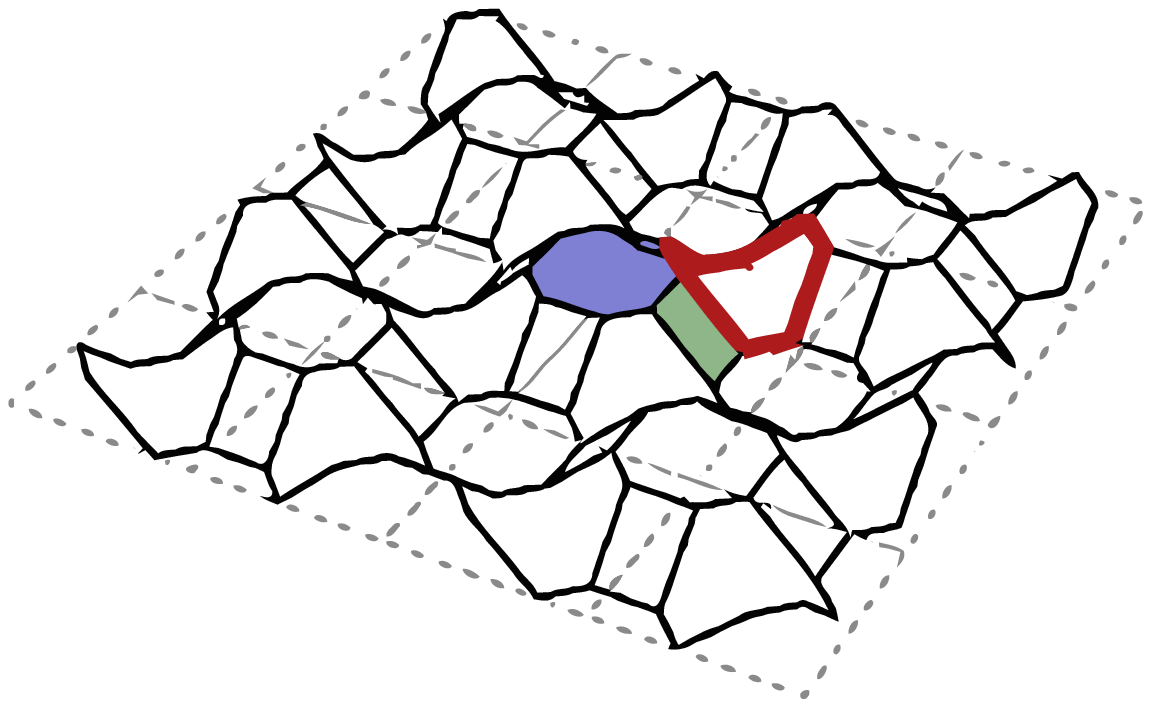}}\\
$P^0$ & $P^{01}$& $P^{02}$& $P^{012}$
\end{tabular}\caption{The Wythoffians derived from $\{4,4\}\#\{\,\}$.}\label{4,4 line}\end{center}\end{figure}

The next regular apeirohedron we examine is $\{\infty,4\}_4\#\{\,\}$, the Petrie-dual of $\{4,4\}\#\{\,\}$, which is isomorphic to $\{\infty,4\}_4$. The symmetry group is $G(\{\infty,4\}_4\#\{\,\})=\langle r_0,r_1,r_2\rangle$ with $r_{0}:=s_{0}t_{0}s_{2}$, $r_{1}:=s_1$, $r_{2}:=s_2$, and $s_0,s_1,s_2,t_0$ as above. Here $r_0$ is a point reflection (through the midpoint of the base edge of the underlying plane tessellation $\{4,4\}$) and $r_1$ and $r_2$ are plane reflections. Individually $s_0,\ s_1,\ s_2$, and $t_0$ are plane reflections in $\mathbb{E}^3$. 

The initial vertices we use come from the same fundamental region as for $\{4,4\}\#\{\,\}$. As with $\{4,4\}\#\{\,\}$, any initial vertex left invariant by $t_0$ will result in the Wythoffian being the same as the corresponding Wythoffian derived from the planar $\{\infty,4\}_4$. Assume all choices of initial vertex are transient under $t_0$, and consequently we will not look at any initial vertices which are invariant under $r_0$. This excludes $P^{1}$, $P^{2}$, and $P^{12}$. For pictures of the Wythoffians, see Figure \ref{infty,4_4 line}.

The first Wythoffian, $P^0$, is the regular apeirohedron $\{\infty,4\}_4\#\{\,\}$ itself whose faces are regular zigzag apeirogons, $\{\infty\}\#\{\,\}$, such that each edge is bisected by the reflection plane of $t_0$. Four of these apeirogons meet at each vertex resulting in a convex, square vertex figure with vertex symbol $(\infty_2^4)$. The projection of this Wythoffian onto the reflection plane of $t_0$ appears as $\{\infty,4\}_4$. 

In the apeirohedron $P^{01}$, the faces of type $F_2^{\{0,1\}}$ are apeirogons which appear as truncations of the faces of $\{\infty,4\}_4\#\{\,\}$, while the faces of type $F_2^{\{1,2\}}$ are convex squares which lie parallel to the reflection plane of $t_0$. Two apeirogons and one square meet at each vertex, yielding the vertex symbol $(4.t\infty_2.t\infty_2)$. The vertex figure is an isosceles triangle. The truncated zigzags are not regular so this Wythoffian is not a uniform apeirohedron. The projection of this Wythoffian onto the reflection plane of $t_0$ appears as the Wythoffian $P^{01}$ of $\{\infty,4\}_4$. 

With the apeirohedron $P^{02}$ the faces of type $F_2^{\{0,1\}}$ are regular zigzag apeirogons which are bisected by the reflection plane of $t_0$, the faces of type $F_2^{\{1,2\}}$ are convex squares parallel to the reflection plane of $t_0$, and the faces of type $F_2^{\{0,2\}}$ are planar crossed quadrilaterals which intersect the reflection plane of $t_0$. Cyclically at each vertex there is an apeirogon, a crossed quadrilateral, a square, and a crossed quadrilateral, resulting in the vertex-symbol  $(4_{\,\bowtie}.4_c.4_{\,\bowtie}.\infty_2)$.  The vertex figure is a convex quadrilateral. The crossed quadrilaterals are not regular so this apeirohedron is not uniform. The projection of this Wythoffian onto the reflection plane of $t_0$ appears as the Wythoffian $P^{02}$ of $\{\infty,4\}_4$. 

For the final apeirohedron, $P^{012}$, the faces of type $F_2^{\{0,1\}}$ are truncated zigzag apeirogons. The faces of type $F_2^{\{1,2\}}$ are convex octagons (truncated squares) which lie parallel to the reflection plane of $t_0$, and the faces of type $F_2^{\{0,2\}}$ are crossed quadrilaterals which intersect the reflection plane of $t_0$ at their centers. There is one face of each type at each vertex, so that the vertex symbol is  $(4_{\,\bowtie}.8_c.t\infty_2)$  and the vertex figure is a triangle. The truncated zigzags and crossed quadrilaterals are not regular so this Wythoffian is not a uniform apeirohedron. The projection of this Wythoffian onto the reflection plane of $t_0$ appears as the Wythoffian $P^{012}$ of $\{\infty,4\}_4$.

\begin{figure}[h]
\begin{center}
\begin{tabular}{cccc}
\raisebox{-\totalheight}{ \includegraphics*[clip, scale = .29]{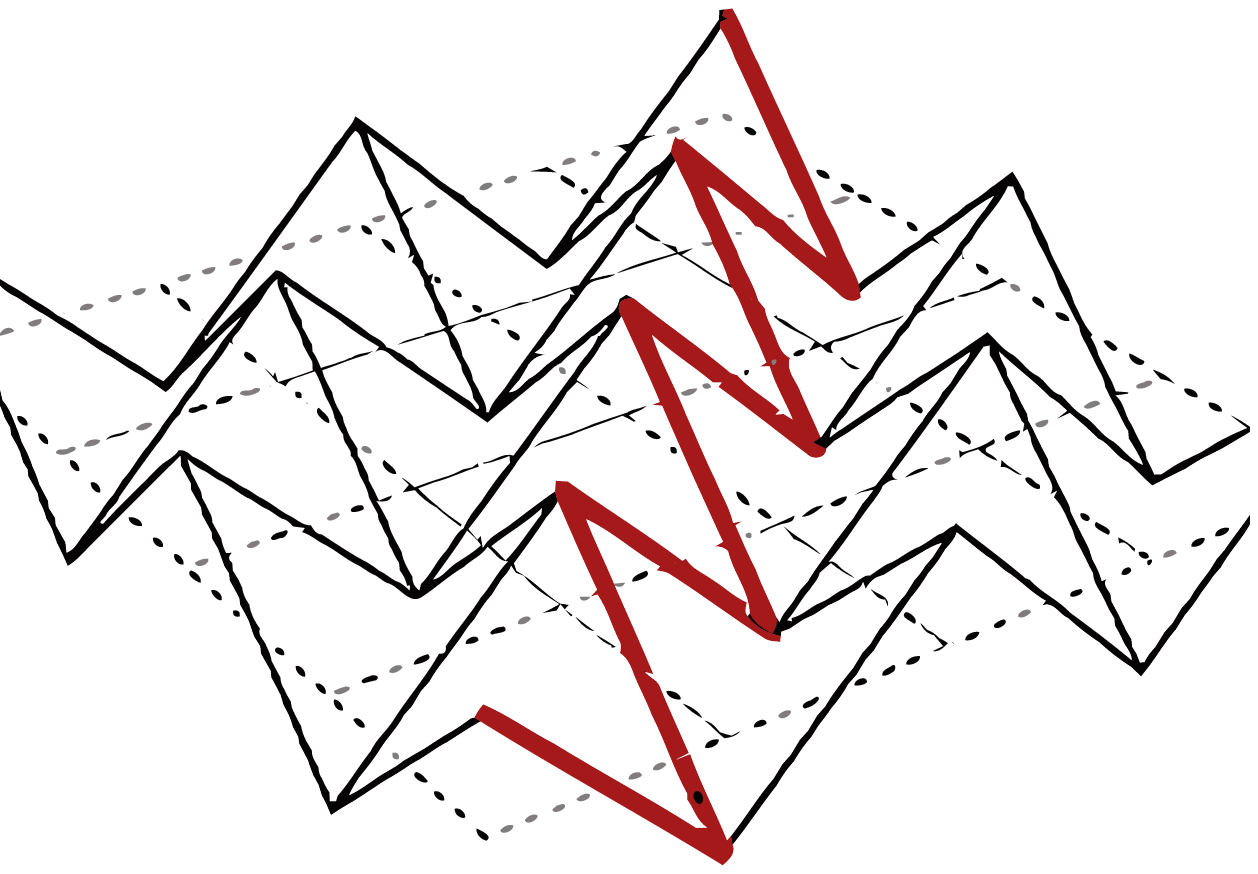}}&\raisebox{-\totalheight}{ \includegraphics*[clip, scale = .29]{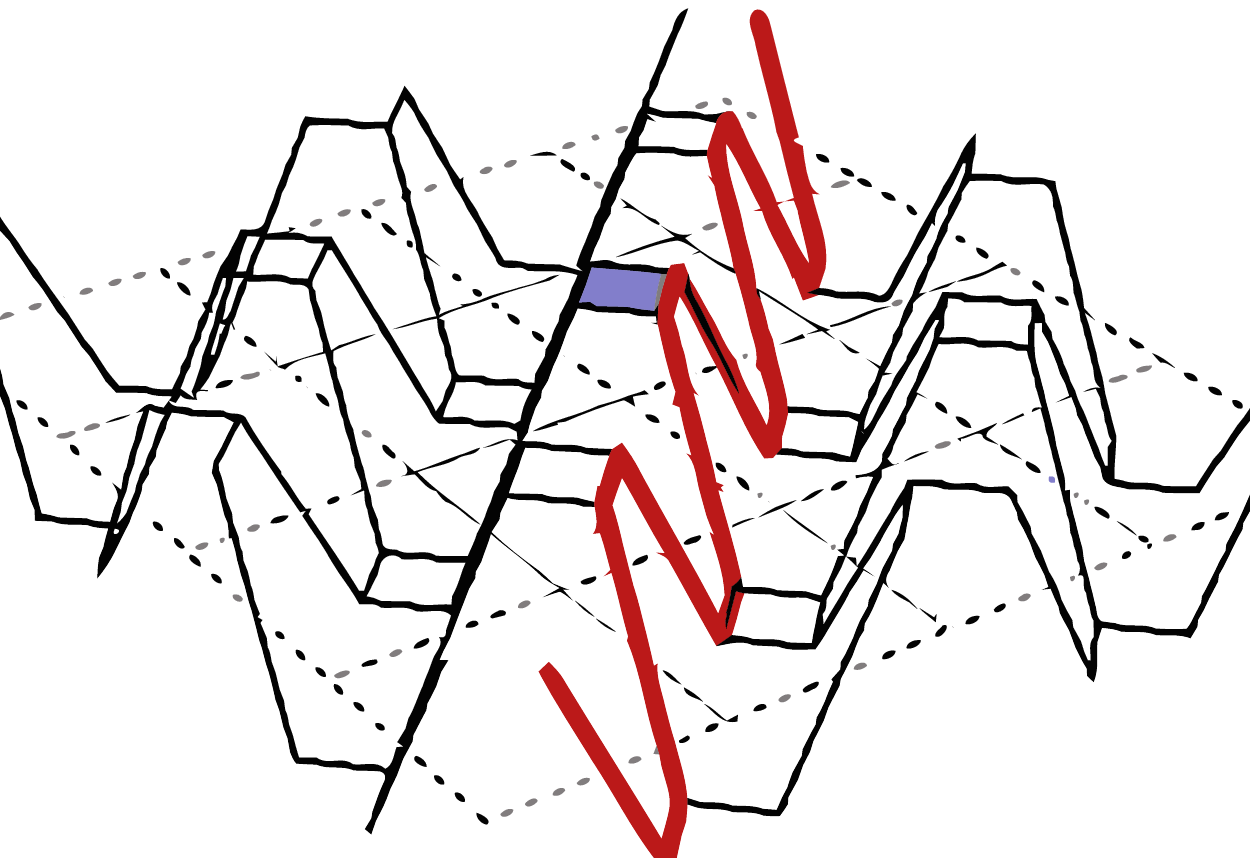}}& \raisebox{-\totalheight}{ \includegraphics*[clip, scale = .29]{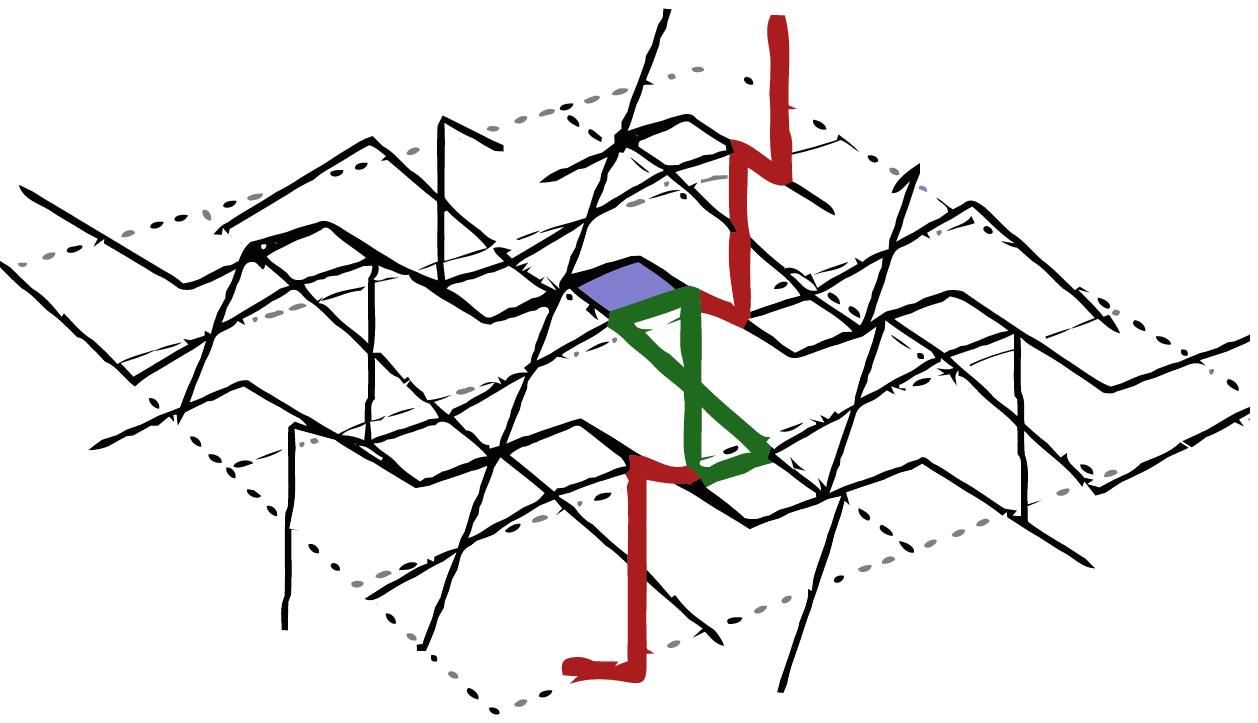}}&\raisebox{-\totalheight}{ \includegraphics*[clip, scale = .29]{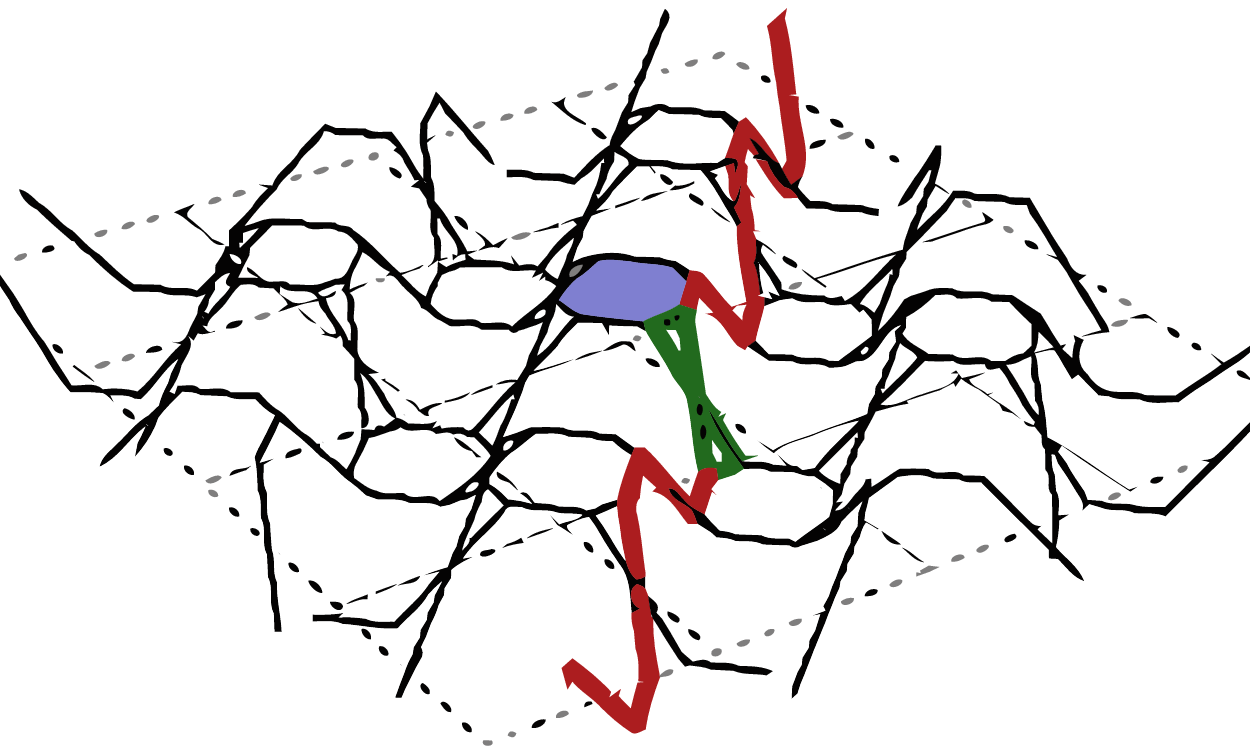}}\\
&&&\\
$P^0$&$P^{01}$&$P^{02}$&$P^{012}$ 
\end{tabular}\caption{The Wythoffians derived from $\{\infty,4\}_4\#\{\,\}$.}\label{infty,4_4 line}\end{center}\end{figure}

Now we consider the blended regular apeirohedron $\{4,4\}\#\{\infty\}$ with symmetry group $G(\{4,4\}\#\{\infty\})=\langle r_0,r_1,r_2\rangle$, where $r_{0}:=s_{0}t_{0}$, $r_{1}:=s_{1}t_{1}$, $r_{2}:=s_2$, and $s_0,s_1,s_2,t_0,t_1$  are as above. Here $r_0$ and $r_1$ are half-turns and $r_2$ is a plane reflection.

In $\mathbb{E}^3$ the reflection planes corresponding to $s_0$, $s_1$, and $s_2$ are orthogonal to the reflection planes corresponding to $t_0$ and $t_1$ which are parallel to one another. There is no point which is invariant under $t_0$ and $t_1$ so this will limit the choice of initial vertex and consequently we will not look at $P^2$. In all cases, any edge of type $F_1^{\{2\}}$ lies parallel to the reflection planes of $t_0$ and $t_1$. All initial vertices are chosen from the fundamental region of $\{4,4\}\#\{\infty\}$. This fundamental region is a right prism over the fundamental triangle of $\{4,4\}$. For pictures of the Wythoffians see Figure \ref{4,4 inf}.

The first Wythoffian, $P^0$, is the regular apeirohedron $\{4,4\}\#\{\infty\}$ itself. Its faces are helical apeirogons spiraling around a cylinder with a square base. Each edge is incident to two helices which spiral upward in opposite orientations. Four helices meet at each vertex resulting in an antiprismatic square vertex figure with vertex symbol $(\infty_4^4)$. The projection of this Wythoffian onto the reflection plane of $t_0$ appears as $\{4,4\}$. 

In the apeirohedron $P^1$ the faces of type $F_2^{\{0,1\}}$ are regular helices over square bases while the faces of type $F_2^{\{1,2\}}$ are antiprismatic squares. At each vertex, in alternating order, there are two helices and two antiprismatic squares, yielding a convex rectangle as a vertex figure with a vertex symbol $(4_s.\infty_4.4_s.\infty_4)$. The faces are all regular polygons so the Wythoffian is a uniform apeirohedron. The projection of this Wythoffian onto the reflection plane of $t_0$ appears as the Wythoffian $P^{1}$ of $\{4,4\}$. 

For $P^{01}$ the faces of type $F_2^{\{0,1\}}$ are helices over an octagon (for initial vertices which lie on a base edge of $\{4,4\}\#\{\infty\}$ these helices are truncations of the helical faces of $\{4,4\}\#\{\infty\}$). The faces of type $F_2^{\{1,2\}}$ are regular squares, for some initial vertices they are skew and for some initial vertices they are convex. There are two helices and one quadrilateral at each vertex resulting in an isosceles triangle vertex figure with vertex symbol $(4.\infty_8^2)$. For a carefully chosen initial vertex the helices are regular helices about octagonal bases and the Wythoffian is uniform. The projection of this Wythoffian onto the reflection plane of $t_0$ appears as the Wythoffian $P^{01}$ of $\{4,4\}$. 

In the apeirohedron $P^{02}$ the faces of type $F_2^{\{0,1\}}$ are regular helices over squares, the faces of type $F_2^{\{1,2\}}$ are convex squares lying parallel to the plane of $t_0$, and the faces of type $F_2^{\{0,2\}}$ are convex rectangles which are not parallel to this plane. Cyclically, about each vertex there is a square, a rectangle, a helix, and a rectangle, with vertex symbol $(4_c.4_c.\infty_4.4_c)$. The resulting vertex figure is a skew quadrilateral. For a carefully chosen initial vertex the faces of type $F_2^{\{0,2\}}$ are regular and the Wythoffian is a uniform apeirohedron. The projection of this Wythoffian onto the reflection plane of $t_0$ appears as the Wythoffian $P^{02}$ of $\{4,4\}$. 

For the Wythoffian $P^{12}$ the faces of the apeirohedron of type $F_2^{\{0,1\}}$ are regular helices over squares and the faces of type $F_2^{\{1,2\}}$ are skew octagons (truncated antiprismatic squares). There are two octagons and one helix meeting at each vertex, yielding $(8_s^2.\infty_4)$ as the vertex symbol and an isosceles triangle as the vertex figure. The truncated antiprismatic squares are not regular so $P^{12}$ is not a uniform apeirohedron. The projection of this Wythoffian onto the reflection plane of $t_0$ appears as the Wythoffian $P^{12}$ of $\{4,4\}$. 

In the apeirohedron $P^{012}$ the faces of type $F_2^{\{0,1\}}$ are apeirogons and appear as helices over octagons (truncated helices over squares). The finite faces are skew octagons of type $F_2^{\{1,2\}}$ (truncated antiprismatic squares) and convex rectangles of type $F_2^{\{0,2\}}$. One face of each type meets at each vertex yielding a triangular vertex figure with vertex symbol $(4_c.8_s.\infty_8)$. The truncated antiprismatic squares are not regular so this Wythoffian is not a uniform apeirohedron. The projection of this Wythoffian onto the reflection plane of $t_0$ appears as the Wythoffian $P^{012}$ of $\{4,4\}$.

\begin{figure}[h]
\begin{center}
\begin{tabular}{cccccc}
\raisebox{-\totalheight}{ \includegraphics*[clip, scale = .29]{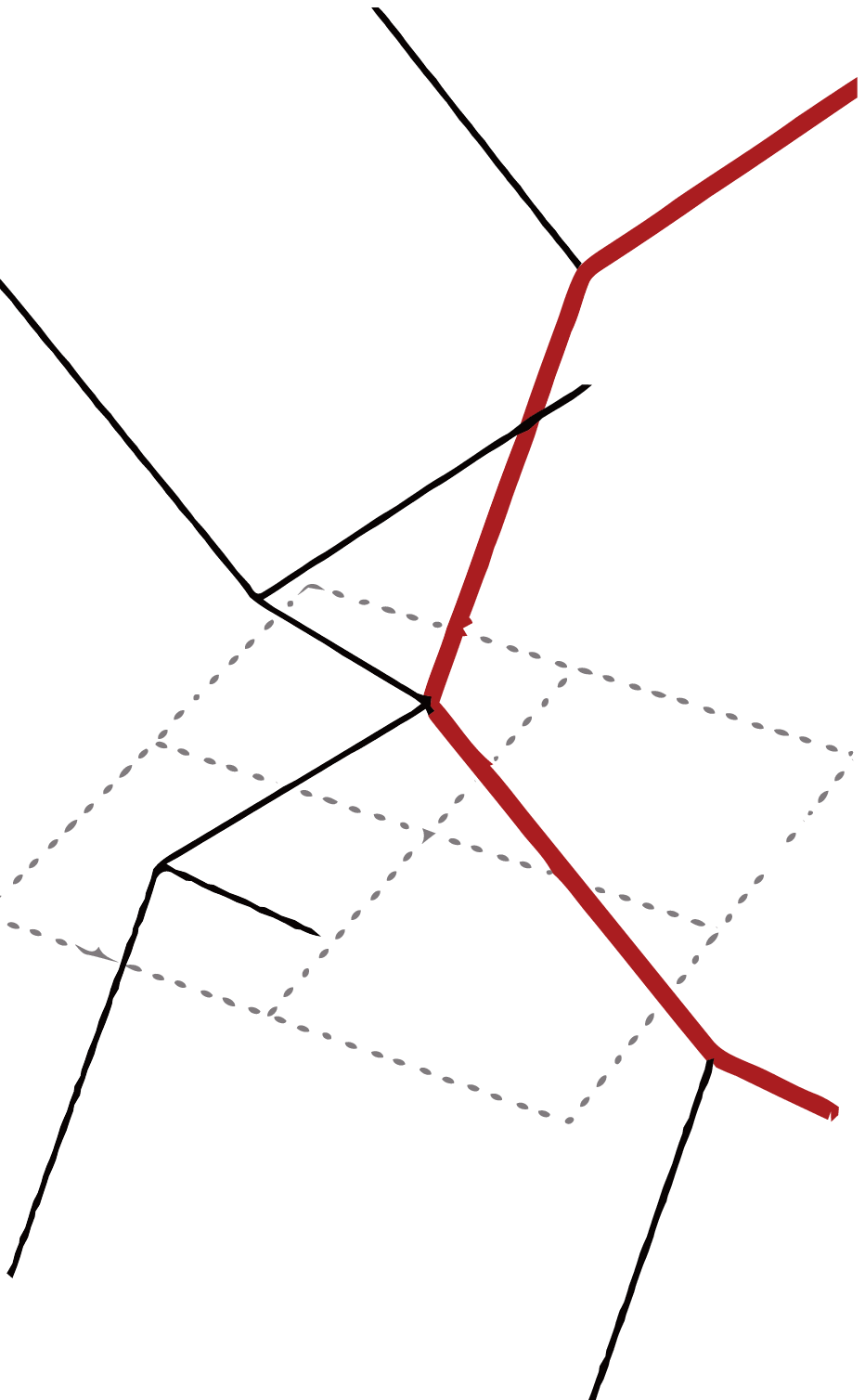}}& \raisebox{-\totalheight}{ \includegraphics*[clip, scale = .29]{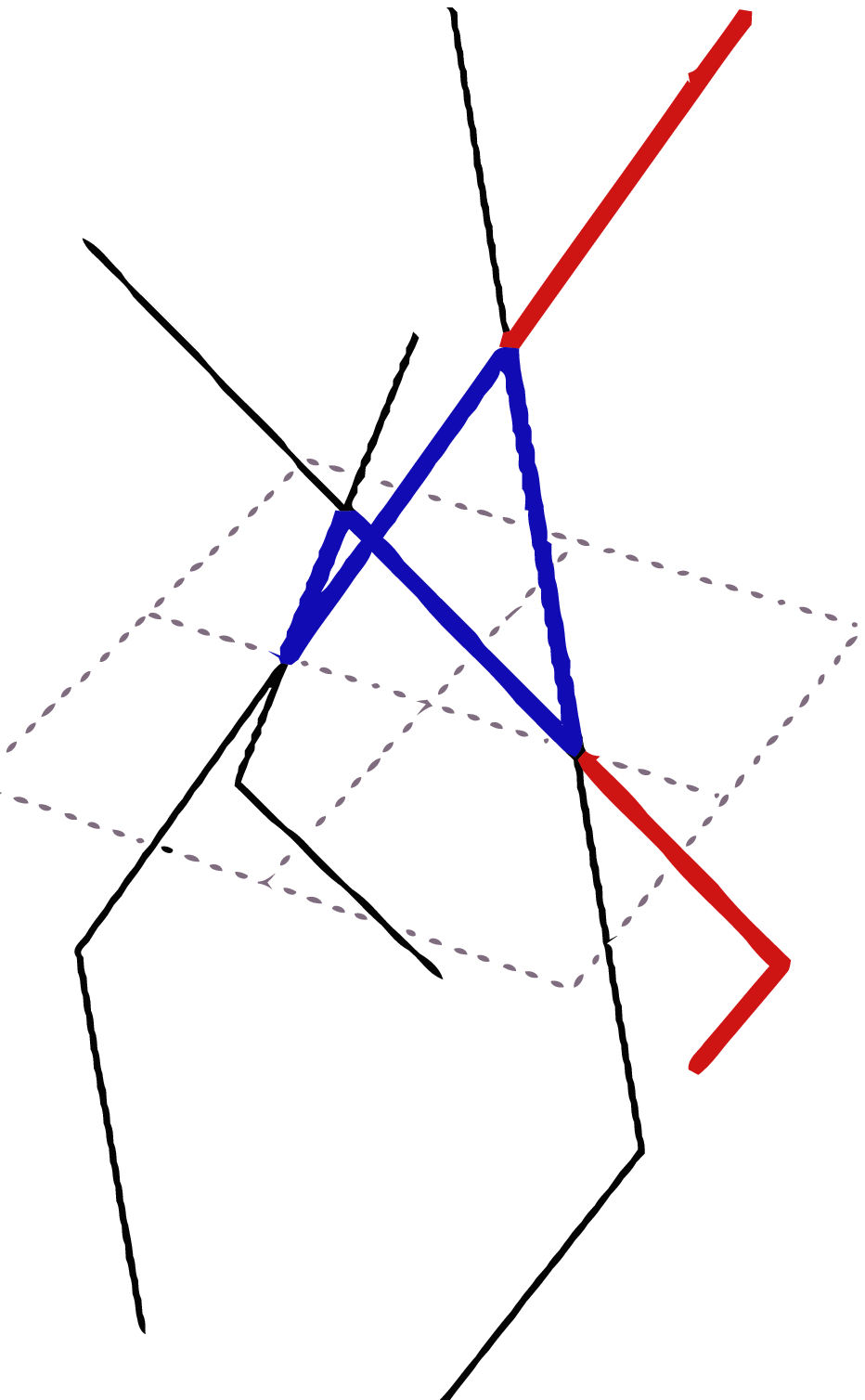}}& \raisebox{-\totalheight}{ \includegraphics*[clip, scale = .29]{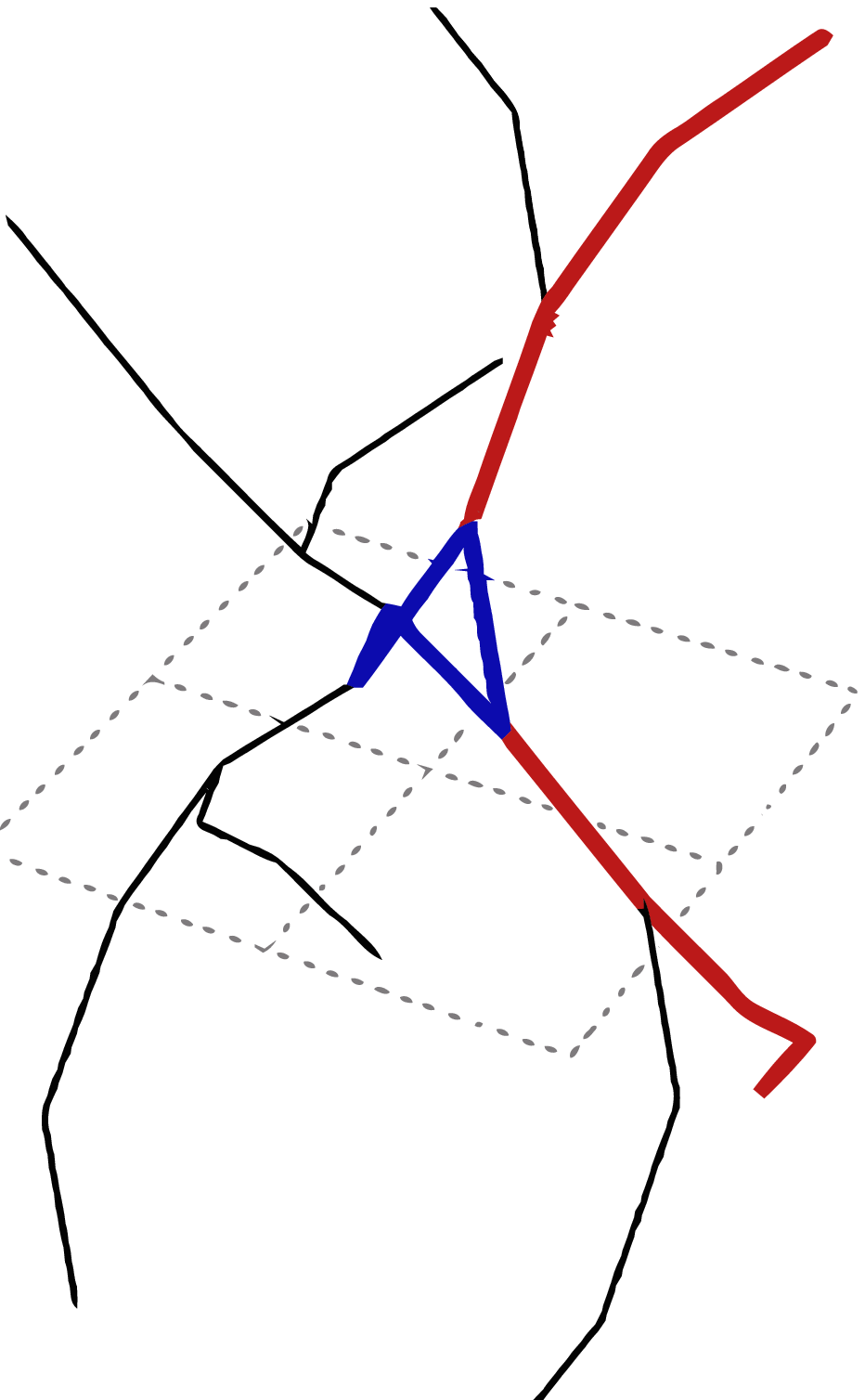}}\\
&&\\
$P^0$&$P^1$&$P^{01}$\\
\raisebox{-\totalheight}{ \includegraphics*[clip, scale = .29]{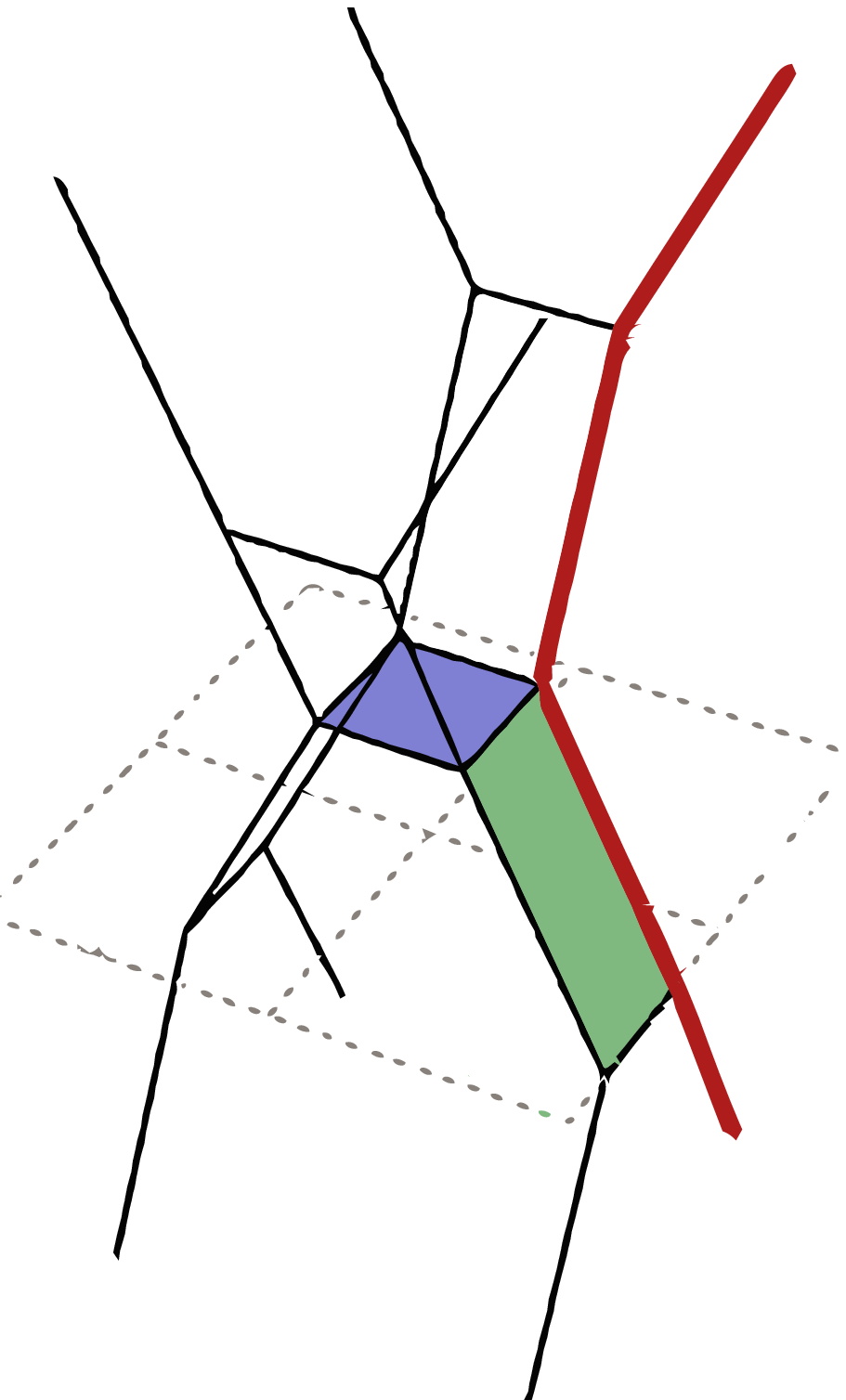}}& \raisebox{-\totalheight}{ \includegraphics*[clip, scale = .29]{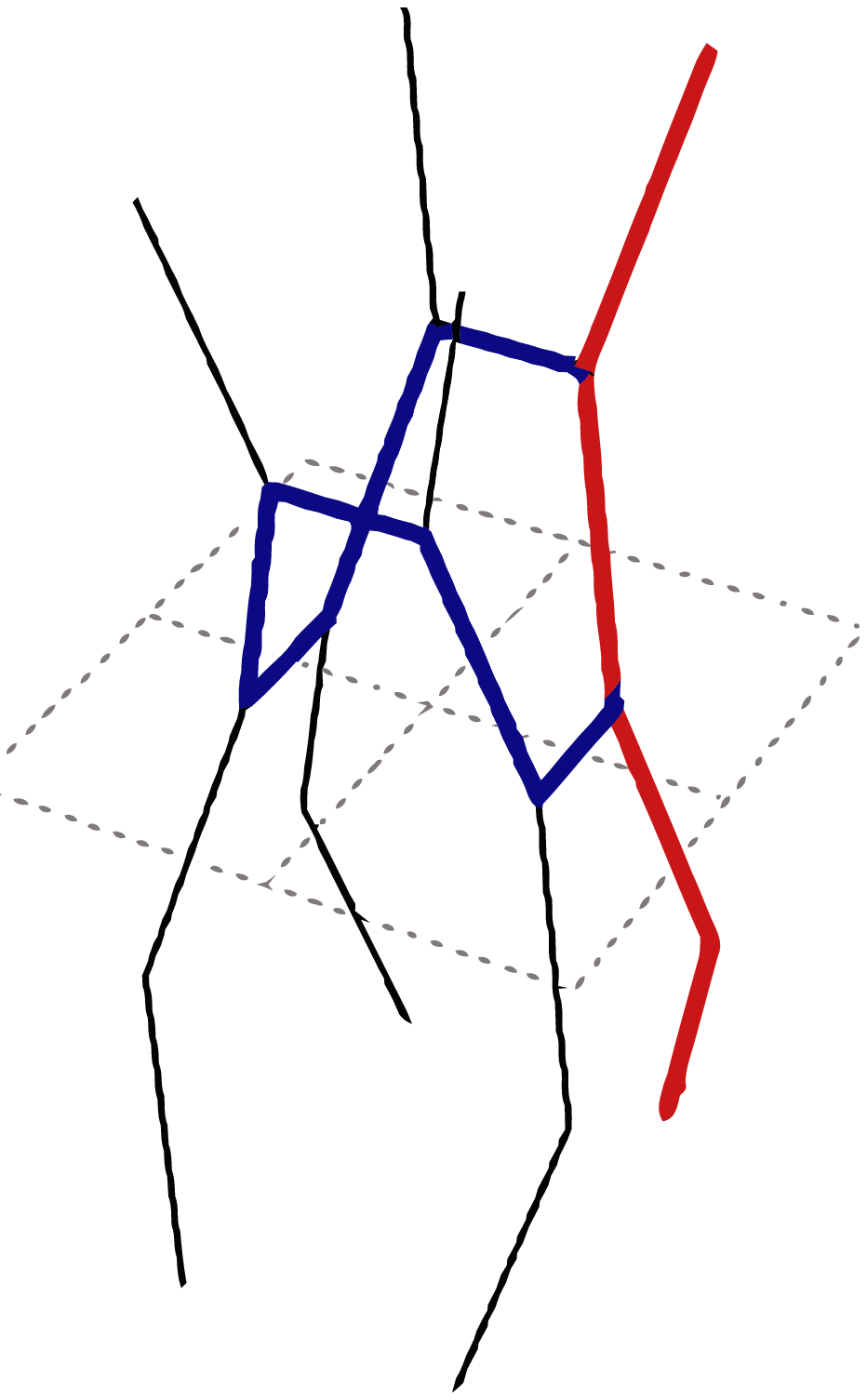}}& \raisebox{-\totalheight}{ \includegraphics*[clip, scale = .29]{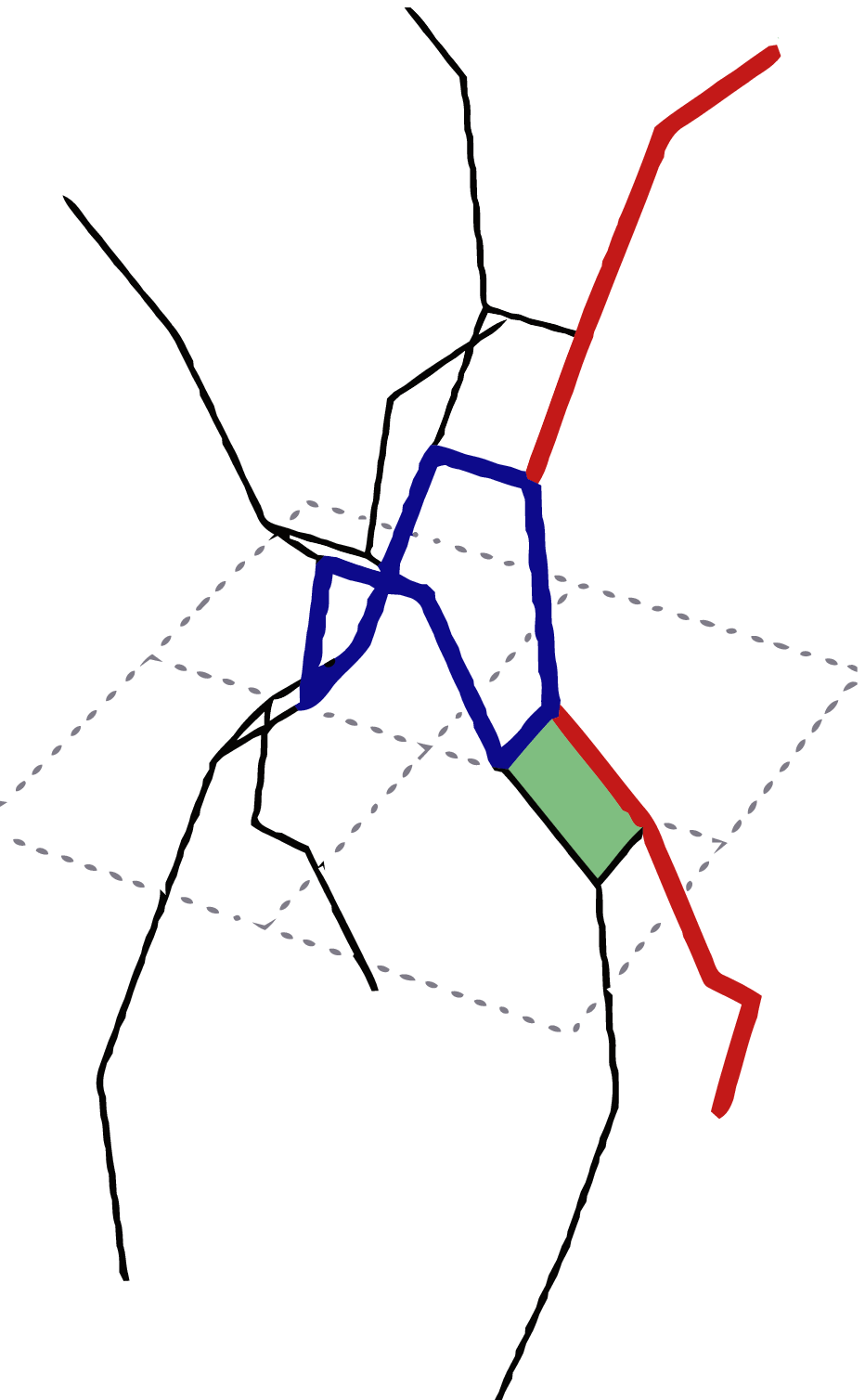}} \\
&&\\
$P^{02}$&$P^{12}$&$P^{012}$
\end{tabular}\caption{The Wythoffians derived from $\{4,4\}\#\{\infty\}$.}\label{4,4 inf}\end{center}\end{figure}

The last geometrically regular blended apeirohedron based on the square tessellation of the plane is $\{\infty,4\}_4\#\{\infty\}$. Letting $s_0,s_1,s_2$ and $t_0,t_1$ be as above, the symmetry group is given by $G(\{\infty,4\}_4\#\{\infty\})=\langle r_0,r_1,r_2\rangle$ with $r_{0}:=s_{0}s_{2}t_{0}$, $r_{1}:=s_1t_1$ and $r_{2}:=s_2$.
Here $r_0$ is a point reflection, $r_1$ is a half-turn, and $r_2$ is a plane reflection. Again there is no point which is invariant under $t_0$ and $t_1$ so this will limit the choice of initial vertex and prevent there being a $P^2$. Additionally, any point which is invariant under $r_0$ is also invariant under~$r_2$, so this excludes $P^{2}$ and $P^{12}$. The initial vertices all come from the fundamental region of $\{4,4\}\#\{\infty\}$.  We further restrict the choice of initial vertex to lie in a base face of $\{\infty,4\}_4\#\{\infty\}$ when applicable ($P^{01}$ and $P^{012}$).  This will ensure that the faces of the Wythoffians will have similar planarity to the faces of $\{\infty,4\}_4\#\{\infty\}$.  For pictures of the Wythoffians see Figure \ref{inf,4_4 inf}.

The first Wythoffian, $P^0$, is the regular apeirohedron $\{\infty,4\}_4\#\{\infty\}$ itself. The faces of type $F_2^{\{0,1\}}$ are regular zigzag apeirogons.   Each apeirogon lies in a plane which crosses through both the reflection planes of $t_0$ and $t_1$.  When projected onto the plane of $t_0$ they appear as planar zigzags. Four zigzags meet at each vertex, yielding the vertex symbol $(\infty_2^4)$ and making the vertex figure an antiprismatic square. The projection of this Wythoffian onto the reflection plane of $t_0$ appears as $\{\infty,4\}_4$. 

In the apeirohedron $P^1$, the faces of type $F_2^{\{0,1\}}$ are linear apeirogons. Each apeirogon corresponds to a zigzag of $\{\infty,4\}_{4}\#\{\infty\}$ and tessellates the line connecting the midpoints of the edges of the zigzag. The faces of type $F_2^{\{1,2\}}$ are antiprismatic squares. There are two squares and two lines alternating about each vertex, giving the vertex symbol $(4_s.\infty.4_s.\infty)$. The vertex figure is a crossed quadrilateral. This is a uniform apeirohedron. The projection of this Wythoffian onto the reflection plane of $t_0$ appears as the Wythoffian $P^{1}$ of $\{\infty,4\}_4$. 

For the apeirohedron $P^{01}$ the faces of type $F_2^{\{0,1\}}$ are  apeirogons.  Each one corresponds to a face of $P^0$ such that the orthogonal projection of $F_2^{\{0,1\}}$ onto the plane of the base face of $\{\infty,4\}_4\#\{\infty\}$ appears as a truncated zigzag. When the initial vertex lies in the base face of $\{\infty,4\}_4\#\{\infty\}$,  the  faces of type $F_2^{\{0,1\}}$ is a truncation of that base face.   The faces of type $F_2^{\{1,2\}}$ are antiprismatic squares. There are two apeirogons and one square at each vertex, with vertex symbol $(4_s.t\infty_2.t\infty_2)$  (where here we use $t\infty_2$ to indicate the apeirogon's relationship with truncated zigzags). The resulting vertex figure is an isosceles triangle. The truncated zigzags are not regular polygons so this Wythoffian is not a uniform apeirohedron. The projection of this Wythoffian onto the reflection plane of $t_0$ appears as the Wythoffian $P^{01}$ of $\{\infty,4\}_4$. 

In the apeirohedron $P^{02}$ the faces of type $F_2^{\{0,1\}}$ are regular zigzags, the faces of type $F_2^{\{1,2\}}$ are convex squares lying parallel to the plane of $t_0$, and the faces of type $F_2^{\{0,2\}}$ are crossed quadrilaterals. At each vertex, in cyclic order, there is a crossed quadrilateral, a square, a crossed quadrilateral, and an apeirogon, giving $(4_{\,\bowtie}.4_c.4_{\,\bowtie}.\infty_2)$  as vertex symbol. The resulting vertex figure is a skew quadrilateral. The crossed quadrilateral faces are not regular so this is not a uniform apeirohedron. The projection of this Wythoffian onto the reflection plane of $t_0$ appears as the Wythoffian $P^{02}$ of $\{\infty,4\}_4$. 

Lastly, we will look at  the apeirohedron $P^{012}$.  Similar to $P^{01}$ the face $F_2^{\{0,1\}}$ is an apeirogon which orthogonally projects as a truncated zigzag onto the plane of the base face of $\{\infty,4\}_4\#\{\infty\}$.  For some initial vertices this apeirogon is planar and for other choices it is not. The faces of type $F_2^{\{1,2\}}$ are skew octagons (truncated antiprismatic squares) and the faces of type $F_2^{\{0,2\}}$ are planar crossed quadrilaterals. There is one face of each type meeting at each vertex yielding a triangular vertex figure with vertex symbol  $(4_{\,\bowtie}.8_s.t\infty_2)$.  None of the faces are regular so $P^{012}$ is not a uniform apeirohedron. The projection of this Wythoffian onto the reflection plane of $t_0$ appears as the Wythoffian $P^{012}$ of $\{\infty,4\}_4$.

\begin{figure}[h]
\begin{center}
\begin{tabular}{ccccc}
\raisebox{-\totalheight}{ \includegraphics*[clip, scale = .29]{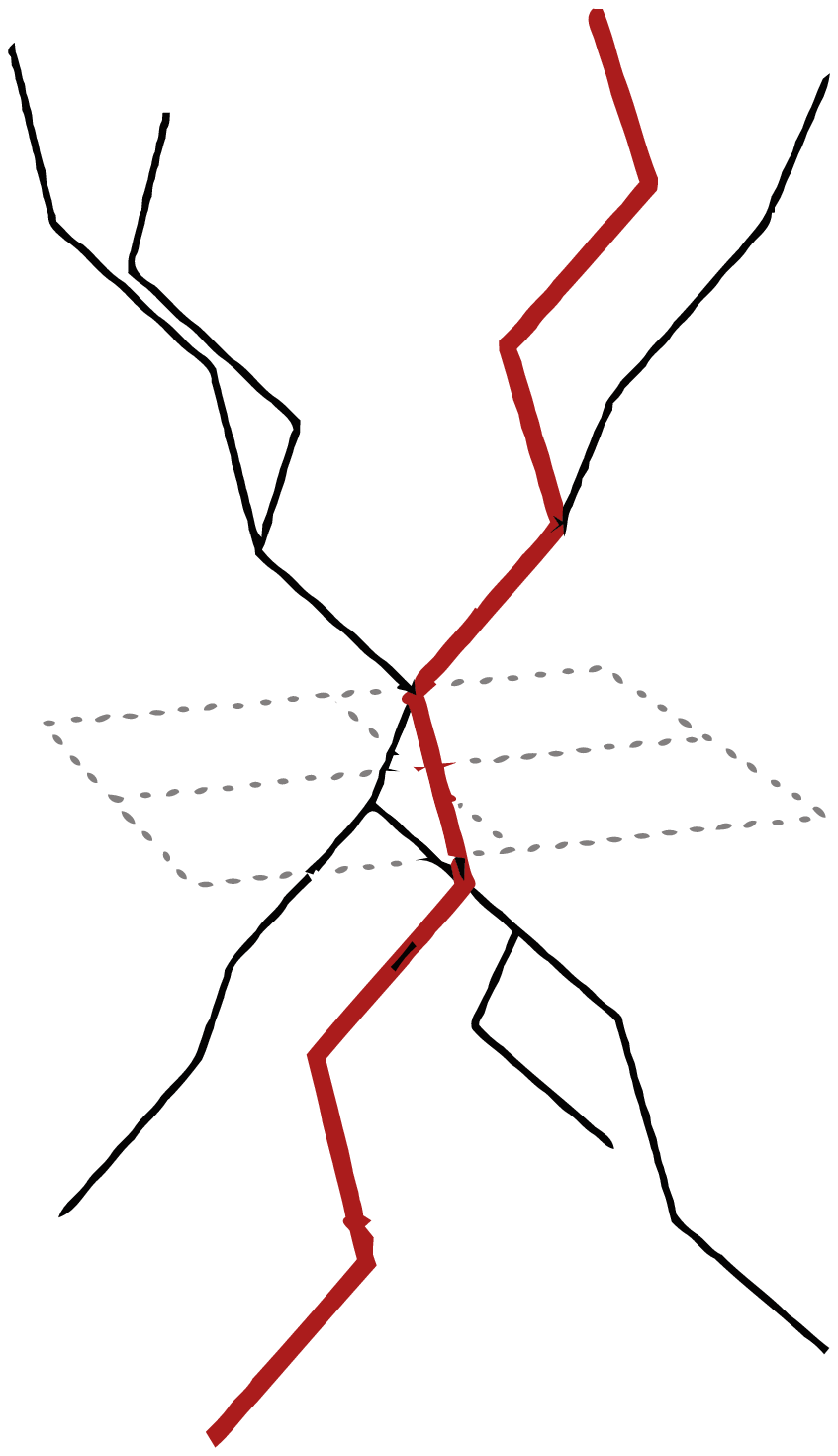}}& \raisebox{-\totalheight}{ \includegraphics*[clip, scale = .29]{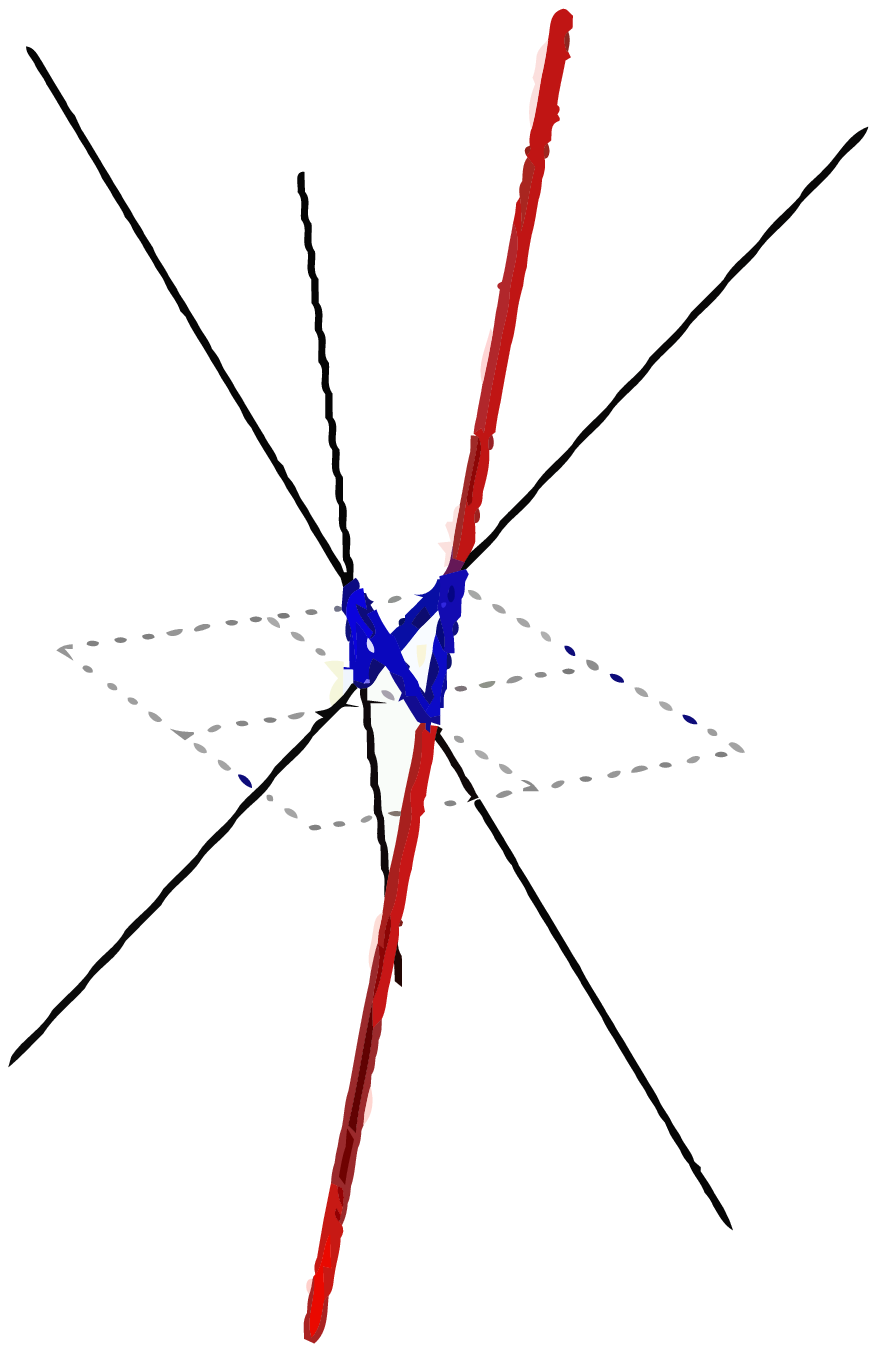}}&\raisebox{-\totalheight}{ \includegraphics*[clip, scale = .29]{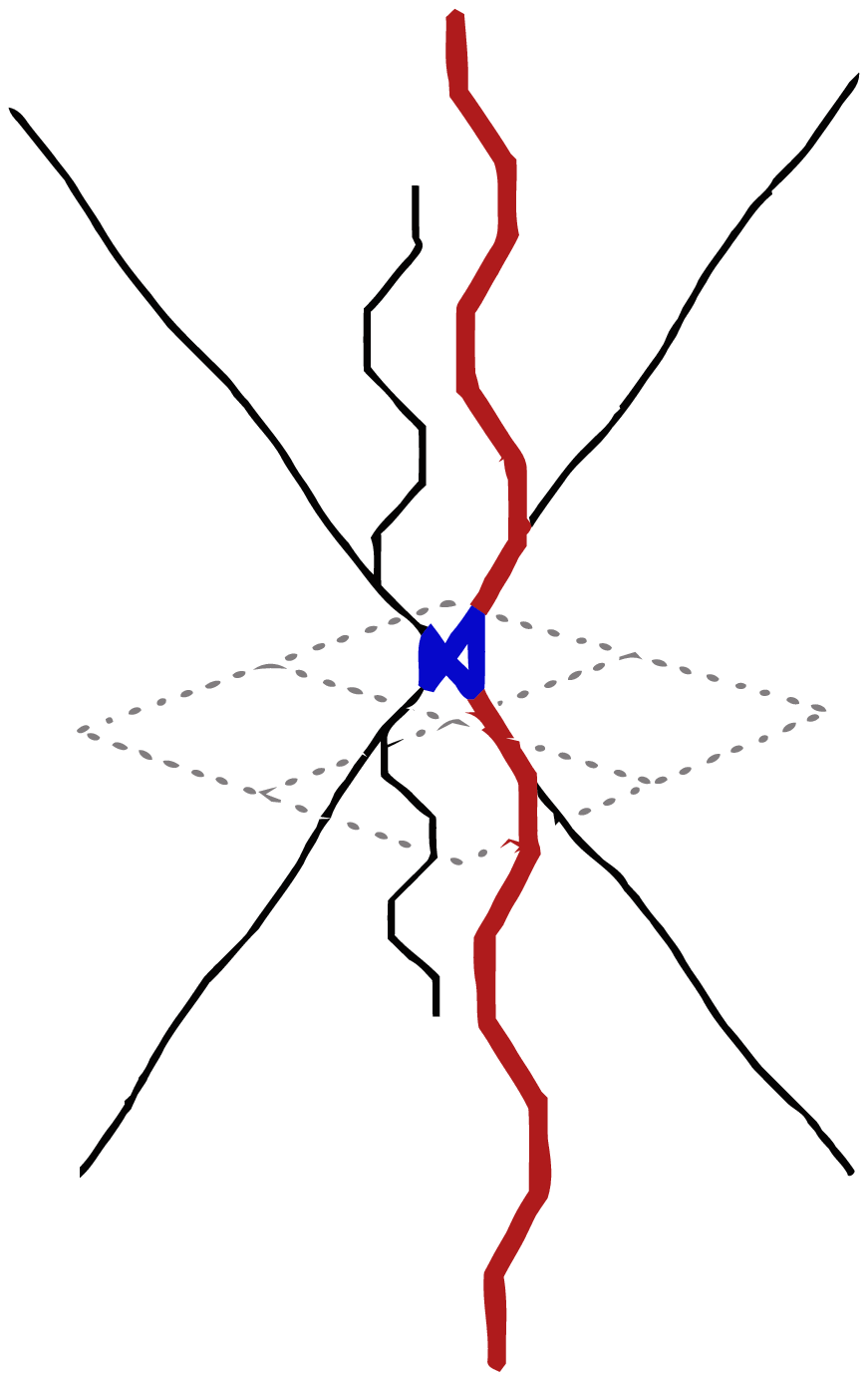}}& \raisebox{-\totalheight}{ \includegraphics*[clip, scale = .29]{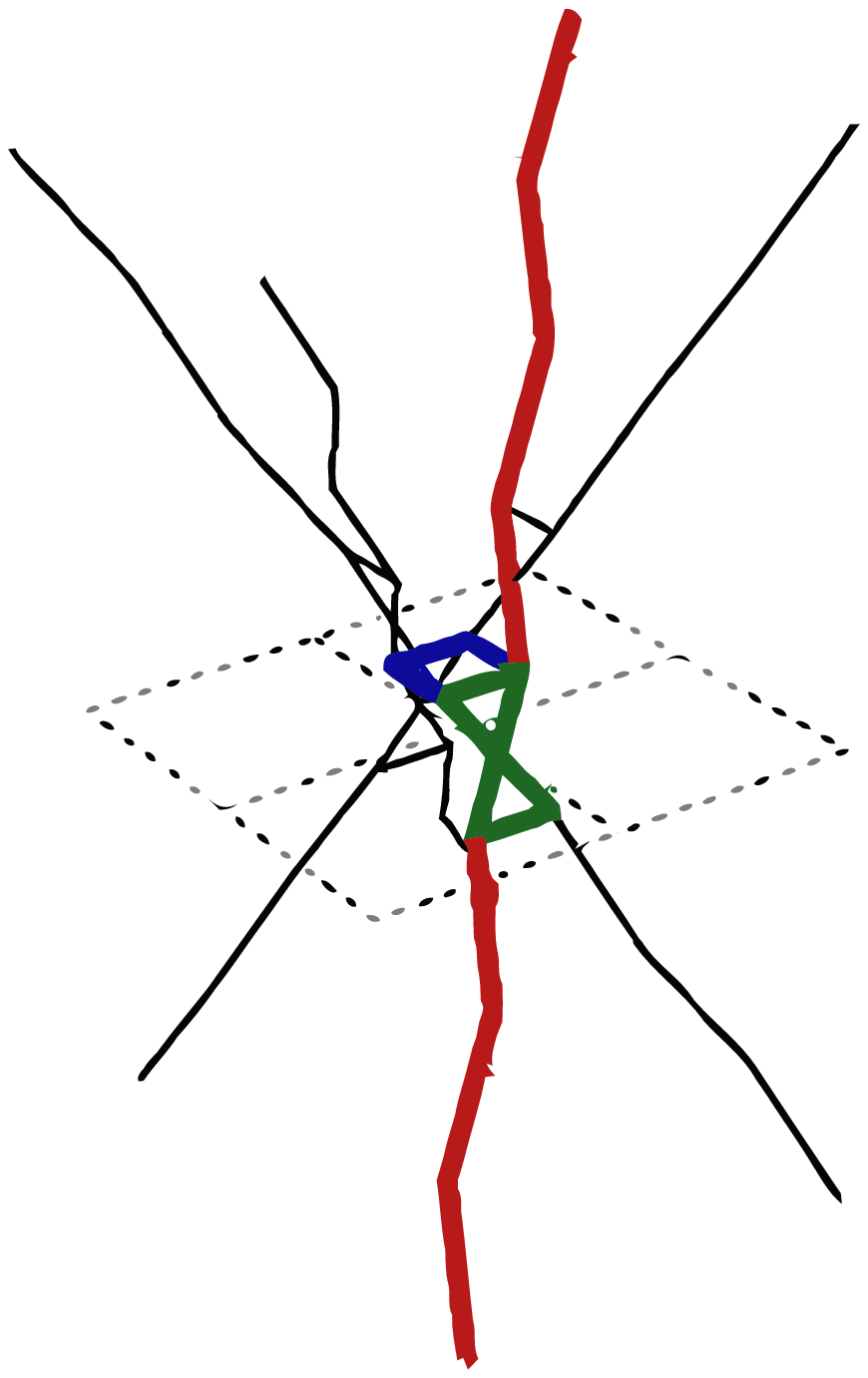}}&\raisebox{-\totalheight}{ \includegraphics*[clip, scale = .29]{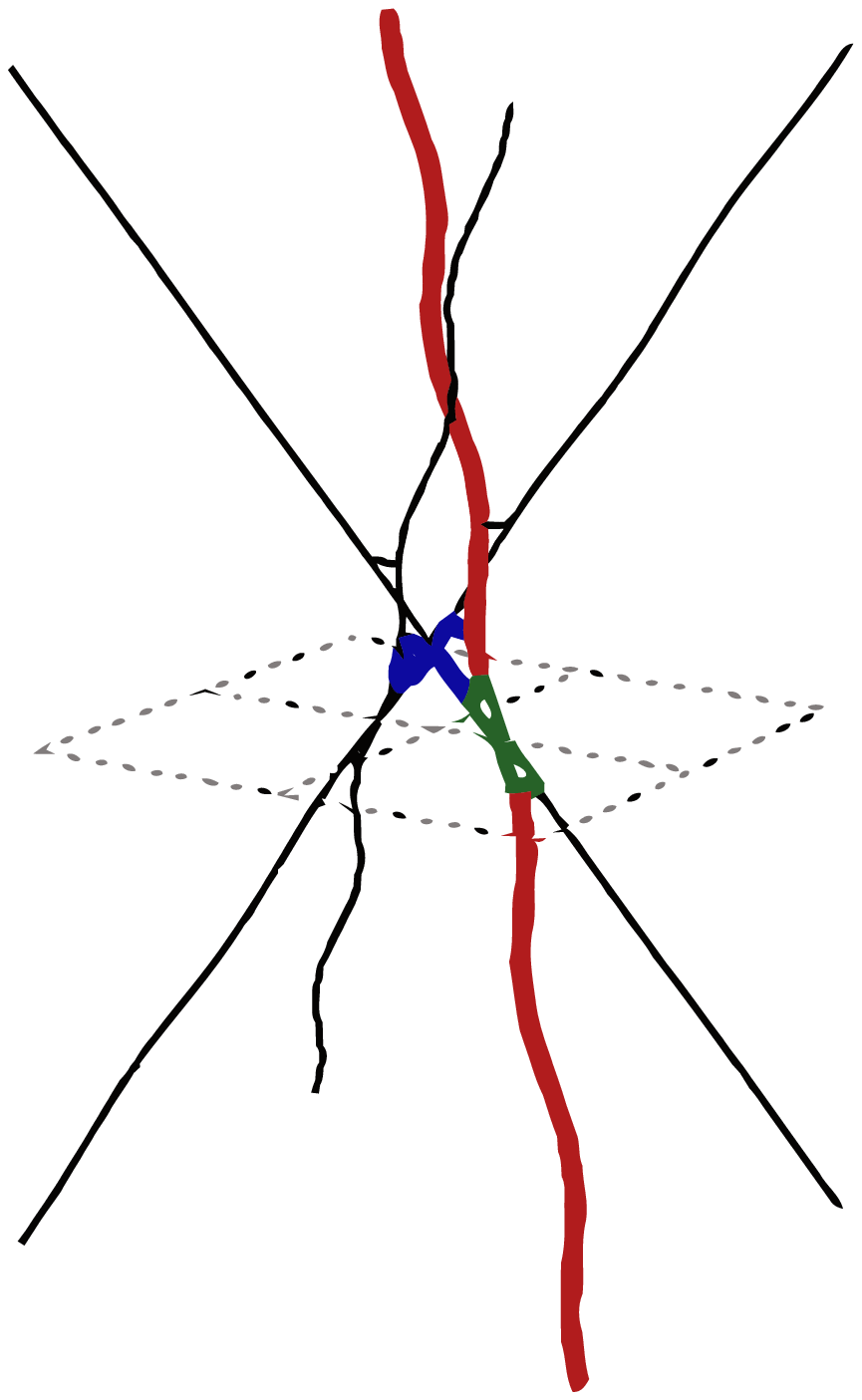}}\\
&&&&\\
$P^0$&$P^1$&$P^{01}$&$P^{02}$&$P^{012}$ 
\end{tabular}\caption{The Wythoffians derived from $\{\infty,4\}_4\#\{\infty\}$.}\label{inf,4_4 inf}\end{center}\end{figure}

\subsection{Petrie-Coxeter polyhedra}

In this final section we examine the Wythoffians of the Petrie-Coxeter polyhedra, the three most prominent examples of pure regular apeirohedra. These have convex faces and skew vertex figures. The symmetry group of each of them can be derived from the symmetry group of the cubical honeycomb, $\{4,3,4\}$. We take this symmetry group in the form $G(\{4,3,4\})=\langle t_0,t_1,t_2,t_3\rangle$, where $t_0,t_1,t_2,t_3$ are the distinguished generators (as in~\cite[p.~231]{SchMc}). The fundamental region of $G(\{4,3,4\})$ in $\mathbb{E}^3$ is a simplex with vertices at the centers of the faces in a flag of $\{4,3,4\}$, and each generator $t_j$ is the reflection in the plane bounding the simplex and opposite to the vertex corresponding to the $j$-face in the flag.

We begin with the Petrie-Coxeter polyhedron $\{4,6|4\}$. From~\cite[p.~231]{SchMc} we know that $G(\{4,6|4\})=\langle r_0,r_1,r_2\rangle$, where $r_{0}:=t_0$, $r_{1}:=t_{1}t_{3}$ and $r_{2}:=t_2$. Note that $r_0$ and $r_2$ are plane reflections, and that $r_1$ is a halfturn. For the Wythoffians the initial vertices have all been chosen so that they are points of the fundamental region within the convex hull of the base face of $\{4,6|4\}$. This choice leads to the resulting figures being more geometrically similar to $\{4,6|4\}$. Other points in the fundamental region which belong to the same Wythoffian class generate combinatorially isomorphic figures, but previously planar faces may become skew, or vice versa. For instance, in the cases where the initial vertex is transient under $r_1$, choosing a point outside of the convex hull of a face of $\{4,6|4\}$ destroys the planarity of the faces of type $F_2^{\{0,1\}}$ but does preserve the isomorphism type of the apeirohedron. For pictures of the Wythoffians see Figure \ref{4,6l4}.

The first apeirohedron, $P^0$, is $\{4,6|4\}$ itself. It has convex square faces, six of which meet at each vertex, and so the vertex symbol is $(4^6)$. The vertex figure is a regular, antiprismatic hexagon. 

For $P^1$, the apeirohedral Wythoffian has convex square faces of type $F_2^{\{0,1\}}$ while the faces of type $F_2^{\{1,2\}}$ are regular, antiprismatic hexagons. Cyclically, at each vertex, there is a hexagon, a square, a hexagon, and a square. Thus the vertex symbol is $(4_c.6_s.4_c.6_s)$ and the vertex figure is a rectangle. The faces are all regular polygons so this is a uniform apeirohedron. 

The Wythoffian $P^2$ is the dual of $\{4,6|4\}$, the regular apeirohedron $\{6,4|4\}$. Each face is a regular, convex hexagon of type $F_2^{\{1,2\}}$. Four come together at each vertex yielding skew quadrilateral as the vertex figure with a vertex symbol $(6_c^4)$. 

For the next apeirohedral Wythoffian, $P^{01}$, the faces of type $F_2^{\{0,1\}}$ are convex octagons and the faces of type $F_2^{\{1,2\}}$ are regular, antiprismatic hexagons. The vertex symbol is $(6_s.8_c^2)$ and an isosceles triangle is the vertex figure. For a specific choice of initial vertex the octagons are regular and the Wythoffian is uniform. Note that for an initial vertex chosen outside of the convex hull of the base face of $\{4,6|4\}$, the octagon would become a truncated antiprismatic quadrilateral which can not be made regular and so in this case the Wythoffian is not uniform. 

In the apeirohedral Wythoffian $P^{02}$ the faces of type $F_2^{\{0,1\}}$ are convex squares; the faces of type $F_2^{\{1,2\}}$ are regular, convex hexagons; and the faces of type $F_2^{\{0,2\}}$ are convex rectangles. Cyclically, at each vertex, there is a square, a rectangle, a hexagon, and a second rectangle, giving a vertex symbol $(4_c.4_c.6_c.4_c)$. The vertex figure is a skew quadrilateral. For certain initial vertex choices the rectangles can be made into squares making the Wythoffian uniform. 

In the apeirohedron $P^{12}$ the faces of type $F_2^{\{0,1\}}$ are convex squares and the faces of type $F_2^{\{1,2\}}$ are skew dodecagons (truncated antiprismatic hexagons). The vertex symbol is $(4.12_s^2)$ which corresponds to an isosceles triangle as the vertex figure. The skew dodecagons cannot be made regular by any vertex choice and thus this Wythoffian is not a uniform apeirohedron for any initial vertex choice. 

Finally, consider $P^{012}$. In this apeirohedron, the faces of type $F_2^{\{0,1\}}$ are convex octagons, the faces of type $F_2^{\{1,2\}}$ are skew dodecagons (truncated antiprismatic hexagons), and the faces of type $F_2^{\{0,2\}}$ are convex rectangles. As with $P^{12}$ the skew dodecagons are never regular so the apeirohedron is not uniform. There is one face of each type at each vertex, yielding $(4_c.8_c.12_s)$ as a vertex symbol and a triangular vertex figure. 

\begin{figure}[h]
\begin{center}
\begin{tabular}{cccc}
\raisebox{-\totalheight}{ \includegraphics*[clip, scale = .25]{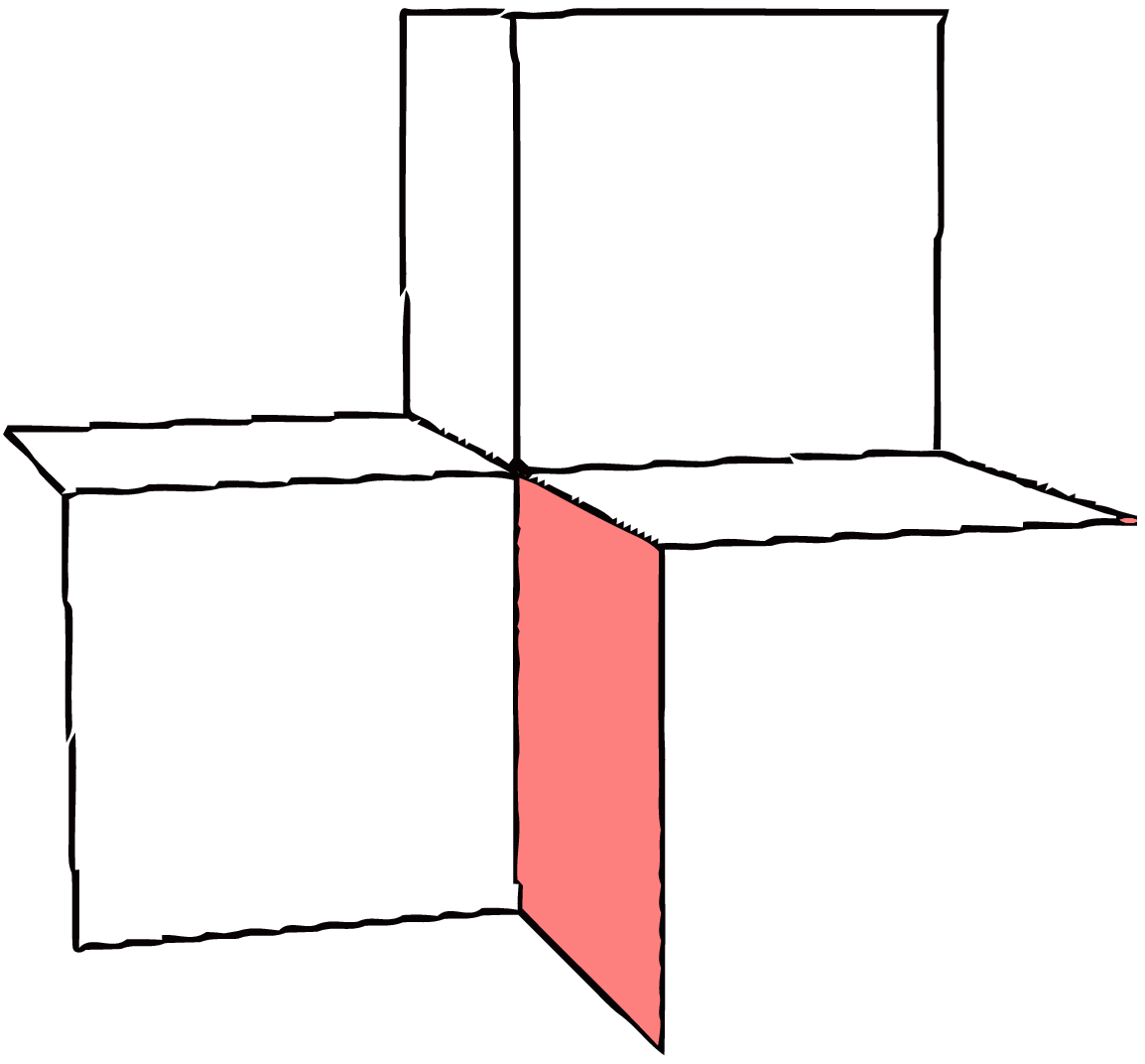}}& \raisebox{-\totalheight}{ \includegraphics*[clip, scale = .25]{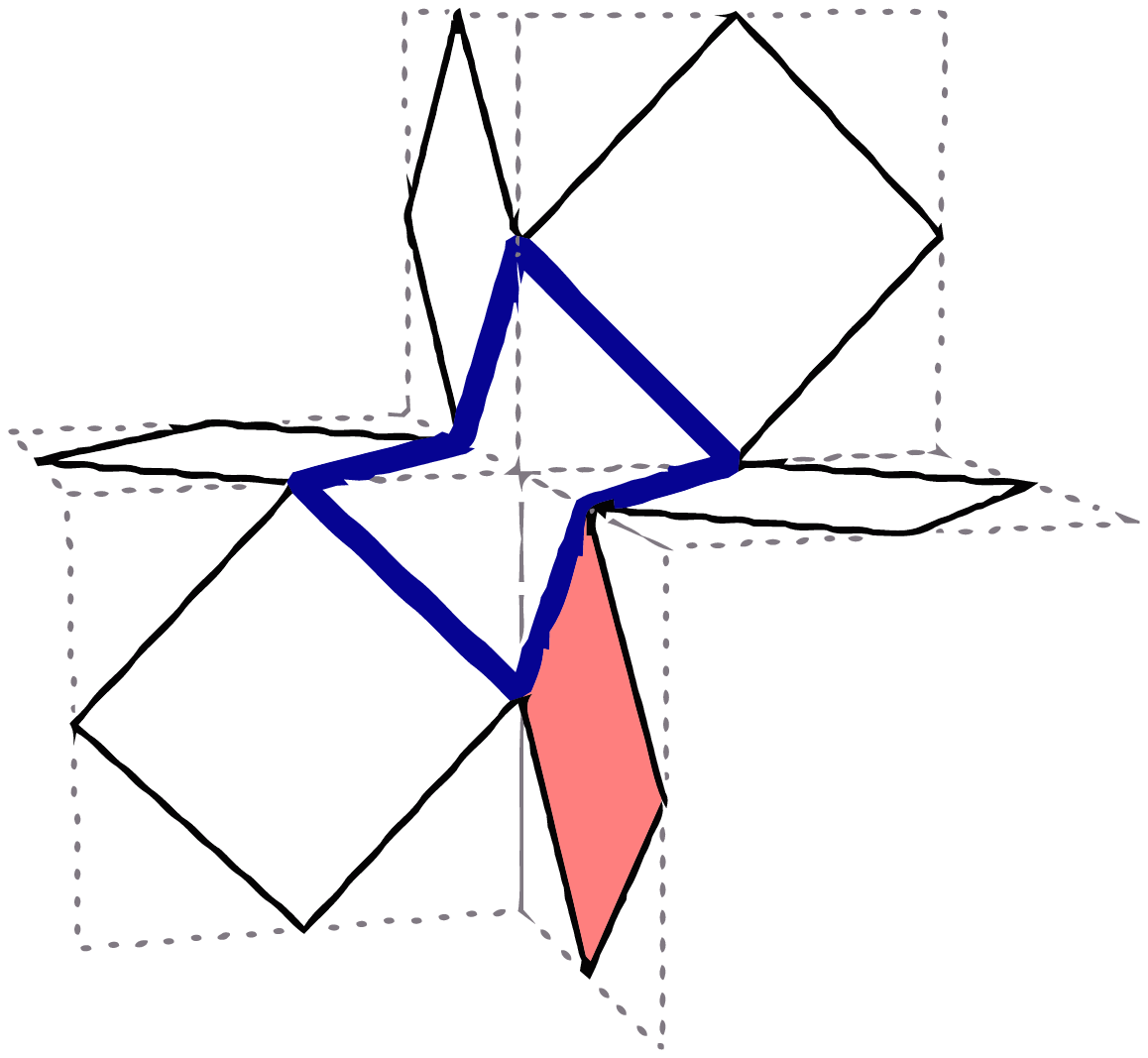}}& \raisebox{-\totalheight}{ \includegraphics*[clip, scale = .25]{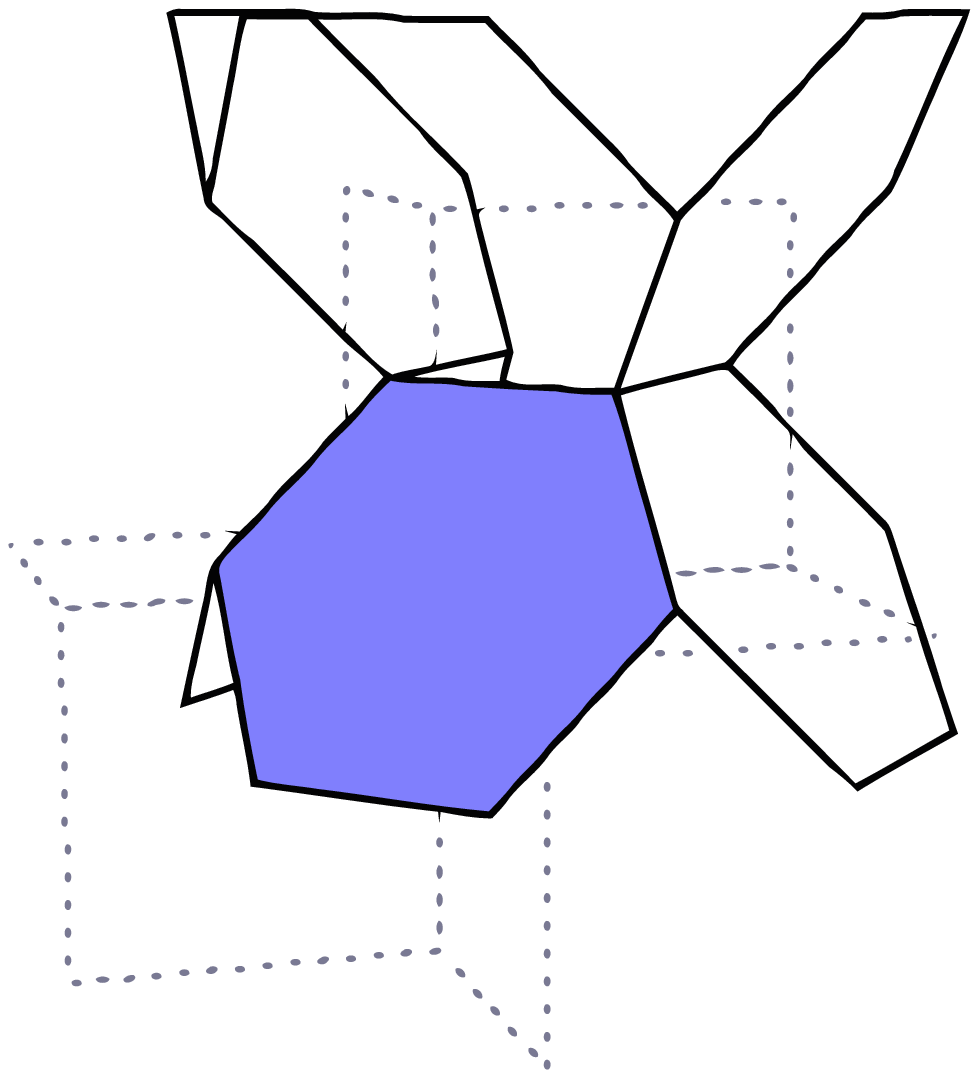}}&\raisebox{-\totalheight}{ \includegraphics*[clip, scale = .25]{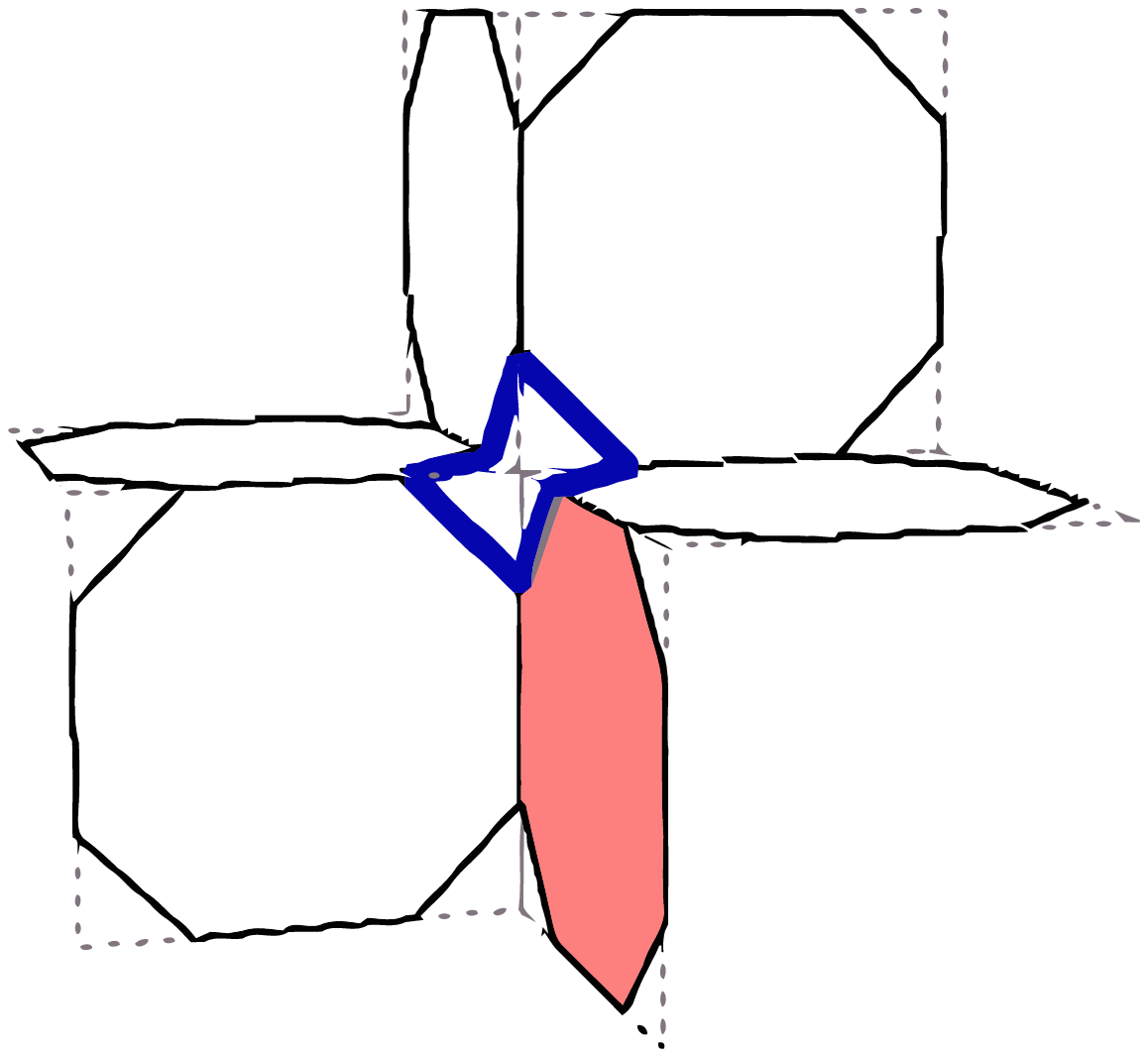}} \\
$P^{0}$ & $P^{1}$ & $P^{2}$&$P^{01}$\\
&&&\\
\raisebox{-\totalheight}{ \includegraphics*[clip, scale = .25]{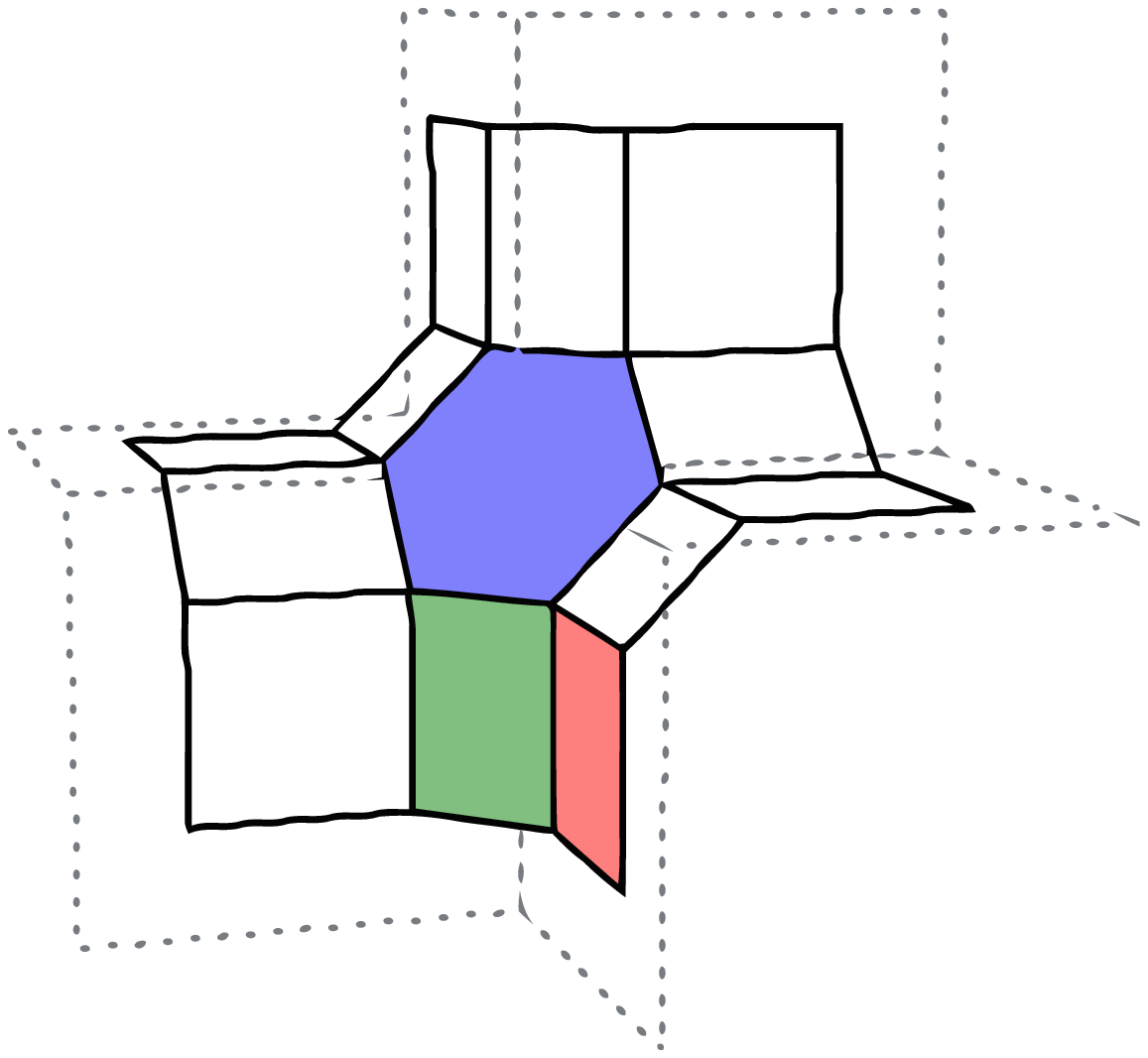}}& \raisebox{-\totalheight}{ \includegraphics*[clip, scale = .25]{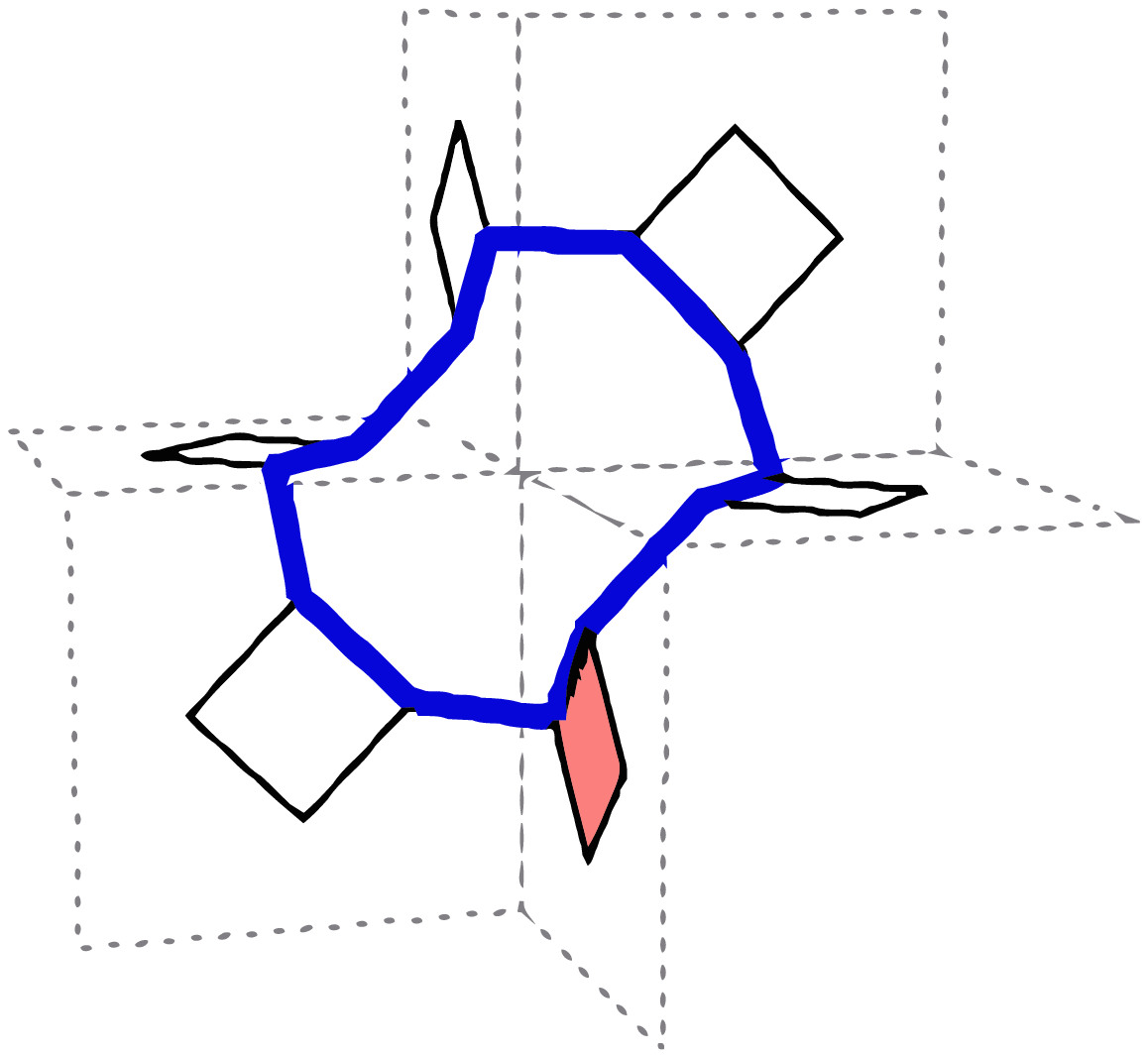}}& \raisebox{-\totalheight}{ \includegraphics*[clip, scale = .25]{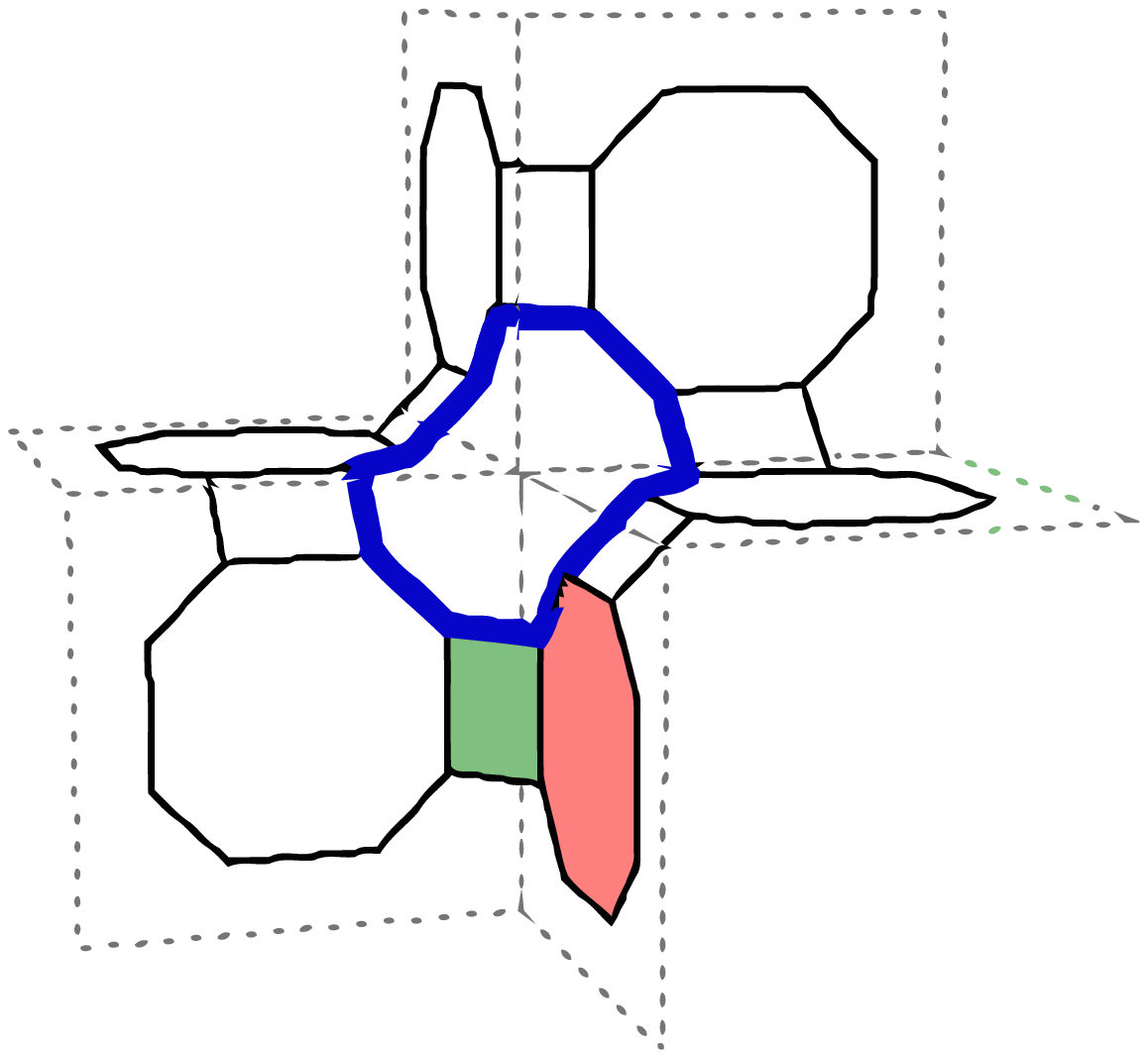}}&\\
$P^{02}$& $P^{12}$ & $P^{012}$&
\end{tabular}\caption{The Wythoffians derived from $\{4,6|4\}$.}\label{4,6l4}\end{center}\end{figure}

Next we investigate the Wythoffians of the Petrie-Coxeter polyhedron $\{6,4|4\}$, the dual of $\{4,6|4\}$. As such, its symmetry group is $G(\{6,4|4\})=\langle r_0,r_1,r_2\rangle$, where $r_{0}:=t_2$, $r_{1}:=t_{1}t_{3}$ and $r_{2}:=t_0$; these are generators of $G(\{4,6|4\})$ in reverse order. Note that the duality of $\{6,4|4\}$ and $\{4,6|4\}$ is geometric: we can produce one polyhedron from the other by reversing the order of the generators of its symmetry group and then applying Wythoff's construction with the new generators. As with $\{4,6|4\}$ we will only consider initial vertices which are contained within the convex hull of the base face and the fundamental region of $\{6,4|4\}$. As before, choosing the vertices in this way makes the faces of the Wythoffians more geometrically similar to the faces of $\{6,4|4\}$. Choosing an initial vertex within the base face versus an initial vertex from outside of the base face (but still within the fundamental region) will only affect the planarity of the faces but not the combinatorial properties. For pictures of the Wythoffians see Figure \ref{6,4l4}.

Due to the geometric duality between $\{6,4|4\}$ and $\{4,6|4\}$ we can interchange 0 and 2 in the superscripts of the Wythoffians of $\{6,4|4\}$ and get the Wythoffians of $\{4,6|4\}$, and vice versa. Note, however, that an initial vertex chosen in the base face of one of $\{6,4|4\}$ or $\{4,6|4\}$ will generally not also lie in the base face of the other. This explains why some of the Wythoffians in Figures~\ref{4,6l4} and \ref{6,4l4} that correspond to each other under the interchange of the subscripts $0$ and $2$ look quite different (although they are isomorphic). For example, $P^{01}$ of Figure~\ref{4,6l4} has convex octagons and skew hexagons as faces, while the corresponding polyhedron $P^{12}$ of Figure~\ref{6,4l4} has skew octagons and convex hexagons as faces. The geometry of the Wythoffians of $\{6,4|4\}$ with initial vertices in the base face of $\{6,4|4\}$ is as follows.

The initial Wythoffian, $P^0$, is the regular apeirohedron $\{6,4|4\}$ itself whose faces are convex, regular hexagons. Four such hexagons meet at each vertex yielding a regular, skew quadrilateral as the vertex figure with vertex symbol $(6_c^4)$. 

In the apeirohedron $P^1$ the faces of type $F_2^{\{0,1\}}$ are convex hexagons and the faces of type $F_2^{\{1,2\}}$ are regular, skew quadrilaterals. The vertex symbol is $(4_s.6_c.4_s.6_c)$ so the vertex figure is a rectangle. This is a uniform apeirohedron. 

The Wythoffian $P^2$ is $\{4,6|4\}$, the dual of $\{6,4|4\}$. The faces are convex squares of type $F_2^{\{1,2\}}$ and there are six circling each vertex with vertex symbol $(4_c^6)$. The vertex figure is a regular, antiprismatic hexagon. 

In the apeirohedron $P^{01}$ the faces of type $F_2^{\{0,1\}}$ are convex dodecagons (truncated hexagons) and the faces of type $F_2^{\{1,2\}}$ are regular, skew quadrilaterals. The vertex symbol is $(4_s.12_c^2)$ yielding an isosceles triangle as a vertex figure. For a carefully chosen initial vertex the dodecagons are regular and this Wythoffian is uniform. 

For the apeirohedron $P^{02}$,  we get a figure which is congruent to $P^{02}$ of $\{4,6|4\}$.  The faces of type $F_2^{\{0,1\}}$ are convex, regular hexagons; the faces of type $F_2^{\{1,2\}}$ are convex squares; and the faces of type $F_2^{\{0,2\}}$ are convex rectangles. At each vertex there is a rectangle, a square, a rectangle, and a hexagon, in cyclic order, yielding a vertex symbol of $(4_c.4_c.6_c.4_c)$. The vertex figure is then a convex quadrilateral. For certain choices of initial vertex the faces are all regular and the apeirohedron is uniform.  

In the apeirohedron $P^{12}$ the faces of type $F_2^{\{0,1\}}$ are convex, regular hexagons and the faces of type $F_2^{\{1,2\}}$ are skew octagons (truncated skew quadrilaterals). The vertex symbol is $(6_c.8_s^2)$ resulting in an isosceles triangle as a vertex figure. For an initial vertex choice outside the convex hull of the base face of $\{6,4|4\}$ the skew octagons will sometimes become convex octagons (possibly regular) and the convex hexagons will sometimes become antiprismatic, regular hexagons.   In this case the Wythoffian would be uniform. 

Finally, examine $P^{012}$. In this apeirohedron the faces of type $F_2^{\{0,1\}}$ are convex dodecagons (truncated hexagons), the faces of type $F_2^{\{1,2\}}$ are skew octagons (truncated skew quadrilaterals), and the faces of type $F_2^{\{0,2\}}$ are convex rectangles. The vertex symbol is $(4_c.8_s.12_c)$ corresponding to a triangular vertex figure. 

\begin{figure}[h]
\begin{center}
\begin{tabular}{cccc}
\raisebox{-\totalheight}{ \includegraphics*[clip, scale = .25]{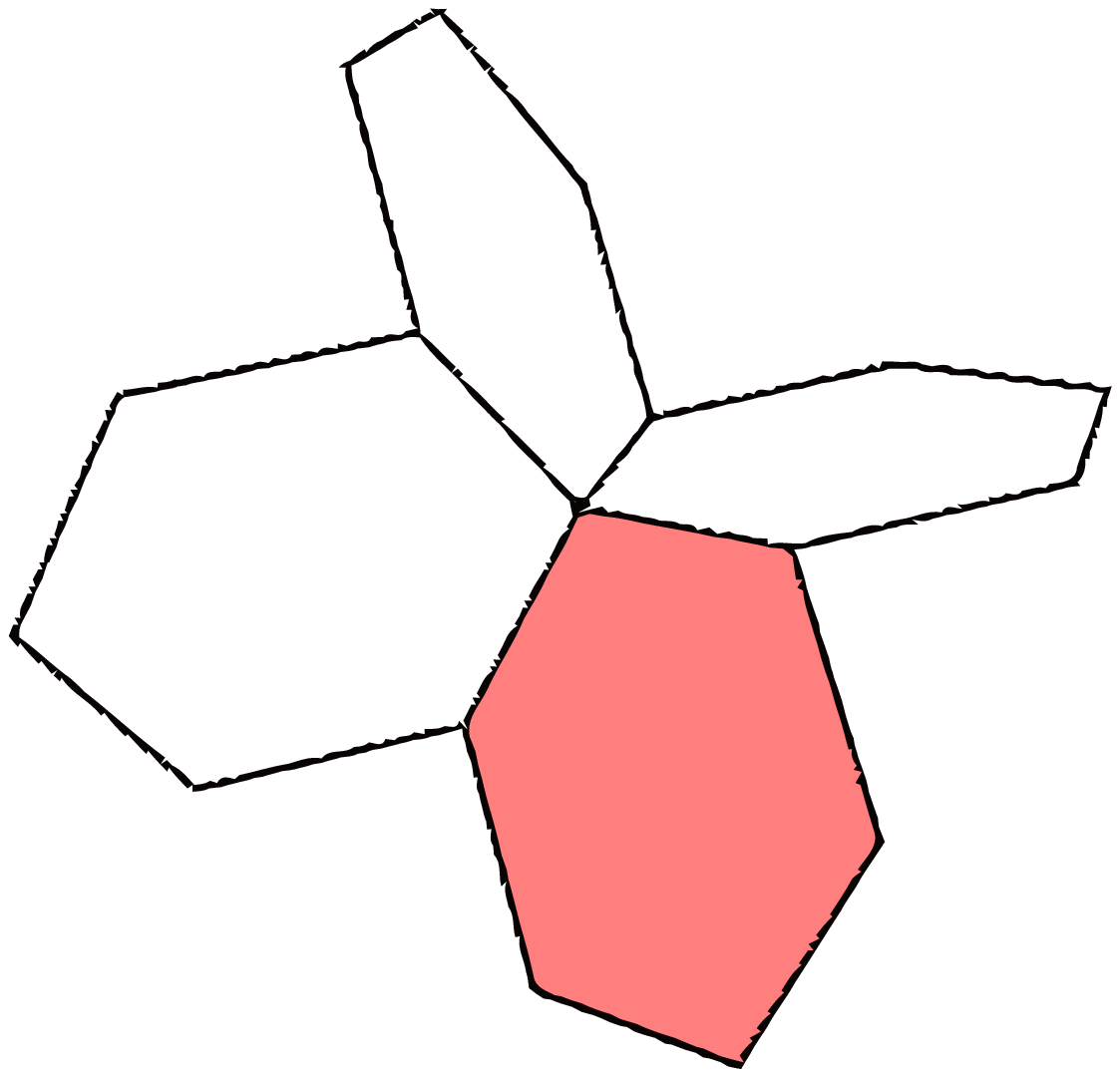}}& \raisebox{-\totalheight}{ \includegraphics*[clip, scale = .25]{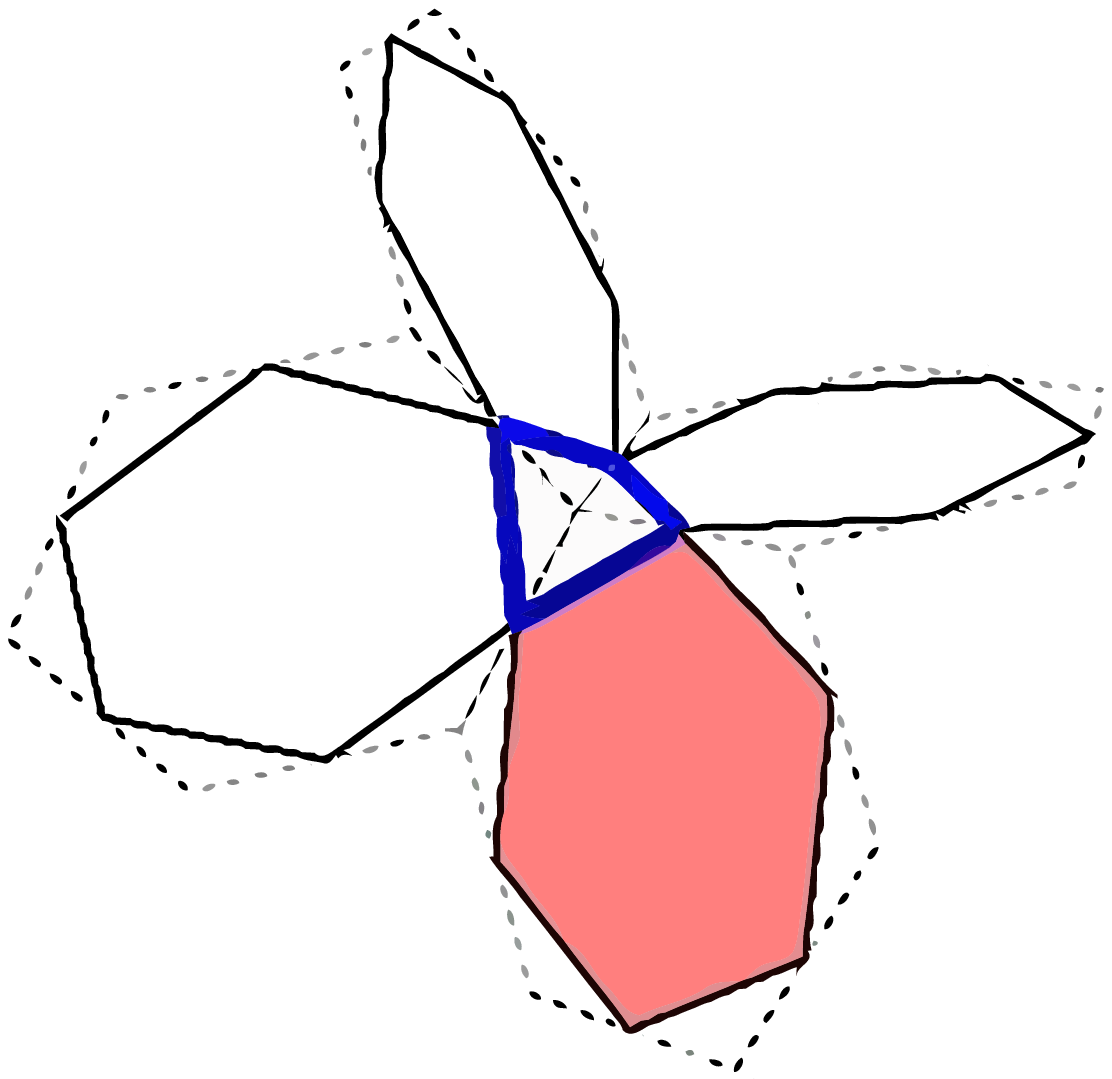}}& \raisebox{-\totalheight}{ \includegraphics*[clip, scale = .25]{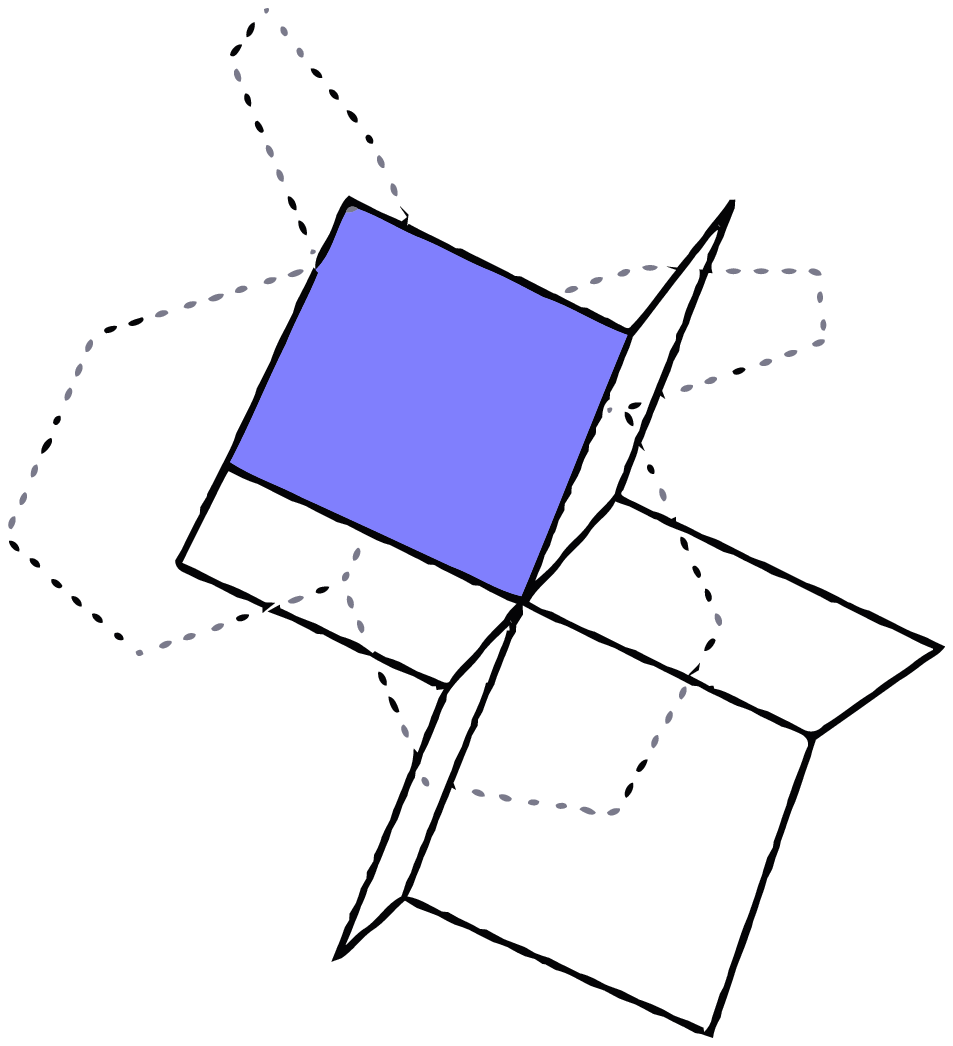}}&\raisebox{-\totalheight}{ \includegraphics*[clip, scale = .25]{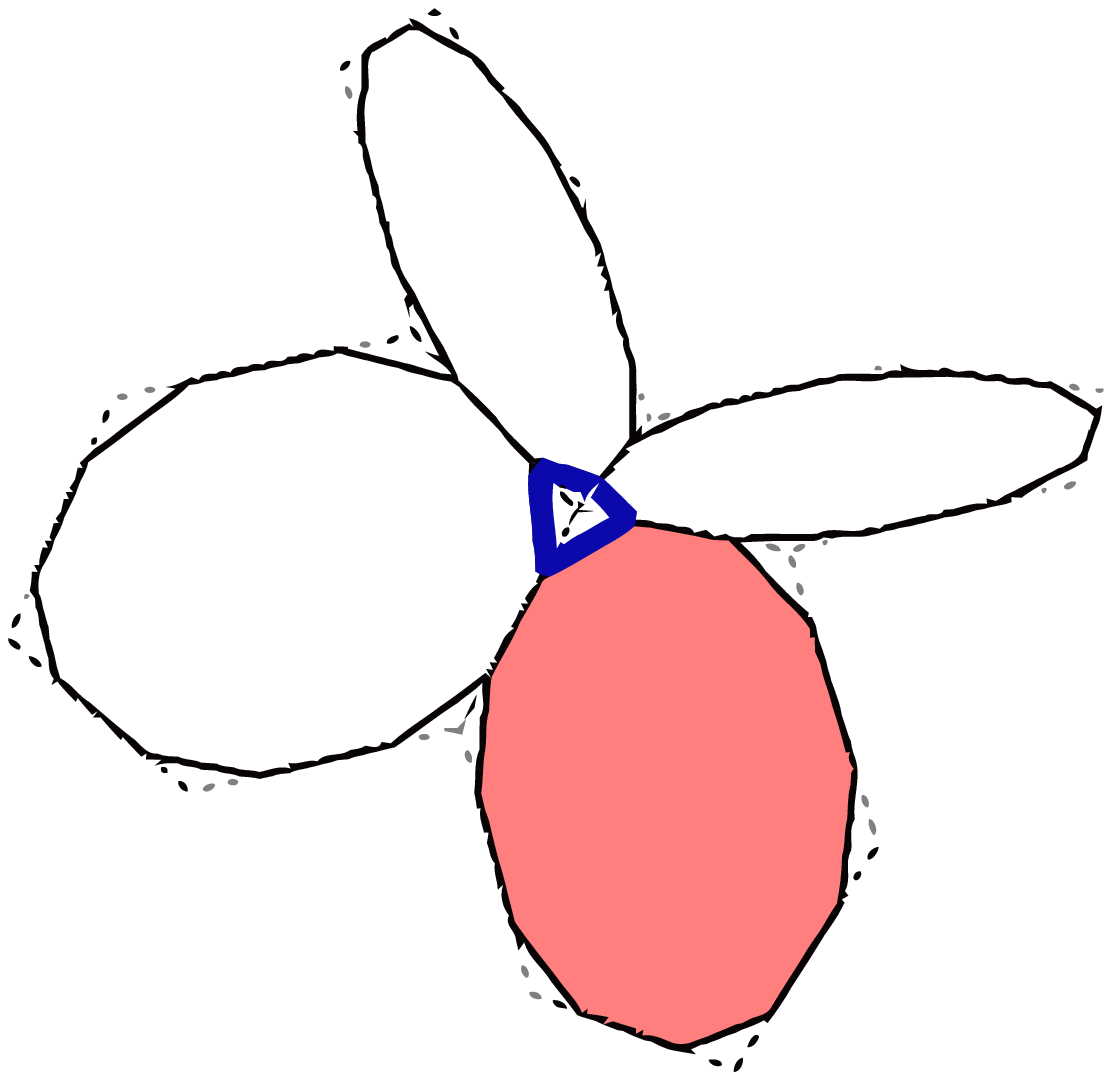}} \\
$P^{0}$ & $P^{1}$ & $P^{2}$&$P^{01}$\\
&&&\\
\raisebox{-\totalheight}{ \includegraphics*[clip, scale = .25]{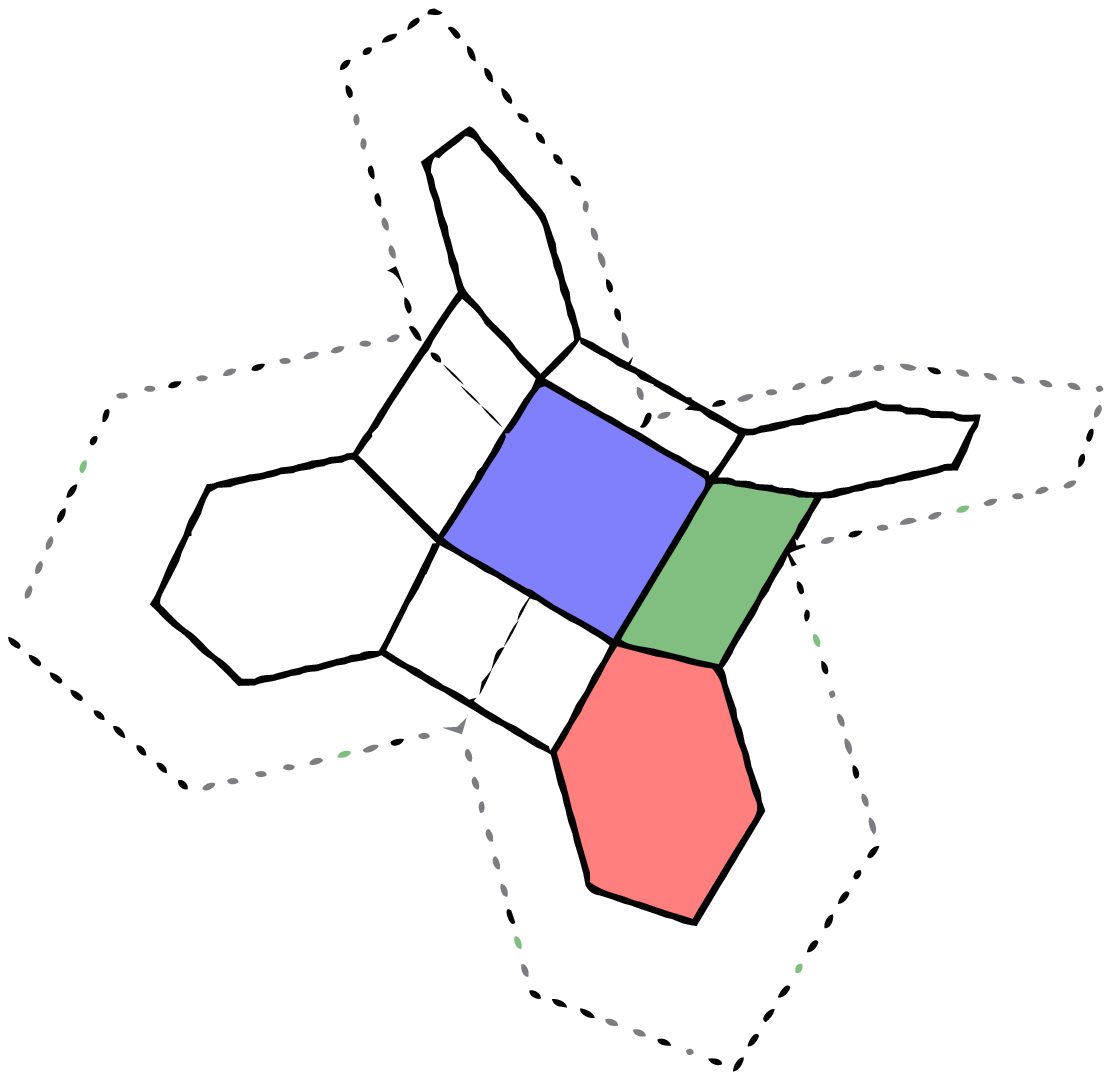}}& \raisebox{-\totalheight}{ \includegraphics*[clip, scale = .25]{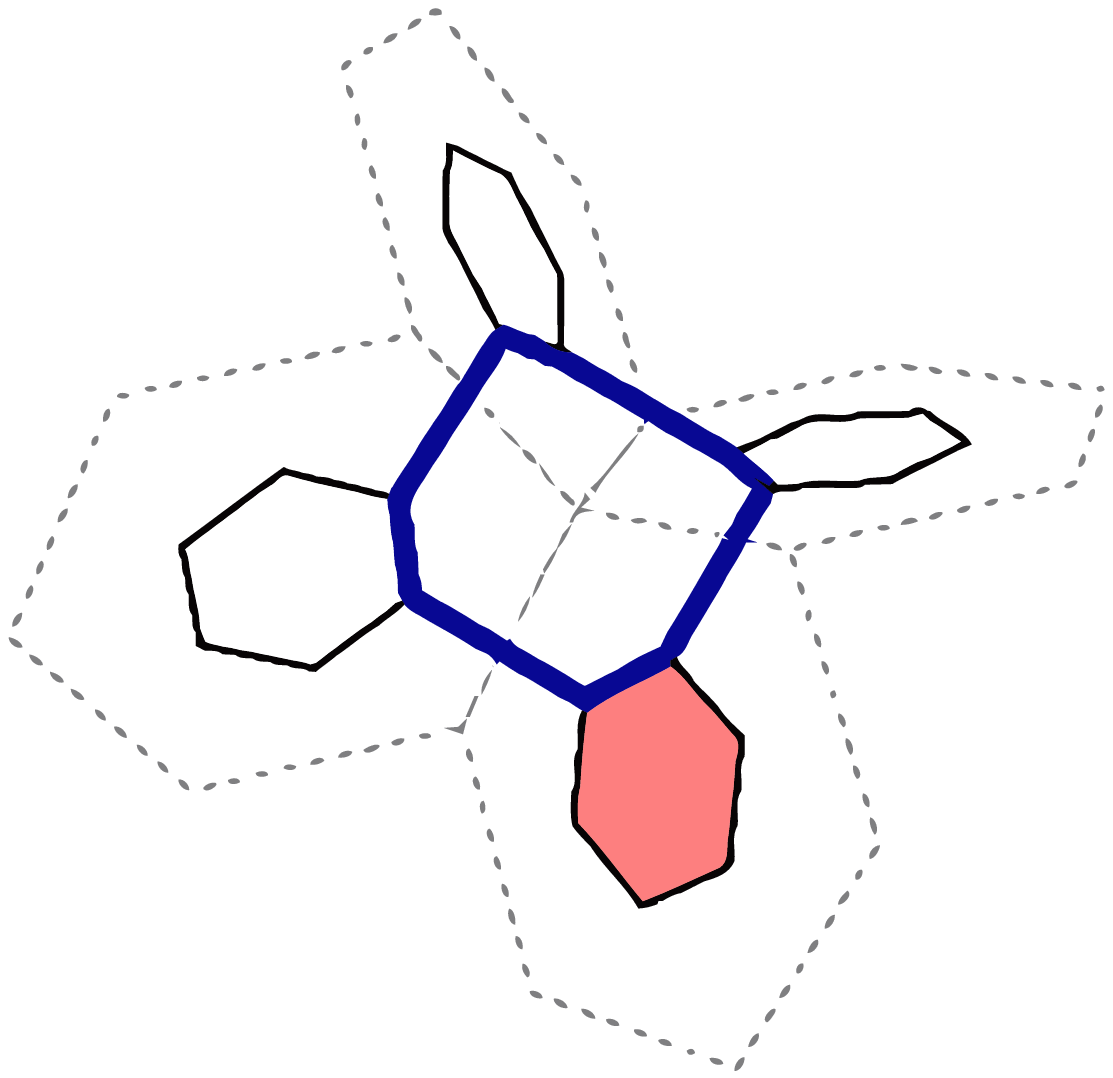}}& \raisebox{-\totalheight}{ \includegraphics*[clip, scale = .25]{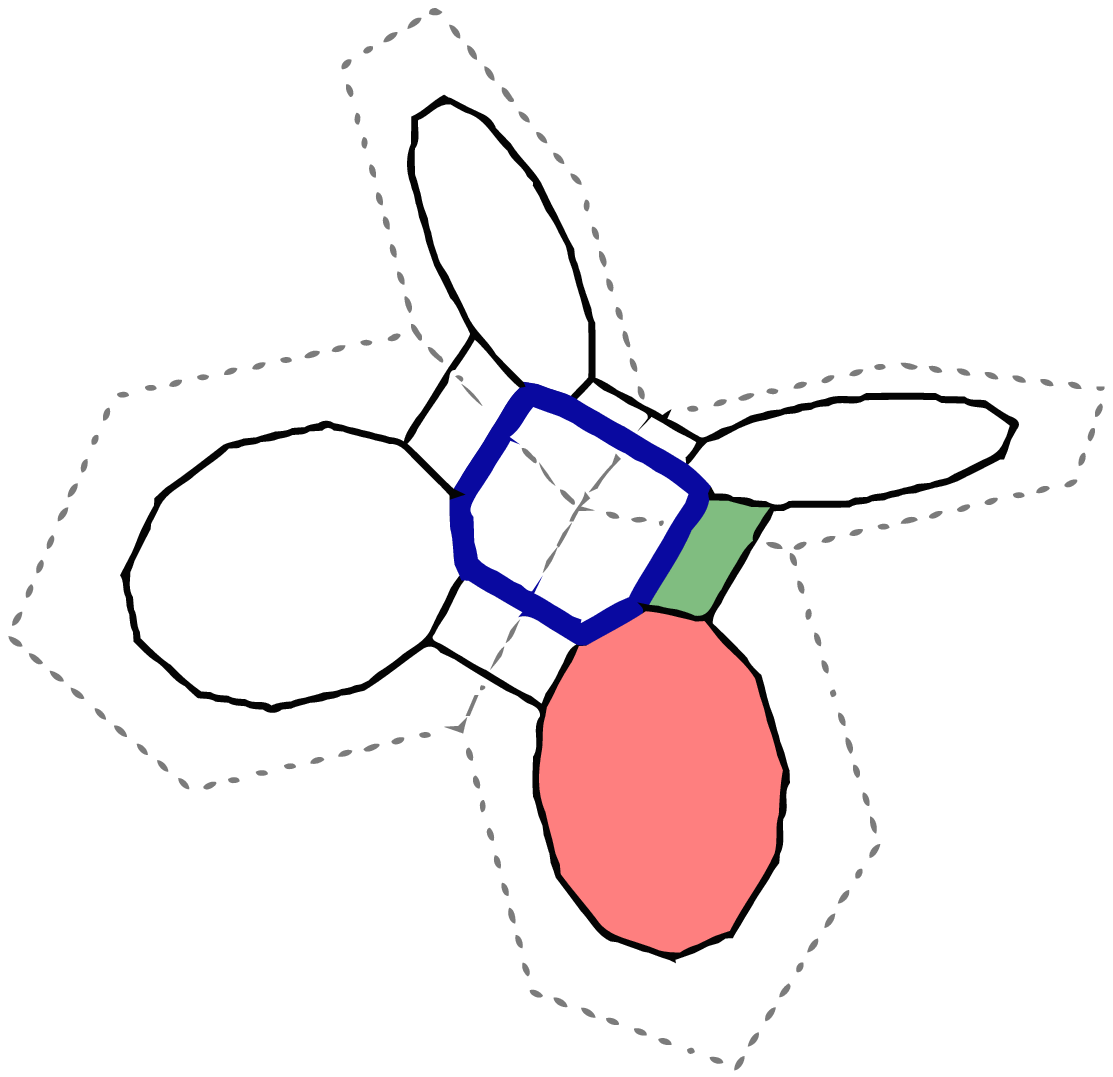}}&\\
$P^{02}$& $P^{12}$ & $P^{012}$&
\end{tabular}\caption{The Wythoffians derived from $\{6,4|4\}$.}\label{6,4l4}\end{center}\end{figure}

The final Petrie-Coxeter polyhedron is $\{6,6|3\}$ with symmetry group $G(\{6,6|3\})=\langle r_0,r_1,r_2\rangle$, where $r_{0}:=(t_0t_1)^2t_2(t_0t_1)^2$, $r_{1}:=t_1t_3$, and $r_{2}:=t_2$, and $t_0,\ldots,t_3$ are as before (see \cite[p. 224]{SchMc}). In particular, $r_0$ and $r_2$ are plane reflections and $r_1$ is a half-turn. Again the initial vertices are points of the base face of $\{6,6|3\}$ that lie in the fundamental region. As before we place this restriction on the initial vertex choices to make the geometry of the Wythoffian similar to the geometry of $\{6,6|3\}$. Note that $\{6,6|3\}$ is geometrically self-dual, and so the collections of Wythoffians $P^2$ and $P^{12}$ are just the same as those of $P^0$ and $P^{01}$, respectively. For pictures of the Wythoffians see Figure \ref{6,6l3}.

The first Wythoffian, $P^0$, is the regular apeirohedron $\{6,6|3\}$ itself. It has regular, convex hexagons for faces. Six such hexagons meet at each vertex yielding a regular, antiprismatic hexagon for the vertex figure. 

In $P^1$ the faces of type $F_2^{\{0,1\}}$ are regular, convex hexagons while the faces of type $F_2^{\{1,2\}}$ are regular, antiprismatic hexagons. Alternating about each vertex are two skew hexagons and two convex hexagons so the vertex symbol is $(6_c.6_s.6_c.6_s)$. The vertex figure is then a convex rectangle. All faces are regular so this Wythoffian is uniform. 

In $P^2$ the resulting figure is again the regular apeirohedron $\{6,6|3\}$, thanks to the self-duality of $\{6,6|3\}$.
In the apeirohedron $P^{01}$ the faces of type $F_2^{\{0,1\}}$ are convex dodecagons (truncated hexagons) and the faces of type $F_2^{\{1,2\}}$ are regular, antiprismatic hexagons. There are two dodecagons and one hexagon meeting at each vertex, yielding a vertex symbol $(6_s.12_c^2)$ and an isosceles triangle as the vertex figure. For a specific choice of initial vertex the dodecagons are regular and the Wythoffian is uniform. 

In the apeirohedron $P^{02}$ the faces of type $F_2^{\{0,1\}}$ are regular, convex hexagons; the faces of type $F_2^{\{1,2\}}$ are regular convex hexagons; and the faces of type $F_2^{\{0,2\}}$ are convex rectangles. At each vertex there is a hexagon of the first type, a rectangle, a hexagon of the second type, and another rectangle, giving a vertex symbol of  $(6_c.4_c.6_c.4_c)$.  The resulting vertex figure is a skew quadrilateral. If a certain initial vertex is chosen the faces of type $F_2^{\{0,2\}}$ are squares and the Wythoffian is uniform. 

In the apeirohedron $P^{12}$ the faces of type $F_2^{\{0,1\}}$ are regular, convex hexagons. The faces of type $F_2^{\{1,2\}}$ are skew dodecagons which appear as the truncations of regular, antiprismatic hexagons. There are two dodecagons and one hexagon at each vertex yielding an isosceles triangle as the vertex figure corresponding to the vertex symbol $(6_c.12_s^2)$. The dodecagons are not regular so the Wythoffian is not uniform. 

Finally consider $P^{012}$. In this apeirohedron the faces of type $F_2^{\{0,1\}}$ are convex dodecagons (truncated hexagons), the faces of type $F_2^{\{1,2\}}$ are skew dodecagons (truncated, anstiprismatic hexagons), and the faces of type $F_2^{\{0,2\}}$ are convex rectangles. The vertex symbol is $(4_c.12_c.12_s)$ yielding a triangular vertex figure. As before, the skew dodecagons are not regular so the Wythoffian is not uniform.

\begin{figure}[h]
\begin{center}
\begin{tabular}{cccc}
\raisebox{-\totalheight}{ \includegraphics*[clip, scale = .25]{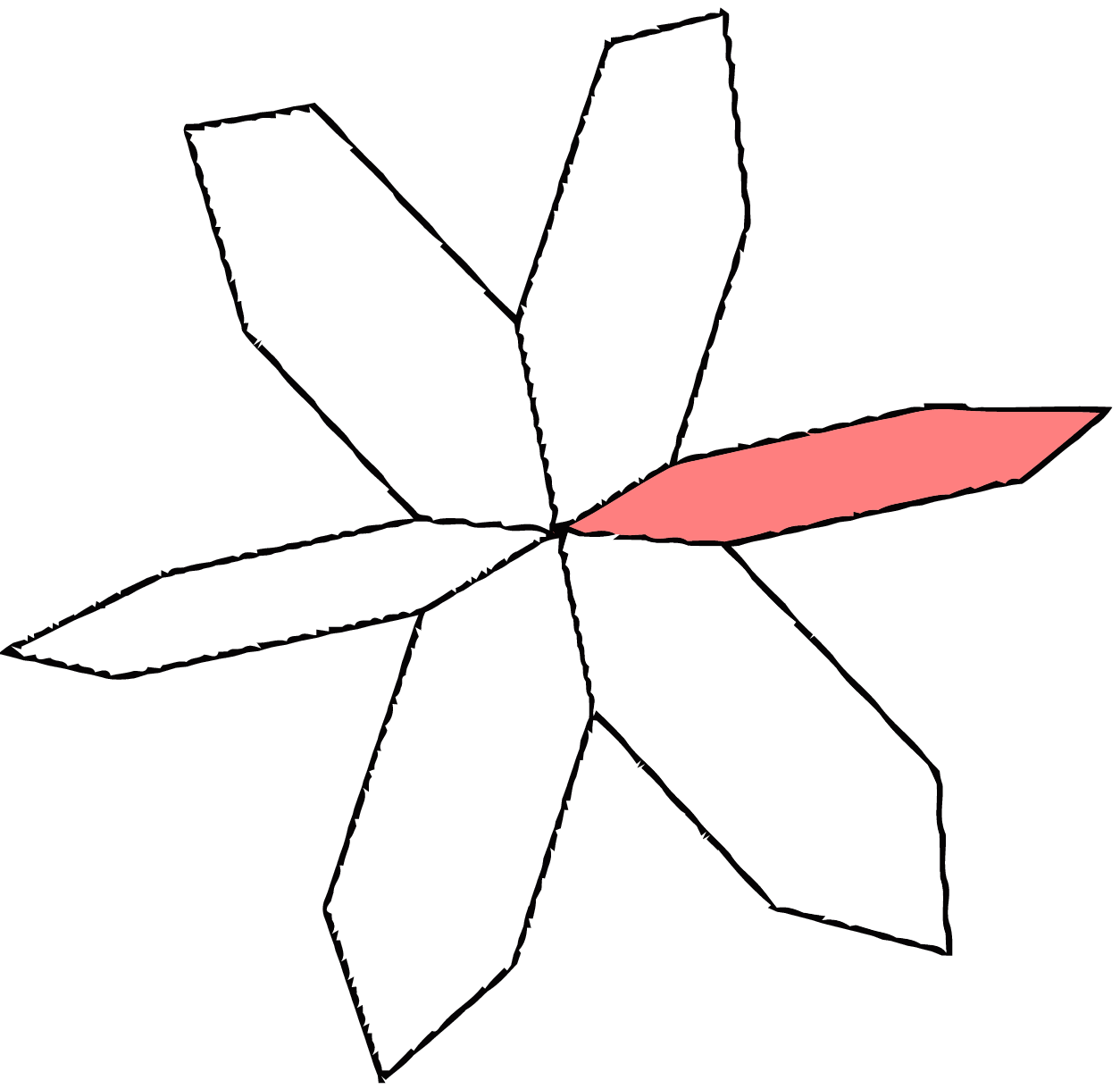}}& \raisebox{-\totalheight}{ \includegraphics*[clip, scale = .25]{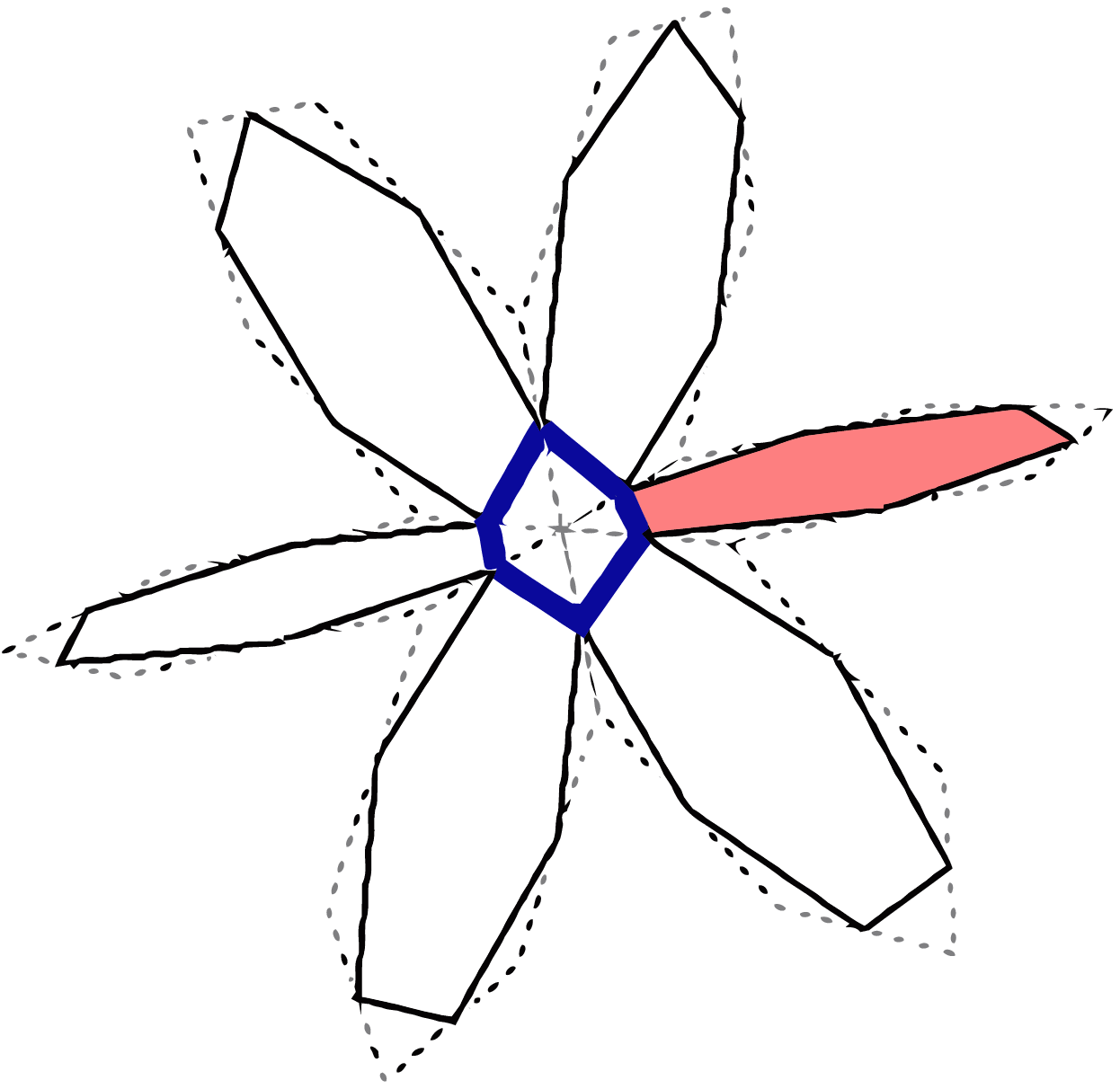}}& \raisebox{-\totalheight}{ \includegraphics*[clip, scale = .25]{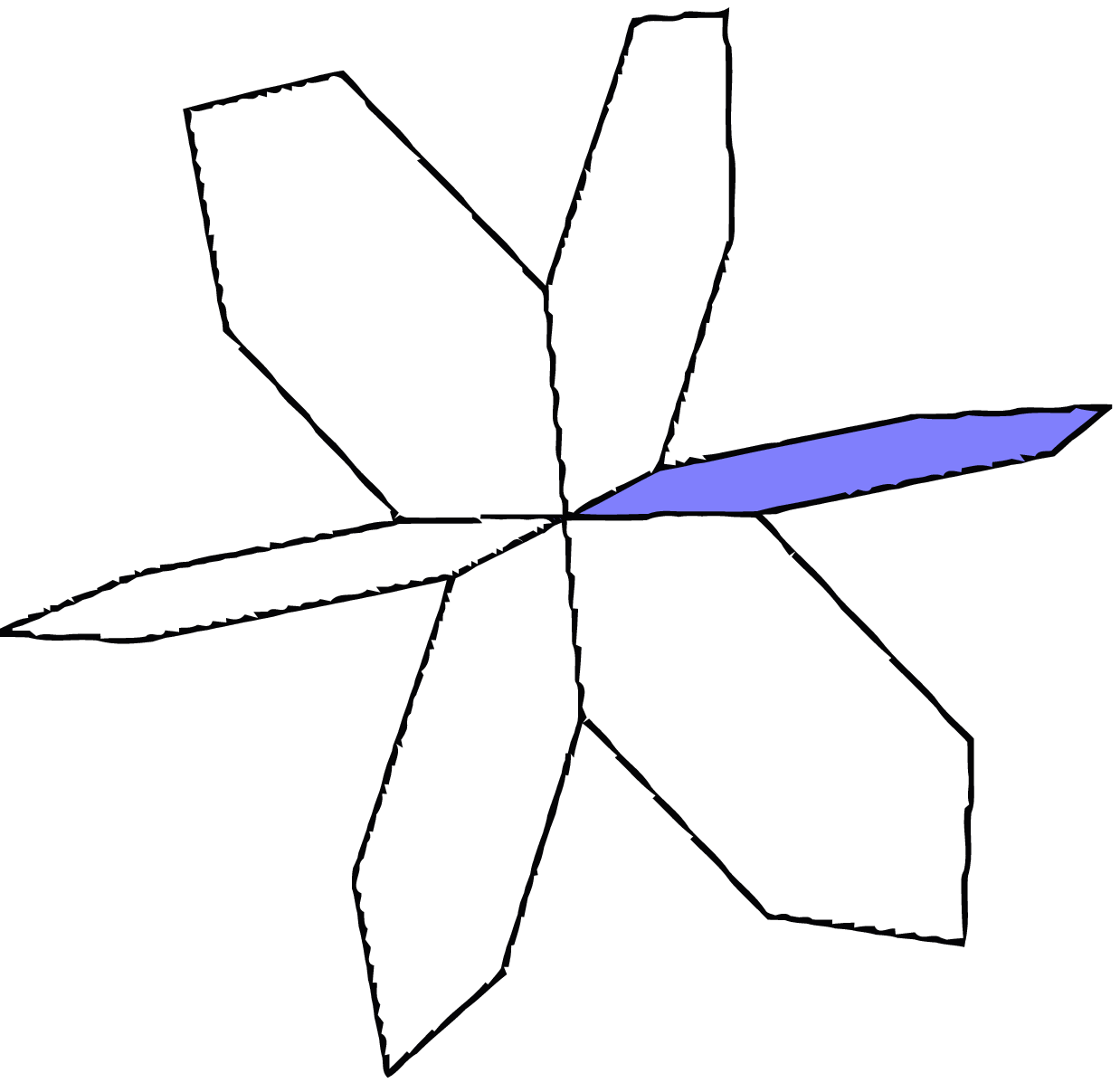}}&\raisebox{-\totalheight}{ \includegraphics*[clip, scale = .25]{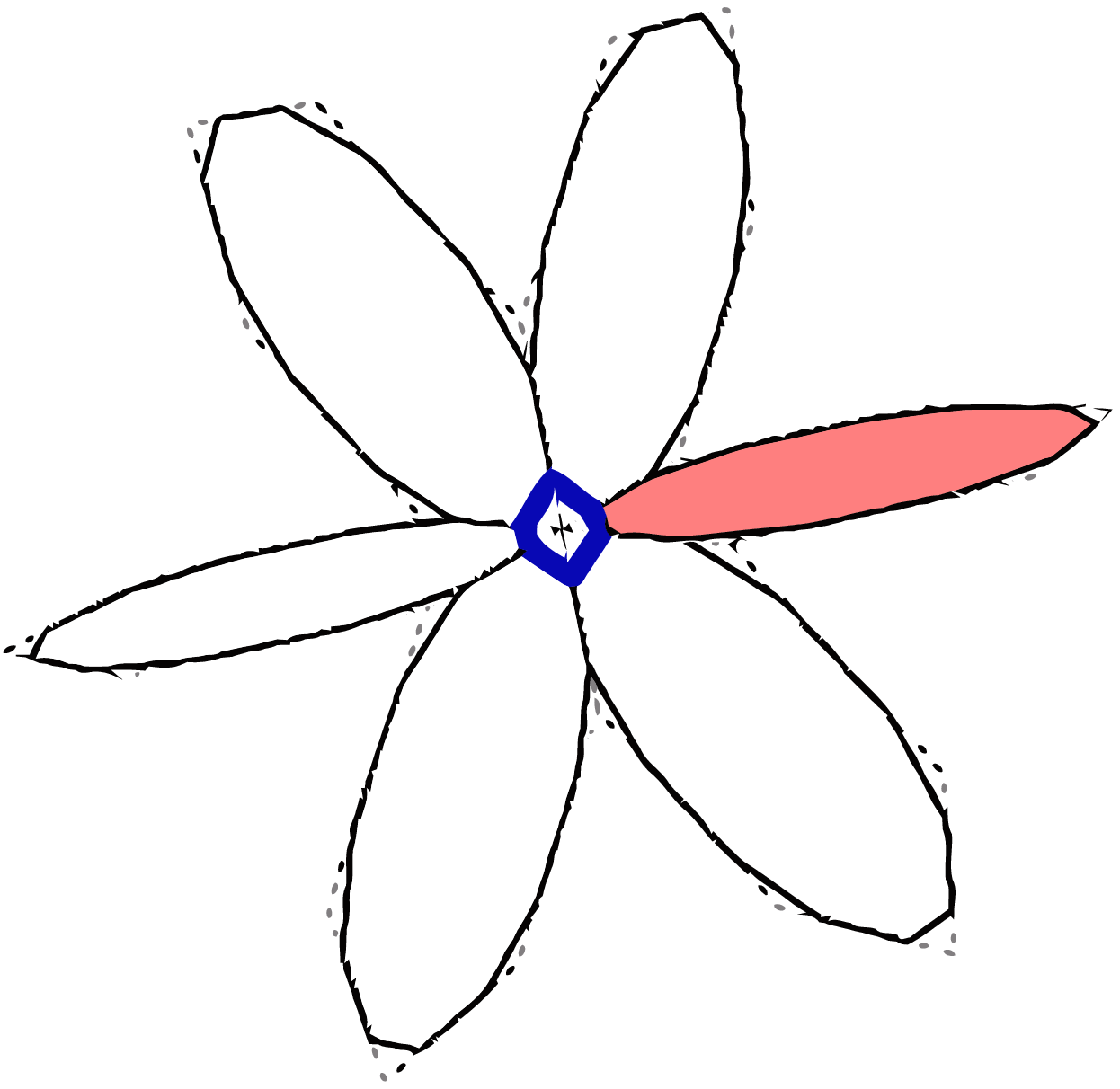}} \\
$P^{0}$ & $P^{1}$ & $P^{2}$&$P^{01}$\\
&&&\\
\raisebox{-\totalheight}{ \includegraphics*[clip, scale = .25]{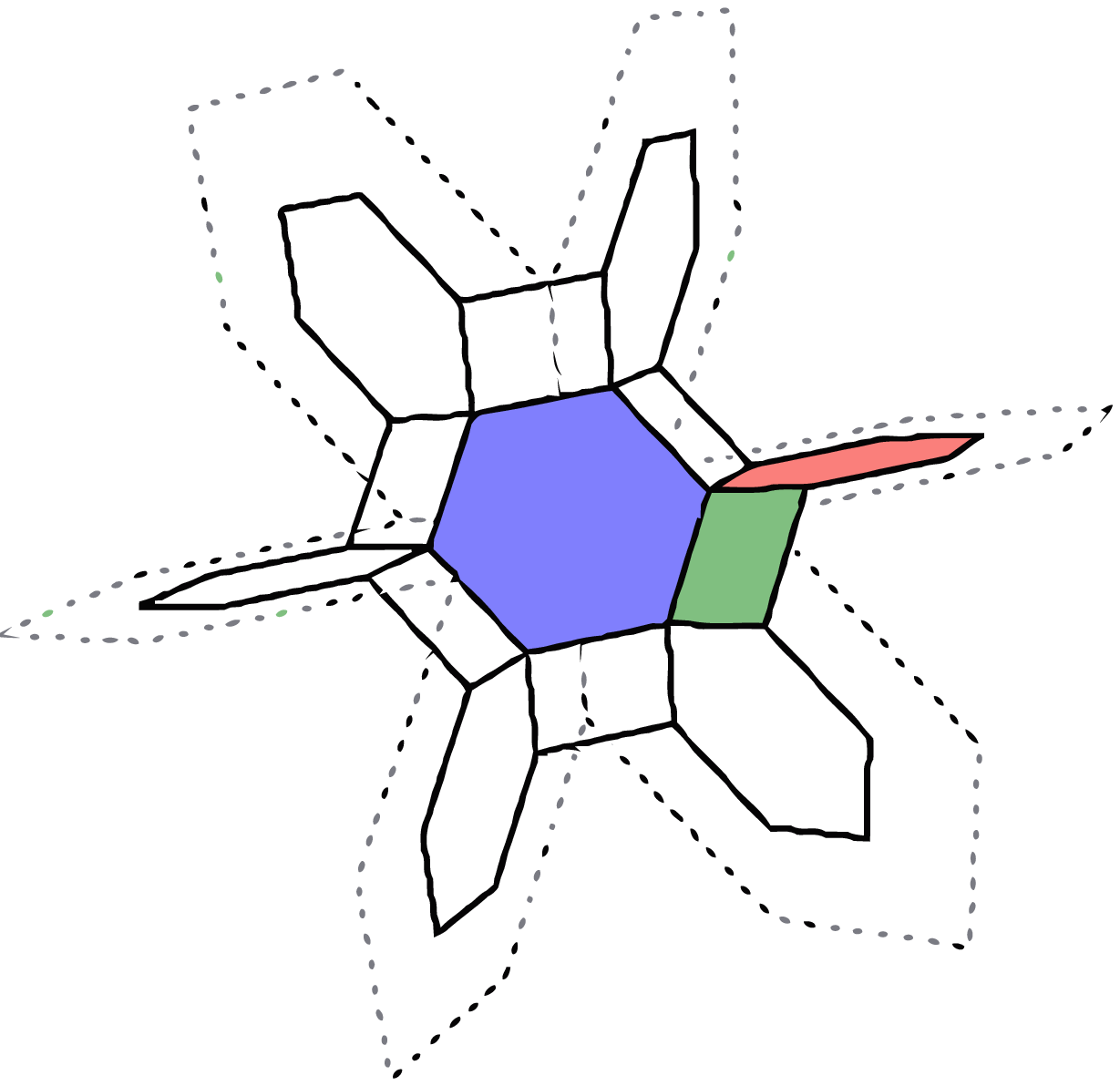}}& \raisebox{-\totalheight}{ \includegraphics*[clip, scale = .25]{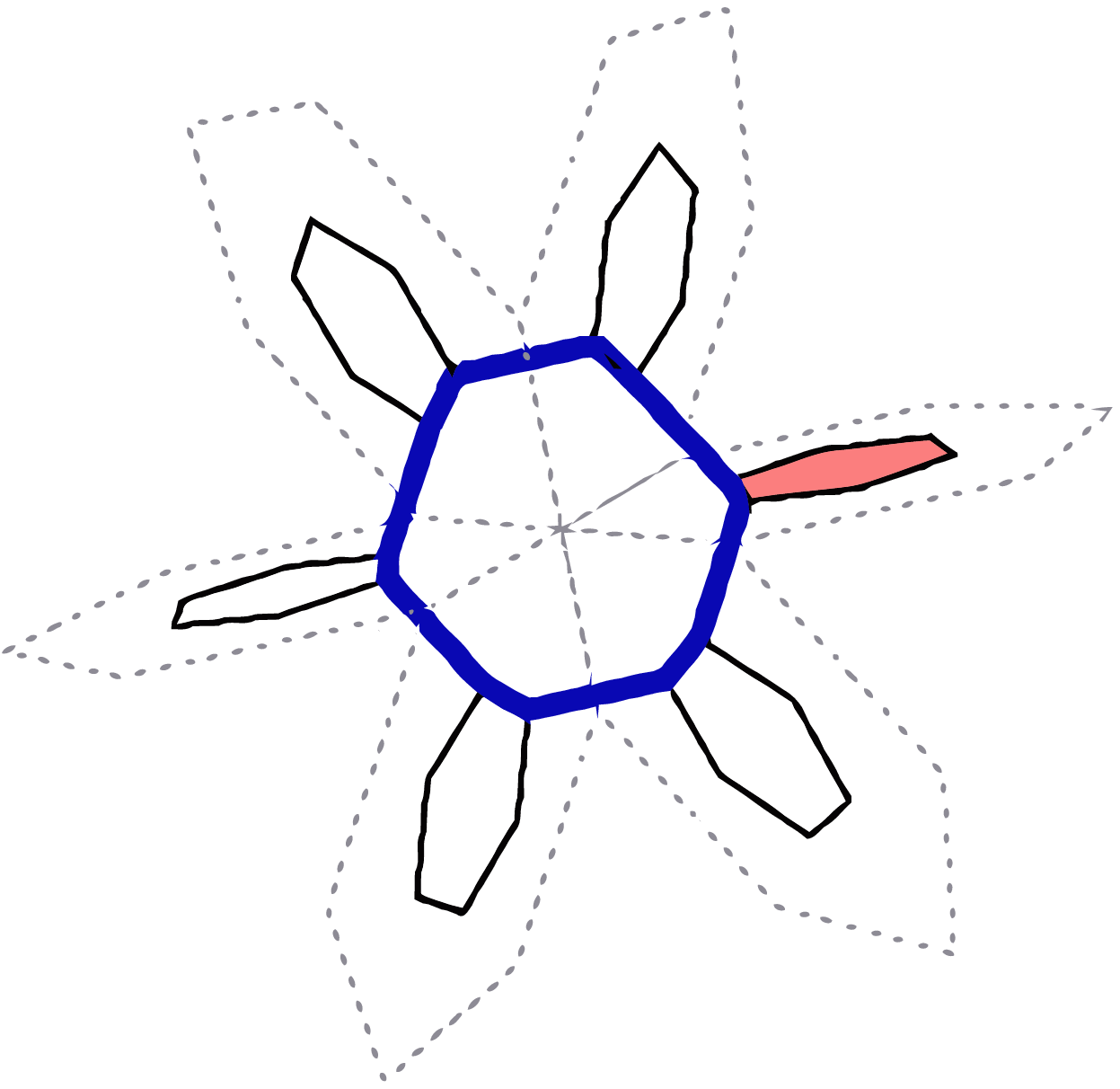}}& \raisebox{-\totalheight}{ \includegraphics*[clip, scale = .25]{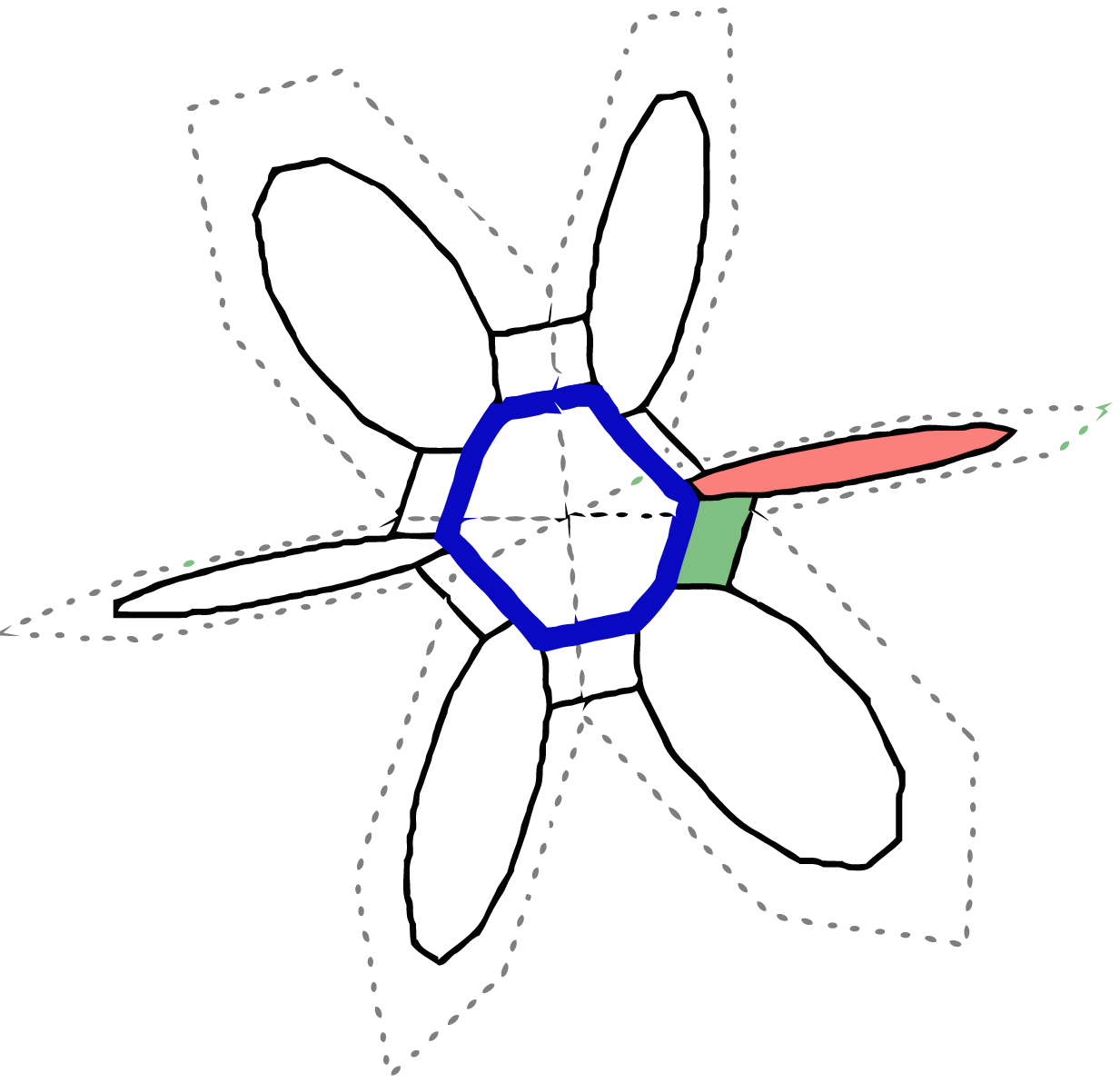}}&\\
$P^{02}$& $P^{12}$ & $P^{012}$&
\end{tabular}\caption{The Wythoffians derived from $\{6,6|3\}$.}\label{6,6l3}\end{center}\end{figure}

\noindent
{\bf Acknowledgment.} We are grateful to the anonymous referees for their careful reading of our original manuscript and their helpful suggestions that have improved our paper.

\bibliography{WythOne}
\bibliographystyle{plain}
\nocite{*}

\end{document}